\documentclass[11pt,leqno]{amsart}

\usepackage{amsfonts}
\usepackage{tikz}
\def\div{{\rm div}\,}

\usepackage[utf8]{inputenc}
\usepackage [english]{babel}
\usepackage [utf8]{inputenc}
 \usepackage [T1] {fontenc}
\usepackage{wasysym}

\usepackage{amssymb}
\usepackage{enumitem}
\usepackage{mathrsfs}
\usepackage[all]{xy}
\usepackage{color}
\usepackage{soul}
\usepackage[dvips]{epsfig}
\usepackage{hyperref}
\usepackage{dsfont}
\def\eps{{\varepsilon}}
\usepackage[margin=3cm,top=2.5cm,bottom=2cm]{geometry}

\usepackage{amsmath, amssymb, amsthm, mathrsfs}
\usepackage{mathtools}
\mathtoolsset{showonlyrefs}
\usepackage{enumitem,microtype,todonotes} 
\usetikzlibrary{arrows,calc,shapes,decorations.pathreplacing}
\usetikzlibrary{decorations.pathmorphing}
\usetikzlibrary{shapes,arrows}
\usetikzlibrary{patterns}

\usetikzlibrary{arrows,shapes,positioning}
\usetikzlibrary{decorations.markings}
\tikzstyle arrowstyle=[scale=2]
\tikzstyle directed=[postaction={decorate,decoration={markings,
    mark=at position .65 with {\arrow[arrowstyle]{stealth}}}}]
\tikzstyle reverse directed=[postaction={decorate,decoration={markings,
    mark=at position .65 with {\arrowreversed[arrowstyle]{stealth};}}}]

\usepackage{lmodern}

\newcommand{\R}{\mathbf R}

\theoremstyle{plain}
\newtheorem{theo}{Theorem}[section]
\newtheorem{prop}[theo]{Proposition}
\newtheorem{lem}[theo]{Lemma}

\theoremstyle{remark}
\newtheorem{rem}[theo]{Remark}
\theoremstyle{definition}

\numberwithin{equation}{section}
\usepackage{float} 

\def\le{\leqslant}
\def\ge{\geqslant}
\def\leq{\leqslant}
\def\geq{\geqslant}

\def\eps{{\varepsilon}}

\allowdisplaybreaks
\begin{document}

\title[A spectral approach to the narrow escape problem ]{A spectral approach to the narrow escape problem in the  disk}
\date{\today}

\author[T. Lelièvre]{Tony Lelièvre}
\author[M. Rachid]{Mohamad Rachid}   
\author[G. Stoltz]{Gabriel Stoltz}

\address[T. Lelièvre]{\'Ecole des Ponts ParisTech, CERMICS and Inria, France.}
\email{tony.lelievre@enpc.fr}

\address[M. Rachid]{\'Ecole des Ponts ParisTech, CERMICS and Inria, France.}
\email{Mohamad.Rachid@enpc.fr}

\address[G. Stoltz]{\'Ecole des Ponts ParisTech, CERMICS and Inria, France.}
\email{gabriel.stoltz@enpc.fr}



\begin{abstract}
We study the narrow escape problem in the  disk, which consists in
identifying   the first exit time and first exit point distribution of
a Brownian particle from the ball in dimension~2, with reflecting
boundary conditions except on small disjoint windows through which it
can escape. This problem is motivated by practical questions arising
in various scientific fields (in particular cellular biology and
molecular dynamics). We apply the quasi-stationary distribution
approach to metastability, which requires to study the eigenvalue
problem for the Laplacian operator with  Dirichlet boundary conditions
on the small absorbing part of the boundary, and  Neumann boundary
conditions on the remaining reflecting part. We obtain rigorous asymptotic estimates of the first eigenvalue and of the normal derivative of the associated eigenfunction in the limit of infinitely small exit regions, which yield asymptotic estimates of the first exit time and first exit point distribution starting from the quasi-stationary distribution within the disk.
\end{abstract}

\maketitle

\section{Introduction and Motivation}\label{sec:intro}
 The \textit{narrow escape problem} is a question arising in various models used in biophysics and cell biology. The mathematical formulation of the problem is as follows: a Brownian particle (representing e.g. an ion) is confined to a bounded domain  (representing e.g. a biological cell), with a reflecting boundary  except for  small disjoint windows (representing e.g. membrane channels) through which it  can escape. 
  The narrow escape problem then consists in precisely describing the exit event, namely the law of the pair of random variables: first exit time and  first exit point. One is for example interested in the average time needed for an ion to find an  ion channel located in the cell membrane, and also in information on which exit channel will be the most likely. Because of the practical importance of this question, there are numerous contributions from the physics community, which concentrate on the asymptotic behavior of the exit time, and mostly on the mean first exit time (see however~\cite{grebenkov2019full} for studies on the law of the exit time). We refer for example to~\cite{bib1,bib2,grigoriev2002kinetics,schuss2007narrow} for a few representative contributions from the physics literature, and for typical practical applications in biology. Our objective in this work is to prove asymptotic results on both the first exit time distribution and the first exit point distribution, using the quasi-stationary distribution approach to metastability~\cite{di2016jump}.

\subsection{Mathematical framework}
 Let us introduce the mathematical model associated with the problem. The motion of a particle in a bounded domain $\Omega\subset\mathbb{R}^{d}$
is described by the diffusion process:
\begin{align}\label{lang}
    dX_t=\sqrt{2}\hspace{0.05cm}dB_t,
\end{align}
where $X_t\in\Omega$ and $(B_t)_{t\geq0}$ is a $d$-dimensional
Brownian motion. Here  $\Omega\subset\mathbb{R}^{d}$ is a bounded
domain   whose boundary is
$\partial\Omega=\overline{\Gamma^{\varepsilon}_{\mathbf{D}}}\cup\overline{\Gamma^{\varepsilon}_{\mathbf{N}}}$
with
$\Gamma^{\varepsilon}_{\mathbf{D}}\cap\Gamma^{\varepsilon}_{\mathbf{N}}=\emptyset$. The
set $\Gamma^{\varepsilon}_{\mathbf{D}}={\cup}_{k=1}^{N}
\Gamma^{\varepsilon}_{\mathbf{D}_k}$ is composed of $N$ small disjoint
open connected absorbing windows  
 $(\Gamma^{\varepsilon}_{\mathbf{D} _k})_{k=1,\ldots,N}$,  and
 $\Gamma^{\varepsilon}_{\mathbf{N}}=\partial \Omega \setminus \overline{\Gamma^{\varepsilon}_{\mathbf{D}}}$ is the reflecting part (see Figure~\ref{Narrow} for a schematic representation in  dimension $d=2$). More precisely, let us introduce $N$ points $x^{(1)},\dots, x^{(N)}$  on the boundary of $\Omega$  and let $\rho_0 \in (0,2]$ be defined as
\begin{align}\label{rho0}
\rho_0= \min_{1\le i < j \le N} \left|x^{(i)}-x^{(j)}\right|.
\end{align}
We then assume that\footnote{The main results (Theorems~\ref{asymlambda N hole} and~\ref{exit point}) will actually be obtained for $\Omega$ the unit disk (thus in dimension $d=2$), and a slightly different geometry for the escape regions that coincides with the description given here in the limit $\varepsilon \to 0$. This is made precise below.} there exists $\varepsilon_{0} > 0$ (which we will need to further reduce later on) such that, for any~$\varepsilon\in (0,\varepsilon_0)$ and~$k\in \{1,\dots,N\}$, there exists $K^{(k)}_{\varepsilon}>0$ satisfying
\begin{align}
\label{assumption on Gamma_D}
    \Gamma^{\varepsilon}_{\mathbf{D}_k}=\partial\Omega\cap \mathbf{B}\left(x^{(k)},\mathrm{e}^{-1/K^{(k)}_{\varepsilon}}\right), \qquad \mathrm{e}^{-1/K^{(k)}_{\varepsilon}} \leq \frac{\rho_0}{2},
\end{align} 
where $\lim_{\varepsilon \to 0} K^{(k)}_{\varepsilon}= 0$, and where $\mathbf{B}(x,r)$ denotes the disk with center $x$ and radius $r$. The second condition in~\eqref{assumption on Gamma_D} ensures that the holes are disjoint. To give a concrete example, $N$ identical holes with sizes of order~$\varepsilon$ are obtained by considering $K^{(k)}_{\varepsilon}= 1/|\log \varepsilon|$.
    
For  $(X_t)_{t\geq0}$ satisfying~\eqref{lang} with an initial condition $X_0 \in \Omega$, we define the first exit time of~$(X_t)_{t\geq0}$ as follows:
\begin{align*}
    \tau:=\text{inf}\{t\geq0,\hspace{0.1cm} X_t\notin \overline{\Omega}\}.
\end{align*}
The exit event from $\overline{\Omega}$ is fully characterized by the couple
of random variables $(\tau,X_{\tau})$.
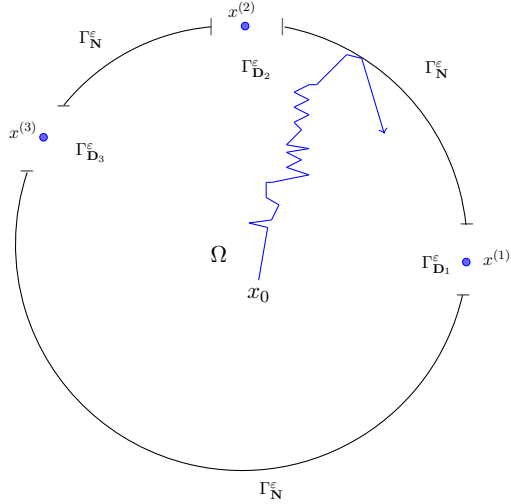
\begin{figure}
\begin{center}
\vspace{-0.4cm}
\begin{tikzpicture}[scale=1]
\draw (2.5,0.4) arc (15:79:3);
 \draw (-0.8,2.57) arc (95:141:2.84);
  \draw (2.6,-0.05) node[scale=0.7] {$-$};
  \draw (2.55,-1) node[scale=0.7] {$-$};
  \draw (-0.8,2.57) node[scale=0.6] {$|$};
   \draw (0.15,2.57) node[scale=0.6] {$|$};
    \draw (-2.76,1.52) node[scale=0.7] {$-$};
    \draw (-3.25,0.65) node[scale=0.7] {$-$};
   \draw[black] (2.5,0.4) arc (15:6:3);
     \draw[black] (2.54,-1) arc (-13:-198.5:3.01);
       \draw[fill=blue!60,draw=blue] (-0.35,2.58) circle (0.5mm);
    \draw[fill=blue!60,draw=blue] (2.59, -0.56) circle (0.5mm);
    \draw (-0.39,2.8) node[scale=0.6] {$x^{(2)}$};
     \draw[fill=blue!60,draw=blue] (-3.03,1.1) circle (0.5mm);
    \draw (-3.3,1.2) node[scale=0.6] {$x^{(3)}$};
         \draw (3,-0.5) node[scale=0.6] {$x^{(1)}$};
   
          \draw (-0.2,2) node[scale=0.6] {$\Gamma^{\varepsilon}_{\mathbf{D}_{2}}$};
          \draw (2.2,-0.6) node[scale=0.6] {$\Gamma^{\varepsilon}_{\mathbf{D}_{1}}$};
          \draw (-2.4,0.9) node[scale=0.6] {$\Gamma^{\varepsilon}_{\mathbf{D}_{3}}$};
           \draw (-2.4,2.4) node[scale=0.6] {$\Gamma^{\varepsilon}_{\mathbf{N}}$};
             \draw (2.2,2) node[scale=0.6] {$\Gamma^{\varepsilon}_{\mathbf{N}}$};
              \draw (0,-3.6) node[scale=0.6] {$\Gamma^{\varepsilon}_{\mathbf{N}}$};
\draw (-0.7, -0.45) node[scale=0.8] {$ \Omega$};
\draw[blue] (0, 0.5) -- (0.5, 0.6);
\draw[blue] (0.5, 0.6) -- (0.2,0.7);
\draw[blue] (0.2, 0.7) -- (0.4,0.8);
\draw[blue] (0.4, 0.8) -- (0.2,0.9);
\draw[blue] (0.2, 0.9) -- (0.5,0.95);
\draw[blue] (0.5, 0.95) -- (0.2,1);
\draw[blue] (0.2,1) -- (0.4,1.2);
\draw[blue] (0.4,1.2) -- (0.3,1.3);
\draw[blue] (0.3,1.3) -- (0.5,1.4);
\draw[blue] (0.5,1.4) -- (0.3,1.5);
\draw[blue] (0.3,1.5) -- (0.5,1.6);
\draw[blue] (0.5,1.6) -- (0.3,1.7);
\draw[blue] (0.3,1.7) -- (0.5,1.8);
\draw[blue] (0.5,1.8) -- (0.6,1.8);
\draw[blue](0.6,1.8) -- (0.7,1.9);
\draw[blue] (0.7,1.9) -- (0.8,2);
\draw[blue] (0.8,2) -- (0.9,2.1);
\draw[blue](0.9,2.1) -- (1,2.2);
\draw[blue] (1,2.2) -- (1.2,2.15);
\draw[blue] [->] (1.2,2.15) -- (1.5,1.15);
\draw[blue] (-0.17, -0.8) -- (-0.05,-0.1);
\draw[blue] (-0.05,-0.1) -- (-0.3,-0.04);
\draw[blue] (-0.3,-0.04) -- (0,0);
\draw[blue] (0,0)-- (0.1,0.2);
\draw[blue] (0.1,0.2)-- (-0.07,0.3);
\draw[blue] (-0.07,0.3)-- (-0.07,0.5);
\draw[blue] (-0.07,0.5)-- (-0,0.5);
\draw (-0.17, -1) node[scale=0.8] {$x_0$};
\end{tikzpicture}
\caption{\label{Narrow} Example of a domain  $\Omega$, ${\Gamma}^{\varepsilon}_{\mathbf{N}}$ and ${\Gamma}^{\varepsilon}_{\mathbf{D}}$ (for   $d=2$ and $N=4$).}
\end{center}
\end{figure}

Since the escape regions are very narrow, the process $(X_t)_{t\geq0}$ is metastable: it stays trapped for a very long time in the domain $\Omega$ before leaving it through one of the small windows. As a consequence, we use the quasi-stationary distribution approach to metastability: we study the exit event for the stochastic process $(X_t)_{t \ge 0}$ {\em starting from the quasi-stationary distribution} within $\Omega$.     We refer to~\cite{Letolupe,di2016jump,LePeutrecNectoux,LelievrePeutrecNectoux} for a justification of this approach (see also Remark~\ref{rem:QSD} below). The quasi-stationary distribution can be defined as the  distribution $\nu_{0}^{\varepsilon}$ such that, for any initial distribution~$\mu$ with support in $\overline{\Omega}$,
    \begin{align}
      \lim_{t \to \infty}\mathbb{P}_{\mu}[X_{t}\in\cdot \, |\, t<\tau]= \nu_{0}^{\varepsilon},
\end{align}
where the subscript~$\mu$ in $\mathbb{P}_{\mu}$ indicates that $X_0 \sim \mu$. In words, the quasi-distribution is the longtime limit of the distribution of configurations associated with trajectories which have not exited~$\Omega$. The distribution $\nu_{0}^{\varepsilon}$ is thus uniquely defined, with support in $\overline{\Omega}$.

Our aim is to determine asymptotic estimates of the law of the first
exit time  $\tau$ and  of  the   first exit point $X_{\tau}$ as
$\varepsilon\rightarrow 0$, with $X_0 \sim \nu_{0}^{\varepsilon}$. A
crucial result, which also explains the interest of considering the
exit event starting from $\nu_{0}^{\varepsilon}$, is the following: if
$X_0 \sim \nu_{0}^{\varepsilon}$, then $\tau$ is exponentially
distributed and independent of $X_\tau$
(see~\cite{Letolupe,collet2013quasi}). Let us mention that this result
is in particular fundamental to justify the use of the kinetic Monte
Carlo algorithm (see~\cite{Monte,di2016jump}) to efficiently sample
the exit event, and thus simulate the process over very long times. In
order to fully characterize the law of $(\tau,X_\tau)$, it therefore remains to identify $\mathbb{E}_{\nu_0^\varepsilon}[\tau]$ and the distribution of $X_\tau$. Here and in the following, the subscript in $\mathbb{E}_{\nu_{0}^{\varepsilon}}[\cdot]$ indicates that the process $(X_t)_{t\geq0}$ is such that $X_0  \sim \nu_{0}^{\varepsilon}$. 

Let us now explain why the above question actually amounts to studying the behavior of the first eigenvalue and first eigenvector of the infinitesimal generator of the stochastic process. The (opposite\footnote{The sign convention is justified by the fact that we prefer to work with positive operators.} of the) infinitesimal generator of the dynamics~\eqref{lang} is $-\Delta$.
In particular, the density function $f(t,x)$ of the process $X_t$ reflected on~$\Gamma^{\varepsilon}_{\mathbf{N}}$ and absorbed on~$\Gamma^{\varepsilon}_{\mathbf{D}}$ satisfies the Fokker--Planck equation: 
$$
\left\{
\begin{aligned}
    \partial_{t}f&=\Delta f \text{ in } \Omega, \\
    \partial_n f &= 0 \text{ on } \Gamma^{\varepsilon}_{\mathbf{N}}\text{ and }
    f = 0 \text{ on } \Gamma^{\varepsilon}_{\mathbf{D}}, \\    
    f(0,\cdot)&=f_0 \text{ in } \Omega,
\end{aligned}
\right.
$$
where $f_0$ is the density of $X_0$, and $\partial_{n}= \vec{n}\cdot\nabla$ denotes the normal derivative with $\Vec{n}$  the unit vector outward normal to $\Omega$.
Let us introduce the principal eigenvalue $\lambda_{0}^{\varepsilon}$ and the associated eigenfunction $u_{0}^{\varepsilon}$ related to this Fokker-Planck equation:
\begin{equation}\label{mixed on disk}
 \left\{
 \begin{aligned}
    -\Delta u_{0}^{\varepsilon}&=\lambda_{0}^{\varepsilon} \, u_{0}^{\varepsilon}\,\,\text{in} \,\,\Omega,\\
    \partial_{n}u_{0}^{\varepsilon}&=0\,\,\,\text{on}\,\,\Gamma^{\varepsilon}_{\mathbf{N}},\\
   u_{0}^{\varepsilon}&=0\,\,\,\text{on}\,\,\Gamma^{\varepsilon}_{\mathbf{D}}.
\end{aligned}
\right.
    \end{equation}
We denote by $\mathcal{L}^{\varepsilon}$ the operator\footnote{The superscript $\varepsilon$ indicates that the boundary conditions, and thus the domain of the operator, depend on~$\varepsilon$ (see the definition~\eqref{domain} below).} associated with~\eqref{mixed on disk}, namely minus the Laplacian with mixed boundary conditions on~$\partial \Omega$. In the following, we normalize $u_{0}^{\varepsilon}$ as
\begin{equation}
\label{eq:normalization_u_0_eps}
\int_{\Omega}(u_{0}^{\varepsilon}(x))^{2} \, \mathrm{d}x= 1, 
\qquad 
u_{0}^{\varepsilon}<0 \textrm{ on } \Omega.
\end{equation}
The second condition can indeed be imposed as it is  standard to show that the first eigenfunction~$u^{\varepsilon}_0$ does not vanish in~$\Omega$ (see for example~\cite[Theorem 8.38]{gilbarg-trudinger-01}), so that it can be chosen to be negative in $\Omega$.

Let us explain why studying the exit problem starting from the quasi-stationary distribution reduces to studying the principal eigenvalue and eigenfunction (see~\cite{Letolupe}). The unique quasi-stationary distribution $\nu_{0}^{\varepsilon}$ of the process~\eqref{lang}
in $\Omega$ can be written as follows: 
\begin{align*}
\nu_{0}^{\varepsilon}(\mathrm{d}x)=\frac{u_{0}^{\varepsilon}(x)\, \mathrm{d}x }{\displaystyle\int_{\Omega}u_{0}^{\varepsilon}(y)\, \mathrm{d}y}.
\end{align*}
Moreover, starting from $\nu_0^\varepsilon$, the exit time~$\tau$ is exponentially distributed
 with parameter $\lambda_{0}^{\varepsilon}$: 
 \begin{align}\label{tau}
   \mathbb{E}_{\nu_{0}^{\varepsilon}}[\tau]=\frac{1}{\lambda_{0}^{\varepsilon}}.
 \end{align}
In addition, the law of $X_{\tau}$ (with support on $\partial \Omega$) is given by (provided that $\partial_{n}u_{0}^{\varepsilon}$ is well defined as an $L^1$ function on $\partial \Omega$, otherwise one should rely on a weak formulation of the normal derivative)
   \begin{align}\label{law of X_t}\frac{\partial_{n}u_{0}^{\varepsilon}(x) \, \sigma(dx)}{\displaystyle\int_{\partial\Omega}\partial_{n}u_{0}^{\varepsilon}(y)\,\sigma(\mathrm{d}y)},
     \end{align}
     where $\sigma$ denotes generically in the following 
      Lebesgue surface measures, here on $\partial \Omega$.
     In particular, the probability of exiting through the window $\Gamma^{\varepsilon}_{\mathbf{D} _k}$ is
     \begin{align}\label{X}
\mathbb{P}_{\nu_{0}^{\varepsilon}}[X_{\tau}\in\Gamma^{\varepsilon}_{\mathbf{D} _k}]=\frac{\displaystyle\int_{\Gamma^{\varepsilon}_{\mathbf{D}_k}}\partial_{n}u_{0}^{\varepsilon}(x)\,\sigma(\mathrm{d}x)}{\displaystyle\int_{\Gamma^{\varepsilon}_{\mathbf{D}}}\partial_{n}u_{0}^{\varepsilon}(y)\,\sigma(\mathrm{d}y)}.
 \end{align}
In order to characterize the law of $(\tau,X_\tau)$ in the limit $\varepsilon\rightarrow 0$, it thus suffices to study the asymptotic behavior of the first eigenvalue $\lambda_{0}^ {\varepsilon}$ and of the normal derivative of  $u_{0}^{\varepsilon}$ in this limiting regime.

In previous works~\cite{55,71}, the quasi-stationary distribution
approach to metastability has been used to study the exit event of the
overdamped Langevin dynamics from the basin of attraction of a local
minimum of a potential energy function: in this context, metastability
is related to {\em energetic barriers} that the process has to
overcome in order to leave a metastable basin. In particular, the
aforementioned works justify the use of the so-called Eyring--Kramers
formulas to parameterize the law of the exit event. In the present
work, we initiate a similar analysis for a situation where
metastability is not due to energetic barriers but to {\em entropic
  barriers}: the process does not have to get over energy levels to
leave $\Omega$, but to find the narrow exit regions on the boundary of
$\Omega$. Let us emphasize that the present work should be seen as a
first application of the quasi-stationary distribution approach to
metastability originating from such purely entropic barriers. As
discussed below, it opens the route to many generalizations (to other
dynamics, and other geometries, see e.g. appendix~\ref{app:3d} for a discussion about generalizations to the three dimensional unit ball) and to numerical algorithms  (to efficiently simulate the exit event relying on the identified asymptotic behaviors).

\begin{rem}\label{rem:QSD}
  As explained above, the quasi-stationary distribution approach to
  metastability is particularly useful to make a rigorous link between
  a metastable Markov  process in a continuous state space and a jump Markov process with values in a discrete state space, namely a set of labels of the metastable states~\cite{di2016jump}. It is also useful to analyze efficient algorithms to sample metastable Markov processes~\cite{lelievre-15}. The only situation where such a link can be rigorously made is when the quasi-stationary distribution is reached (up to machine precision) before the exit occurs, since the jump Markov process requires the exit time to be exponentially distributed and independent of the exit point, and the quasi-stationary distribution is the only initial condition which leads to an exit time independent of the exit point. 
  The fact that the stochastic process leaves the domain after the
  quasi-stationary distribution was reached depends in general on the
  initial condition within the state, but also on the realization of
  the noise. Notice that if this is not the case it means that the
  process actually quickly left the domain, so that there is no need
  to rely on an approximation to model the exit event: a simple cheap
  simulation of the original model can be used. Moreover, in a
  situation where metastability can be reinforced (e.g. by lowering
  the temperature for energetic barriers, or reducing the sizes of the
  exit windows for entropic barriers), it can typically be checked
  that the exit time grows much faster than the time to reach the
  quasi-stationary distribution. Identifying for a given model the initial conditions for which the law of the exit event is close to the law of the exit event starting from the quasi-stationary distribution can be done using so-called leveling results, see for example~\cite[Section 2.2]{lelievrePeutrecNectoux1}. In our specific context of entropic barriers, we intend to explore this in future works. Let us mention that the dependence of the mean exit time on the initial condition for entropic barriers has actually already been studied in the physics literature, see Section~\ref{sec:biblio} below for bibliographic comments.
\end{rem}

\subsection{Main results.}\label{Resultats principaux sur Omega}
Let us now present our main results. To perform our analysis, we consider the case when~$\Omega$ is the unit disk, \emph{i.e.}
\[
\Omega = \mathbf{B}(0,1).
\]
The first result concerns the spectral analysis of the operator $\mathcal{L}^{\varepsilon}$. A relevant parameter for the analysis is  
 \begin{equation}\label{eq:barK}
 \overline{K}_{\varepsilon}=\displaystyle\sum_{\substack{k=1}}^{N}K^{(k)}_{\varepsilon}.
 \end{equation}
Upon reducing $\varepsilon_0$, one can assume without loss of generality that, for all $\varepsilon \in(0,\varepsilon_0)$, 
 \begin{equation}\label{eq:barK_hyp}
 \overline{K}_{\varepsilon} < \frac12 \min \left( \frac{1}{|\log (\rho_0/2)|} ,1 \right),
 \end{equation}
an upper bound that will be used below.
 
 \begin{theo}\label{One eigenvalue results N holes}
Let 
$\mathcal{L}^{\varepsilon}$  be the   unbounded operator defined on $L^{2}(\Omega)$ with domain 
\begin{align}\label{domain}
     \mathcal{D}(\mathcal{L}^{\varepsilon})=\left\{u \in  H^{1}(\Omega), \, \, \Delta u \in L^{2}(\Omega), \, \,  \partial_{n}u_{|_{\Gamma^{\varepsilon}_{\mathbf{N}}}}=0,\, \,u_{|_{\Gamma^{\varepsilon}_{\mathbf{D}}}}=0 \right\}
\end{align} 
and acting as $\mathcal{L}^{\varepsilon}=-\Delta$.
Then, 
\begin{enumerate}
    \item[(i)] the operator $\mathcal{L}^{\varepsilon}$ is nonnegative self-adjoint and has  compact resolvent;
 \item [(ii)] there exist $c > 0$ and $\varepsilon_{0}>0$ such that, for any $\varepsilon\in (0,\varepsilon_{0})$,  
 \begin{align*}
     \operatorname{dim} \operatorname{Ran}\hspace{0.07cm}\pi_{[0,c\hspace{0.01cm}\overline{K}_{\varepsilon}]}(\mathcal{L}^{\varepsilon})=1,
 \end{align*}
 with $\pi_{[0,c\hspace{0.01cm}\overline{K}_{\varepsilon}]}(\mathcal{L}^{\varepsilon})$ is the spectral projection of $\mathcal{L}^{\varepsilon}$ on $[0,c\hspace{0.01cm}\overline{K}_{\varepsilon}]$. 
\end{enumerate}
\end{theo}
In particular, this result shows that the first eigenpair
$(\lambda_0^\varepsilon,u_0^\varepsilon)$ is well defined, and also
provides the upper bound~$0 \leq \lambda_0^\varepsilon \leq c
\overline{K}_{\varepsilon}$ in terms of the quantity introduced
in~\eqref{eq:barK}. Let us mention that Theorem~\ref{Oneholepropertie}
in Appendix~\ref{Laplacian p forms} gives the counterpart of this
result for~$p$-forms ($p \ge 1$).

The second result describes the exit event. In this work, we consider
that the domain is actually a small modification of the disk (the unit
ball in dimension 2), denoted by~$\widetilde{\Omega}_{\varepsilon}$,
see Figure \ref{Domain  for N=3} for a schematic representation. Let
us now present why introducing this specific
domain allows us to obtain precise asymptotic estimates of the first
eigenvalue and of the law of the first exit point.

\subsubsection*{Heuristic construction of approximate solutions to~\eqref{mixed on disk}.}
In order to define the modified domain~$\widetilde{\Omega}_\varepsilon$, we need to introduce the quasimode which is used to approximate the first eigenfunction~$u_0^\varepsilon$. The discussion below is informal: rigorous justifications of the quality of the quasimode are given in the forthcoming sections.

Let us first recall that, for a self-adjoint operator $T$ on a Hilbert space~$\mathcal{H}$, an $\mathcal{H}$-normalized element~$u$ such that~$\|\pi_{[b,+\infty)}(T)u\|$ is small is called a quasi-mode for the spectrum in $[0, b]$ of~$T$. Bounds on the spectral projection~$\|\pi_{[b,+\infty)}(T)u\|$ are typically obtained from bounds on the quadratic form associated with~$T$, see~\cite[Lemma 2.4.5]{Helffernier} (as well as Lemma~\ref{estimate} below in the context of this work).

The use of quasimodes to analyze the asymptotic spectral properties of operator is standard in semiclassical analysis (see~\cite{cycon-froese-kirsch-simon-87,helffer1988semi,dimassi1999spectral}). We adopt here a similar strategy, but in a different context since the parameter~$\varepsilon$ is not a semiclassical parameter. Recall that the eigenvalue problem of interest is given by~\eqref{mixed on disk}, with $|\Gamma^{\varepsilon}_{\mathbf{D}}|\rightarrow 0$ (in view of~\eqref{assumption on Gamma_D}). The operator $\mathcal{L}^{\varepsilon}$ defined in~Theorem~\ref{One eigenvalue results N holes} can therefore be thought of as a perturbation of the Laplacian with Neumann boundary conditions. 
In particular, one expects that $\lambda_{0}^{\varepsilon}\rightarrow 0$ when $\varepsilon\rightarrow 0$ (which is compatible with the intuition that the mean exit time should go to infinity in this regime, see~\eqref{tau}). The~$L^2$ normalized eigenvector associated with the first eigenvalue~0 of the Neumann Laplacian is (up to a sign) the constant function~$-1/\sqrt{\pi}$. This suggests that~$u_0^\varepsilon$ should be a perturbation of~$-1/\sqrt{\pi}$, at least sufficiently far away from the Dirichlet zones. This motivates introducing the function $\overline{u}^{\varepsilon}=u_{0}^{\varepsilon}+ 1/\sqrt{\pi}$. In view of~\eqref{mixed on disk}, the function $\overline{u}^{\varepsilon}$ should satisfy, at leading order in~$\varepsilon$ and away from the Dirichlet regions,
\begin{align}\label{problem MFPT}
 \left\{
    \begin{array}{rll}        
    \Delta \overline{u}^{\varepsilon} & \!\!\!\! \displaystyle = \frac{\lambda_{0}^{\varepsilon}}{\sqrt{\pi}} & \quad \text{ in } \Omega,\\
    \partial_{n}\overline{u}^{\varepsilon}& \!\!\!\! = 0 & \quad\text{ on }\Gamma^{\varepsilon}_{\mathbf{N}}.
     \end{array}
\right.
    \end{align}
This suggests to rewrite~$\overline{u}^\varepsilon$ as
    \[
    \overline{u}^\varepsilon = - \frac{\lambda_0^\varepsilon}{\sqrt \pi} f
\]
for some function~$f$ formally satisfying (by sending  $\varepsilon$ to 0 in~\eqref{problem MFPT}):
\begin{align}
 \left\{
    \begin{array}{rll}        
    \Delta f & \!\!\!\! \displaystyle = -1 & \quad \text{ in } \Omega,\\
    \partial_{n}f& \!\!\!\! = 0 
& \quad\text{  on } \partial \Omega \setminus \{x^{(1)},\ldots,x^{(K)}\}.
     \end{array}
\right.
    \end{align}

For a single exit region ($N=1$) centered at~$x^{(1)} \in \partial \Omega$, the function~$f$ formally satisfies 
\begin{equation}
\label{eq:Delta_f_Neumann_bord}
\left\{
    \begin{array}{rll}    
    \Delta f & \!\!\!\! = -1 & \quad \text{ in } \Omega,\\
    \partial_{n} f& \!\!\!\! = -|\Omega|\delta_{x^{(1)}} & \quad\text{ on } \partial \Omega,
     \end{array}
\right.
\end{equation}
where the Dirac mass on the boundary comes from the compatibility condition
\[
\int_{\partial \Omega} \partial_n f = \int_\Omega \Delta f = -|\Omega|.
\]
In view of~\cite[Equation~(3.13)]{Ammari2}, the solution~$f$ to~\eqref{eq:Delta_f_Neumann_bord} is (up to an irrelevant additive constant\footnote{The additive constant is irrelevant in the following since, as will become clear below, the crucial quantities to prove that $\varphi^\varepsilon$ is a good quasimode, and to then identify the first eigenvalue and the first exit point distribution, are derivatives of $\varphi^\varepsilon$.})
\[
f(x)=\log \left|x-x^{(1)}\right|+\frac{1-|x|^{2}}{4},
\]
where $|\cdot|$ denotes the Euclidean norm. Notice that similar
formulas for~$f$ are also used in works using formal asymptotic
expansions to get mean first passage times, see
e.g.~\cite[Equation~(2.23)]{SSH2}, and~\cite[Equation~(2.14)]{PWK} for
extensions to more general domains. For~$N=1$, an approximation of the  solution $u_0^\varepsilon$ to~\eqref{mixed on disk}--\eqref{eq:normalization_u_0_eps}  is therefore
\[
\varphi^\varepsilon = -\frac{1}{\sqrt{\pi}}-\frac{\lambda_0^\varepsilon}{\sqrt{\pi}}f.
\]
The value of~$\lambda_0^\varepsilon$ allowing to satisfy the Dirichlet boundary conditions at a distance~$\mathrm{e}^{-1/K_\varepsilon^{(1)}}$ of~$x^{(1)}$ is approximately determined by the condition~$1+\lambda_0^\varepsilon \log |x-x^{(1)}| = 0$, so that $\lambda_0^\varepsilon = K_\varepsilon^{(1)}$ (where we recall that the parameters $(K^{(k)}_{\varepsilon})_{k=1,\ldots,N}$ are related to the sizes of the absorbing windows, see~\eqref{assumption on Gamma_D}).

When there are $N$ absorbing windows, the above analysis suggests that~$u_0^\varepsilon$ solution to~\eqref{mixed on disk}--\eqref{eq:normalization_u_0_eps} should be well approximated away from the Dirichlet regions by
\begin{equation}
       \label{QuasiNhole}
   {\varphi}^{\varepsilon} := -\frac{1}{\sqrt{\pi}}-\frac{1}{\sqrt{\pi}}\sum_{k=1}^{N}K^{(k)}_{\varepsilon}f_{k}, 
   \qquad 
   f_k(x) = \log \left|x-x^{(k)}\right|+\frac{1-|x|^{2}}{4}.
\end{equation}
The link between the parameter
$K^{(k)}_{\varepsilon}$ and the Dirichlet region $\Gamma^{\varepsilon}_{\mathbf{D}_k}$ can be understood as
follows. Notice first that the function $\varphi^\varepsilon$ converges pointwise to
$-1/\sqrt{\pi}$ in~$\Omega$ in the limit $\varepsilon \to 0$, and, for
a fixed small $\varepsilon$, it is negative in a large central part of
$\Omega$, while it diverges to~$+\infty$ in the vicinity of the
boundary points $x^{(k)}$. In the small $\varepsilon$ limit, $\varphi^\varepsilon(x) =
0$ approximately translates into $K_\varepsilon^{(k)}f_k(x) = 1$
around the exit~$x^{(k)}$ (since only the function~$f_k$ diverges, the
other ones remaining bounded as~$x$ approaches the
boundary). Since~$(1-|x|^2)/4$ is bounded and
$K^{(k)}_{\varepsilon}\rightarrow 0$ as $\varepsilon\rightarrow0$, the
condition $K_\varepsilon^{(k)}f_k(x) = -1$ means
that $\Gamma^{\varepsilon}_{\mathbf{D}_k}$ asymptotically
satisfies~\eqref{assumption on Gamma_D}, see Lemma~\ref{levelset} below for precise statements.
Besides, the computation of~$\Delta \varphi^\varepsilon =
\overline{K}_\varepsilon/\sqrt{\pi}$ suggests that, in the limit $\varepsilon \to 0$, at leading order,
the eigenvalue~$\lambda_0^\varepsilon$ is~$\overline{K}_{\varepsilon}$
 (as defined by~\eqref{eq:barK}).

In order to make the previous reasoning rigorous, we face two
difficulties. First, the function~$\varphi^\varepsilon$ is not in
$\mathcal D(\mathcal L^\varepsilon)$ since it does not satisfy the
Dirichlet boundary conditions on
$\Gamma^{\varepsilon}_{\mathbf{D}}$. Second, another more fundamental
issue is related to the fact that functions in ${\mathcal D}(\mathcal
L^\varepsilon)$ do not admit in general normal derivatives in
$L^2(\partial \Omega)$. This is due to singularities which appear at
the intersections of the Neumann and Dirichlet boundaries (which meet
with an angle $\pi$), see for example the counterexample mentioned in
the introduction of~\cite{jakab2009regularity}. This lack of regularity implies difficulties to define the exit point distribution as~\eqref{law of X_t}, and to perform integration by parts in computations to prove that $\varphi^\varepsilon$ is a good quadi-mode. We will circumvent these two difficulties by working on a modified domain.


\subsubsection*{Construction of the modified domain.}
We are now in position to precisely introduce the modified domain $\widetilde{\Omega}_\varepsilon$, and
the associated modification of the eigenvalue problem~\eqref{mixed on disk} that we will consider in this work. The modified domain is defined by:
\begin{align}
  \label{The regularity domain}
  \widetilde{\Omega}_\varepsilon = \Omega \cap \left(\varphi^\varepsilon\right)^{-1}(-\infty,0].
\end{align}
The associated~$N$ Dirichlet regions around the exit points~$x^{(k)}$
are denoted by~$\widetilde{\Gamma}^{\varepsilon}_{\mathbf{D}_k}$
for~$k\in \{1,\dots,N\}$,
with~$\widetilde{\Gamma}^{\varepsilon}_{\mathbf{D}_k}$ defined as the
connected component of~$\Omega \cap
\left(\varphi^\varepsilon\right)^{-1}\{0\}$ which is the closest to
$x^{(k)}$ (see Lemma~\ref{lem:widetilde_Gamma_D_well_def} below for a proper definition of $\widetilde{\Gamma}^{\varepsilon}_{\mathbf{D}_k}$). We refer to Figure~\ref{Domain for N=3} for a schematic representation in the case $N=3$. We are thus interested in the exit event $(\widetilde{\tau}_\varepsilon,X_{\widetilde{\tau}_\varepsilon})$, where
\begin{equation}\label{tildetau}
    \widetilde{\tau}_\varepsilon:=\inf \left\{t\geq0,\hspace{0.1cm} X_t\notin\overline{\widetilde\Omega_\varepsilon} \, \right\}.
\end{equation}

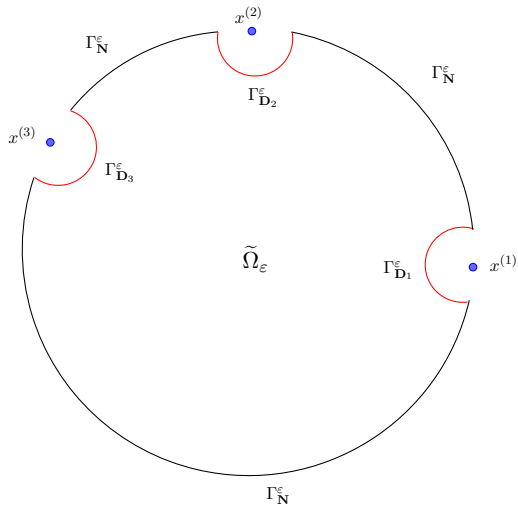
\begin{figure}
\begin{center}
\begin{tikzpicture}[scale=1]
 \draw (2.5,0.4) arc (15:79:3);
 \draw (-0.8,2.57) arc (95:141:2.84);
 \draw[red] (2.6,-0.05) arc (73:277:0.5);
 \draw[red] (-0.8,2.57) arc (170:370:0.5);
 \draw[red] (-2.76,1.52)  arc (71:-128:0.51);
   \draw[black] (2.5,0.4) arc (15:6:3);
     \draw[black] (2.54,-1) arc (-13:-198.5:3.01);
       \draw[fill=blue!60,draw=blue] (-0.35,2.58) circle (0.5mm);
    \draw[fill=blue!60,draw=blue] (2.59, -0.56) circle (0.5mm);
    \draw (-0.39,2.8) node[scale=0.6] {$x^{(2)}$};
     \draw[fill=blue!60,draw=blue] (-3.03,1.1) circle (0.5mm);
    \draw (-3.4,1.2) node[scale=0.6] {$x^{(3)}$};
         \draw (3,-0.5) node[scale=0.6] {$x^{(1)}$};
   
          \draw (-0.2,1.7) node[scale=0.6] {$\Gamma^{\varepsilon}_{\mathbf{D}_{2}}$};
          \draw (1.6,-0.6) node[scale=0.6] {$\Gamma^{\varepsilon}_{\mathbf{D}_{1}}$};
          \draw (-2.1,0.7) node[scale=0.6] {$\Gamma^{\varepsilon}_{\mathbf{D}_{3}}$};
           \draw (-2.4,2.4) node[scale=0.6] {$\Gamma^{\varepsilon}_{\mathbf{N}}$};
             \draw (2.2,2) node[scale=0.6] {$\Gamma^{\varepsilon}_{\mathbf{N}}$};
              \draw (0,-3.6) node[scale=0.6] {$\Gamma^{\varepsilon}_{\mathbf{N}}$};
\draw (-0.3, -0.45) node[scale=0.8] {$ \widetilde{\Omega}_{\varepsilon}$};
\end{tikzpicture}
\caption{The domain $\widetilde{\Omega}_{\varepsilon}$ for $N=3$}\label{Domain for N=3}
\end{center}
\end{figure}

As quantified in Lemmas~\ref{levelset} below, the
domain~$\widetilde{\Omega}_\varepsilon$ and the associated exit
regions are small modifications of~$\Omega$ and of the exit
regions~\eqref{assumption on Gamma_D}, and it is therefore expected
that the first exit time and first exit point distribution are not
substantially changed when passing from~$\Omega$
to~$\widetilde{\Omega}_\varepsilon$ (for the sake of conciseness, this
is however not formally proven in this article, and left for future
work). In particular, as explained in Section~\ref{sec:ppties_modified_operator}, the results stated on $\mathcal L^\varepsilon$ in Theorem~\ref{One eigenvalue results N holes} also hold for~$\widetilde{\mathcal L}^\varepsilon$, which is defined similarly as $\mathcal L^\varepsilon$ but on the modified domain~$\widetilde{\Omega}_\varepsilon$.

The following lemma follows directly from Lemmas \ref{Cercle 1} and \ref{cercle 2} (with~$\alpha=1$), and  shows that the exit regions on the modified domain $\widetilde{\Omega}_\varepsilon$ are indeed asymptotically close to the exit regions~\eqref{assumption on Gamma_D} introduced on the original domain $\Omega$.

\begin{lem}\label{levelset}
  There exists $C_{-}>C_{+}>0$ such that, for any $\eps \in (0,\varepsilon_{0})$, 
 \begin{align*}
\left( \bigcup\limits_{k=1}^{N} \mathbf{B}\left(x^{(k)},r^{(k)}_{\varepsilon,-}\right)\right) \cap  \Omega \subset({\varphi}^{\varepsilon})^{-1}\left([0,+\infty[\right)\cap  \Omega\subset \left( \bigcup\limits_{k=1}^{N} \mathbf{B}\left(x^{(k)}, r^{(k)}_{\varepsilon,+}\right)\right) \cap \Omega,
 \end{align*}
 where ${r^{(k)}_{\varepsilon,-}}=\mathrm{e}^{-C_{-}/K^{(k)}_{\varepsilon}}$ and $
{r^{(k)}_{\varepsilon,+}}=\mathrm{e}^{-{C}_{+}/K^{(k)}_{\varepsilon}}$ (see~\eqref{r+ and r-} below).
\end{lem}

According to Lemma~\ref{levelset}, the function ${\varphi}^{\varepsilon}$ vanishes in $\Omega$ on $N$ disjoint curves $\left(\widetilde{\Gamma}^{\varepsilon}_{\mathbf{D}_k}\right)_{k=1\ldots,N}$, contained in an annulus centered at $x^{(k)}$, with radii~$r^{(k)}_{\varepsilon,-}$ and $r^{(k)}_{\varepsilon,+}$. More precisely, for $k\in \{1,\dots,N\}$, the curve~$\widetilde{\Gamma}^{\varepsilon}_{\mathbf{D}_k} \subset \left(\varphi^\varepsilon\right)^{-1}\{0\} \cup \partial\Omega$ is the connected component of $\Omega \cap (\varphi^\varepsilon)^{-1}\{0\}$ closest to~$x^{(k)}$; see Lemma~\ref{eq:further_reduced_varepsilon_0} (which requires further reducing~$\varepsilon_0$ to ensure that~\eqref{lem:widetilde_Gamma_D_well_def} holds). The boundary of~$\widetilde{\Omega}_\varepsilon$ is thus composed of two parts:
\begin{equation}\label{tilde partial}
\partial \widetilde{\Omega}_\varepsilon = \overline{\widetilde{\Gamma}^{\varepsilon}_{\mathbf{D}}} \cup \overline{\widetilde{\Gamma}^{\varepsilon}_{\mathbf{N}}}, 
\qquad \widetilde{\Gamma}^{\varepsilon}_{\mathbf{D}} =\bigcup_{k=1}^{N}\widetilde{\Gamma}^{\varepsilon}_{\mathbf{D}_{k}},
\qquad
\widetilde{\Gamma}^{\varepsilon}_{\mathbf{D}} \cap \widetilde{\Gamma}^{\varepsilon}_{\mathbf{N}} = \emptyset,
\end{equation}
where~$\widetilde{\Gamma}^{\varepsilon}_{\mathbf{N}}\subset \partial \Omega$ is the reflecting part. 

Our goal now is to examine the narrow escape problem on the domain $\widetilde{\Omega}_{\varepsilon}$; more precisely to study the following eigenvalue problem, which is a modification of~\eqref{mixed on disk} where~$\Omega$ is replaced
by~$\widetilde{\Omega}_{\varepsilon}$: 
\begin{align}\label{mixed}
 \left\{
      \hspace{-0.1cm}\begin{array}{ll}
   -\Delta  \widetilde{u}_{0}^{\varepsilon}&\hspace{-0.3cm}=\widetilde{\lambda}_{0}^{\varepsilon} \, \widetilde{u}_{0}^{\varepsilon}\,\,\text{in} \,\,\widetilde{\Omega}_{\varepsilon},\\
    \partial_{n} \widetilde{u}_{0}^{\varepsilon}&\hspace{-0.3cm}=0\,\,\,\text{on}\,\,\widetilde{\Gamma}^{\varepsilon}_{\mathbf{N}},\\
   \hspace{0.5cm}\widetilde{u}_{0}^{\varepsilon}&\hspace{-0.3cm}=0\,\,\,\text{on}\,\,\widetilde{\Gamma}^{\varepsilon}_{\mathbf{D}}.
     \end{array}
\right.
    \end{align} 
The eigenfunction is normalized so that 
\begin{equation}
\label{eq:normalization_widetilde_u_0_eps}
\int_{\widetilde{\Omega}_{\varepsilon}}\left(\widetilde{u}_{0}^{\varepsilon}(x)\right)^{2} \, \mathrm{d}x= 1, 
\qquad 
\widetilde{u}_{0}^{\varepsilon}<0 \textrm{ on } \widetilde{\Omega}_{\varepsilon}.
\end{equation}
In the following, we denote by $\widetilde{\mathcal L}^\varepsilon$ the operator associated with~\eqref{mixed}. More precisely, $\widetilde{\mathcal L}^\varepsilon$ is the operator with domain
\begin{align}\label{domaintilde}
     \mathcal{D}(\widetilde{\mathcal{L}}^{\varepsilon})=\left\{u \in  H^{1}(\widetilde{\Omega}_\varepsilon), \, \, \Delta u \in L^{2}(\widetilde{\Omega}_\varepsilon), \, \,  \partial_{n}u_{|_{\widetilde{\Gamma}^{\varepsilon}_{\mathbf{N}}}}=0,\, \,u_{|_{\widetilde{\Gamma}^{\varepsilon}_{\mathbf{D}}}}=0 \right\},
\end{align} 
acting as $\widetilde{\mathcal L}^\varepsilon = -\Delta$. The unique quasi-stationary distribution of the process~\eqref{lang} in $\widetilde{\Omega}_\varepsilon$ is
$$
\widetilde{\nu}_{0}^{\varepsilon}(\mathrm{d}x)=\frac{\widetilde u_{0}^{\varepsilon}(x)\hspace{0.1cm}\mathrm{d}x}{\displaystyle\int_{\Omega}\widetilde  u_{0}^{\varepsilon}(y)\hspace{0.1cm}\mathrm{d}y}.
$$

Working on  the modified domain $\widetilde{\Omega}_{\varepsilon}$ allows us to obtain a precise asymptotic estimate of the first eigenvalue and eigenvector $(\widetilde{\lambda}_0^\varepsilon,\widetilde{u}_0^\varepsilon)$, using the fact that the quasimode $\varphi^\varepsilon$ is by construction in the domain of $\widetilde{\mathcal L}^\varepsilon$; in particular, it satisfies the Neumann (resp. Dirichlet) boundary conditions on $\widetilde{\Gamma}^{\varepsilon}_{\mathbf{N}}$ (resp. $\widetilde{\Gamma}^{\varepsilon}_{\mathbf{D}}$), see~\eqref{normal of f_k} below for the Neumann boundary condition. Moreover, 
\begin{equation}
  \label{eq:angle_intersection}
  \textrm{for any $k\in\{1,\ldots,N\}$, the curves~$\widetilde{\Gamma}^{\varepsilon}_{\mathbf{D}_k}$ and $\widetilde{\Gamma}^{\varepsilon}_{\mathbf{N}}$ meet at an angle equal to $\frac{\pi}{2}$}
\end{equation}
since~$\partial_n \varphi^\varepsilon$ vanishes at the intersection points between~$\partial \Omega$ (whose normal is~$\Vec{n}$) and~$\left(\varphi^\varepsilon\right)^{-1}\{0\}$ (whose normal is~$\nabla\varphi^\varepsilon$). This implies that functions in ${\mathcal D}(\widetilde{\mathcal L}^\varepsilon)$ admit a normal derivative in $L^2(\partial \widetilde{\Omega}_{\varepsilon})$, as made precise in the next lemma.

\begin{lem}
\label{lem:reg_L2_normal_derivative}
   For any $v \in  \mathcal{D}(\widetilde{\mathcal{L}}^{\varepsilon})$, the normal derivative $\partial_n v$ is well defined as an $L^2$ function on $\partial \widetilde{\Omega}_{\varepsilon}$. 
\end{lem}
\begin{proof}
  Since  $\widetilde{\Gamma}^{\varepsilon}_{\mathbf{D}_{k}}$ and $\widetilde{\Gamma}^{\varepsilon}_{\mathbf{N}}$ meet at an angle equal to~$\frac{\pi}{2}$ for all $k\in\{1,\dots,N\}$ (see~\eqref{eq:angle_intersection}, as well as~\cite[Definition 31]{LelievrePeutrecNectoux} for more details about the angle between two hypersurfaces), then \cite[Proposition 32]{LelievrePeutrecNectoux}, which is a direct consequence of~\cite{jakab2009regularity,gol2011hodge}, implies that
$       \partial_{n}{v}\in L^{2}(\partial\widetilde{\Omega}_{\varepsilon})$.
\end{proof}

From Lemma~\ref{lem:reg_L2_normal_derivative}, the normal derivative of the first eigenvector $\widetilde{u}^{\varepsilon}_0$ is in $L^{2}(\partial\widetilde{\Omega}_{\varepsilon})$. This is important to identify the first exit point distribution as
\begin{align}\label{law of tildeX_t}\frac{\partial_{n} \widetilde{u}_{0}^{\varepsilon}(x) \, \sigma(dx)}{\displaystyle\int_{\partial\Omega}\partial_{n}\widetilde{u}_{0}^{\varepsilon}(y)\,\sigma(\mathrm{d}y)},
\end{align}
where $\sigma$ here denotes the Lebesgue surface measures on $\partial \widetilde{\Omega}_{\varepsilon}$. From a more technical viewpoint, this will be useful  to prove that $\partial_n \varphi^\varepsilon$ is an excellent approximation of~$\partial_n \widetilde{u}_0^\varepsilon$ on $\widetilde \Gamma^\varepsilon_{\mathbf D}$. 

Since~$-\Delta f_k = 1$ away from~$x^{(k)}$ (see~\eqref{laplacian fk} below), one immediately gets that the following equality holds pointwise in~$\widetilde{\Omega}_\varepsilon$ (using the notation $\overline{K}_{\varepsilon}$ introduced in~\eqref{eq:barK}):
\[
  -\Delta{\varphi}^{\varepsilon} =- \frac{1}{\sqrt{\pi}}\sum_{k=1}^{N}K^{(k)}_{\varepsilon} = \overline{K}_\varepsilon {\varphi}^{\varepsilon}+ \mathrm{o}(1).
  \]
  Note that we crucially use here that~$\varphi^\varepsilon$ is bounded on~$\widetilde{\Omega}_\varepsilon$ uniformy in~$\varepsilon \in (0,\varepsilon_0)$. This calculation suggests that~$\widetilde{\lambda}_{0}^{\varepsilon}$ is at dominant order equal to $\overline{K}_\varepsilon$, as made precise in Theorem~\ref{asymlambda N hole}. Let us recall that this result characterizes the law of the
first exit time $\widetilde{\tau}_\varepsilon$, which is indeed exponential with parameter $\widetilde{\lambda}_{0}^{\varepsilon}$.
\begin{theo}\label{asymlambda N hole}
The first eigenvalue  $\widetilde{\lambda}_{0}^{\varepsilon}$  of the system \eqref{mixed} satisfies, in the limit $\varepsilon \to 0$,  
\begin{align}\label{lambdaNhole}
    \widetilde{\lambda}_{0}^{\varepsilon}= \overline{K}_{\varepsilon}+\mathrm{O}\left(\overline{K}^{2}_{\varepsilon}\right).
\end{align}
\end{theo}
This result is fully consistent with results obtained in the physics literature on the mean first passage time, whose inverse indeed has the same asymptotic behavior as $\widetilde{\lambda}_{0}^{\varepsilon}$, see in particular~\cite{SSH2} for our specific context of a particle confined in a disk. Using formal asymptotic expansions, the authors obtain the same asymptotic behavior, with additional results concerning the dependency on the initial distribution and higher order terms in the expansion.
 
The final result concerns the law of the first exit point $X_{\widetilde{\tau}_\varepsilon}$, which is given by~\eqref{law of tildeX_t}. Indeed, one can prove the following asymptotic behavior of the normal derivative of $\widetilde{u}_{0}^{\varepsilon}$ as $\varepsilon \to 0$.

\begin{theo}\label{exit point}
Let $k\in\{1,\dots,N\}$. Let $\widetilde{u}_{0}^{\varepsilon}$ 
 be the eigenfunction of $\widetilde{\mathcal L}^{\varepsilon}$ associated with the first eigenvalue $\widetilde{\lambda}_{0}^{\varepsilon}$ (see~\eqref{mixed}) and normalized as~\eqref{eq:normalization_widetilde_u_0_eps}. Then,
in the limit $\varepsilon\rightarrow0$,
 \begin{align}\label{normal derivatice on Gamma D}
    \int_{\widetilde{\Gamma}^{\varepsilon}_{\mathbf{D}}} \partial_n \widetilde{u}_{0}^{\varepsilon}=\sqrt{\pi}\, \overline{K}_\varepsilon+\mathrm{O}\left(\overline{K}_\varepsilon^2\right).
    \end{align}
   Moreover, in the limit $\varepsilon\rightarrow0$, for all $k\in\{1,\dots,N\}$, 
   \begin{align}\label{normal derivatice on Gamma Dk}
    \int_{\widetilde{\Gamma}^{\varepsilon}_{\mathbf{D} _k}} \partial_n\widetilde{u}_{0}^{\varepsilon}=\sqrt{\pi}{K}^{(k)}_\varepsilon+\mathrm{O}\left(\overline{K}^{3/2}_\varepsilon\right).
    \end{align}
    As a consequence, in the limit $\varepsilon\rightarrow0$, for all $k\in\{1,\dots,N\}$, 
    \begin{align}\label{resultat proba}
\mathbb{P}_{\widetilde{\nu}_{0}^{\varepsilon}}\left[X_{\widetilde{\tau}_\varepsilon}\in\Gamma^{\varepsilon}_{\mathbf{D} _k}\right]=\frac{{K}^{(k)}_\varepsilon}{\overline{K}_\varepsilon}+\mathrm{O}\left(\sqrt{\overline{K}_\varepsilon}\right).
    \end{align}
\end{theo}

Note that the expression~\eqref{normal derivatice on Gamma D} is not obtained by summing~\eqref{normal derivatice on Gamma Dk} over $k\in \{1, \ldots,N\}$, as the error term in~\eqref{normal derivatice on Gamma D} is smaller than the one in~\eqref{normal derivatice on Gamma Dk}. 

In order for this result to be useful, the remainder term of
order~$\overline{K}_\varepsilon^{1/2}$ in~\eqref{resultat proba}
should be small compared to
$K^{(k)}_\varepsilon/\overline{K}_\varepsilon$, that is
${K}^{(k)}_\varepsilon$ should be of order at most~$\overline{K}_\varepsilon^{3/2}$. This can be illustrated on two prototypical situations:
\begin{itemize}
\item If all the windows shrink with the same scaling,
  namely if for all $k \in \{1, \ldots, N\}$,
  $\mathrm{e}^{-1/K^{(k)}_{\varepsilon}}=a_k \varepsilon$ for some
  positive real numbers $(a_k)_{k=1,\ldots,N}$, then $K^{(k)}_{\varepsilon}=-\frac{1}{\log a_k + \log \varepsilon}$: in this case, $\mathbb{P}_{\widetilde{\nu}_{0}^{\varepsilon}}\left[X_{\widetilde{\tau}_\varepsilon}\in\Gamma^{\varepsilon}_{\mathbf{D} _k}\right]$ converges to $1/N$ whatever the values of $(a_k)_{k=1,\ldots,N}$.
\item If the windows shrink at different scales, namely if for all $k
  \in \{1, \ldots,
  N\}$~$\mathrm{e}^{-1/K^{(k)}_{\varepsilon}}=\varepsilon^{a_k}$ for
  some positive real numbers $(a_k)_{k=1,\ldots,N}$, then $K^{(k)}_{\varepsilon}=-\frac{1}{a_k \log \varepsilon}$: in this case, $\mathbb{P}_{\widetilde{\nu}_{0}^{\varepsilon}}\left[X_{\widetilde{\tau}_\varepsilon}\in\Gamma^{\varepsilon}_{\mathbf{D} _k}\right]$ converges to $\frac{1/a_k}{\sum_{k=1}^N 1/a_k}$. Of course, windows which shrink at slower scales are more likely exits.
\end{itemize}


\subsection{Bibliographic comments and perspectives}\label{sec:biblio}
As already mentioned in the introduction, the narrow escape problem attracted a lot of attention from the physics community,
and early works on the asymptotic behavior of the mean first exit time
for a single infinitely small absorbing window
can for example be found
in~\cite{grigoriev2002kinetics,SSH,SSH2,singer2008narrow}, and
in~\cite{bib2,grebenkov2019full} where the authors are particularly
interested in studying the role of the initial condition. These works
typically rely on formal expansions of Neumann Green's function for
the Laplacian in the domain of interest, with Dirac masses at the
escape points, and on results from potential
theory~\cite{kellogg1967,england1979mixed,Ammari1}. In particular,
explicit asymptotic results can be obtained for domains with specific forms for which these Green's functions are analytically known (disks, balls, squares, etc). These early results have been extended to general domains and a finite number of infinitely small absorbing windows in~\cite{PWK,CWS}, using the method of matched asymptotic expansions. Similar techniques are used in \cite{li2014,LiLin2023} to study another geometry, namely a domain
which is composed of a relatively big head and several absorbing
narrow necks, which leads to  Neumann--Robin instead of Dirichlet
boundary conditions on the small exit regions; see
also~\cite{dagdug2003equilibration} for a study of two chambers
connected by a narrow tube. Let us also mention related works on the
analysis of how the eigenelements are perturbed by introducing small
Dirichlet boundary conditions in~\cite{ward1993strong}. There is thus a very large body of physics literature on this subject - our aim here is not to be exhaustive.

Rigorous derivations of the first and second order terms in the asymptotic expansion of the mean first exit time have been obtained in~\cite{Ammari2} using layer potential techniques from~\cite{Ammari1}, see also~\cite{nursultanov2021mean} for generalizations to Brownian particles on Riemannian manifolds. More recently,  mathematical proofs of the results from~\cite{CWS} have been obtained in~\cite{ChenFriedman}, where the authors prove the asymptotic expansion rigorously using subsolutions and supersolutions.

As mentioned earlier, our results differ from what has been done previously in two ways. First, our work is, to the best of our knowledge, the first to provide mathematical results on  the first exit point distribution (see however results in~\cite{chevalier2010first,cheviakov2011optimizing} on so-called splitting probabilities, a.k.a. committor functions, which are derived from formal asymptotic expansions: it would be interesting to get rigorous proofs of these results and compare them to what we obtain using the quasi-stationary distribution approach). Notice that studying the first exit point distribution requires asymptotic estimates in high Sobolev index norms (typically larger than $3/2$) on the first eigenvector, since this exit point distribution is expressed in terms of the normal derivative of the first eigenvector on the boundary, see~\eqref{law of X_t}. 
Second, the mathematical approach and tools used in our work to study the narrow escape problem (namely the quasi-stationary distribution approach to metastability and spectral techniques in the spirit of those used in semi-classical analysis) also seem to be new. As mentioned above, semi-classical techniques have been used a lot to study the exit problem in the case of an energetic barrier in the small temperature regime, see for example~\cite{HelfferNierKlein,Helffernier, LelievreNier, GesGiaLeliPeutrec, GesuLelievrePeutrecNectoux2, PeutrecMichel, LePeutrecNectoux, lelievrePeutrecNectoux1,nectoux2021mean,LelievrePeutrecNectoux,Nectoux}: our work is a first attempt to adapt these techniques to the case of entropic barriers. A major difference between the energetic and the entropic case is that for entropic barriers, the 1-eigenforms do not concentrate around the exit region as fast as for energetic barriers (there are no counterparts to Agmon estimates). We circumvent this difficulty by using an excellent quasimode of the 0-eigenform, as outlined above. To conclude this literature review, let us point out the work~\cite{hillairet2010eigenvalues} where spectral analytic tools are used to study the asymptotic behavior of the eigenvalues of the Laplacian with mixed Dirichlet-Neumann boundary conditions but in another geometric framework, namely domains with small slits.

This work opens many perspectives that we would like to explore in the
future. First, it would be interesting to prove that modifying the
domain from $\Omega$ to $\widetilde \Omega_\varepsilon$ (as expplained
in the previous section) indeed does not modify the asymptotic results that we have obtained on the exit event. Moreover, it should be possible to apply the same method to study the narrow escape problem in a general bounded domain in $\mathbb{R}^{d}$. Besides, it would be interesting to investigate if similar techniques can be used to study the narrow escape problem for the underdamped Langevin dynamics (commonly used in molecular dynamics): this is more challenging since the associated infinitesimal generator is non-selfadjoint and hypoelliptic, and this has major implications on the functional setting (in particular to state the boundary conditions, see e.g. \cite{LelievreRamilReygner} for more details) and on the spectral properties of the operator. Finally, as already mentioned above, the quasi-stationary approach to metastability is deeply linked with a class of numerical methods to efficienlty sample the exit events of metastable dynamics~\cite{di2016jump,lelievre-15,perez2024recent}. For example,  it would be interesting to study how the results we have obtained can be used to obtain generalizations of the temperature accelerated dynamics algorithm~\cite{sorensen-voter-00} in our context (the size of the exit windows playing the role of the temperature).

\subsection{Outline of the paper}
In Section~\ref{Spectral analysis}, we investigate the spectral properties of the Laplacian with mixed Dirichlet--Neumann boundary conditions.  In Section~\ref{Section law of first exit time}, we present the results regarding the law of the first exit time. Section~\ref{Section Law of the first exit point} focuses on the law of the  first exit point.  In Appendix~\ref{Laplacian p forms}, we present the spectral analysis of the Laplacian on $p$-forms with mixed tangential-normal boundary conditions. Appendix~\ref{app:3d} illustrates the generality of the approach presented in this work by quickly outlining how it could be applied to the three-dimensional
ball.

Some of the theoretical results on the spectral properties of the $p$-Laplacian with mixed tangential-normal boundary conditions  are illustrated along the manuscript by numerical experiments, using in particular Raviart--Thomas orthogonal finite elements (a.k.a. Nedelec finite elements) in order to obtain stable discretizations~\cite{arnold2014periodic}.


\section{Spectral analysis}\label{Spectral analysis}

In this section, we perform the spectral analysis of the operator
$\mathcal{L}^{\varepsilon}$. In Section
\ref{properties of L},  we properly define  the operator
$\mathcal{L}^{\varepsilon}$ and we introduce a few properties which
yield the proof of item $(i)$ of Theorem \ref{One eigenvalue results N holes}, as well as a
result concerning the second eigenvalue of  the operator~$\mathcal{L}^{\varepsilon}$.  The proof of item $(ii)$ of Theorem \ref{One eigenvalue results N holes} is given in Section \ref{Proof of Theo 1.1}. A numerical illustration of our theoretical findings is given in Section~\ref{Numerical illustration of 1.1}.

In all this section, we consider the operator
$\mathcal{L}^{\varepsilon}$ with domain~$\mathcal
D(\mathcal{L}^{\varepsilon})$, see~\eqref{domain}. Our aim is to
provide spectral properties of the operator
$\mathcal{L}^{\varepsilon}$, that will actually also hold for the
mixed Laplacian on the modified domain~$\widetilde\Omega_\varepsilon$
defined by~\eqref{The regularity domain},
denoted by~$\widetilde  {\mathcal L}^{\varepsilon}$
(with domain~\eqref{domaintilde}).

\subsection{Definition and first properties of the operator \texorpdfstring{$\mathcal{L}^{\varepsilon}$}{}}\label{properties of L}

Let us first study the spectral properties of the operator $\mathcal{L}^{\varepsilon}$, namely the mixed Laplacian in the unit disk $\Omega$ (see Figure~\ref{Narrow}). Recall that the boundary of the domain~$\Omega$ contains $N$ small disjoint exit regions $(\Gamma^{\varepsilon}_{\mathbf{D}_{k}})_{k= 1,\dots, N}$  of respective sizes~$\mathrm{e}^{-1/K^{(k)}_{\varepsilon}}$ (recall~\eqref{assumption on Gamma_D}). The boundary $\partial \Omega$ is thus composed of two
parts:
$\Gamma^{\varepsilon}_{\mathbf{D}}={\cup}_{k=1}^{N}\Gamma^{\varepsilon}_{\mathbf{D}_{k}}
$  which is the absorbing part  and $\Gamma^{\varepsilon}_{\mathbf{N}}$ which is the reflecting part.
We consider the mixed Dirichlet--Neumann Laplacian eigenvalue
problem~\eqref{mixed on disk}  on $\Omega$.  Let us introduce 
\begin{align}\label{defH1Nholes}
  H^{1}_{0,\Gamma^{\varepsilon}_{\mathbf{D}}}(\Omega):= \left\{u \in H^{1}(\Omega), \, \, u_{|_{\Gamma^{\varepsilon}_{\mathbf{D}}}}=0 \right\} ,
\end{align}
the space of functions in $H^{1}(\Omega)$ whose trace vanishes on $\Gamma^{\varepsilon}_{\mathbf{D}}$. 

\begin{prop}\label{def of domain}
The quadratic form
\begin{align}
    Q(u)=\int_{\Omega}{\vert \nabla u\vert}^{2}
\end{align}
with form domain $\mathcal D(Q)=H^{1}_{0,\Gamma^{\varepsilon}_{\mathbf{D}}}(\Omega)$ is closed and symmetric. Its Friedrichs extension is the operator
$\mathcal{L}^{\varepsilon}=-\Delta$ 
with domain  
\begin{align}\label{Domain N holes}
    \mathcal{D}(\mathcal{L}^{\varepsilon})=\left\{u \in  H^{1}(\Omega), \, \, \Delta u \in L^{2}(\Omega), \, \,  \partial_{n}u_{|_{\Gamma^{\varepsilon}_{\mathbf{N}}}}=0,\, \,u_{|_{\Gamma^{\varepsilon}_{\mathbf{D}}}}=0 \right\}.
    \end{align} 
    The operator $(\mathcal{L}^{\varepsilon},\mathcal{D}(\mathcal{L}^{\varepsilon}))$ is nonnegative and selfadjoint.
\end{prop}

\begin{proof}
    Let us consider the following quadratic form
    \begin{align*}
        u\in C_{\Gamma^{\varepsilon}_{\mathbf{D}}}^{\infty}(\Omega)\mapsto \int_{\Omega}{\vert \nabla u\vert}^{2},
    \end{align*}
    where $C_{\Gamma^{\varepsilon}_{\mathbf{D}}}^{\infty}(\Omega)$ is the set of functions belonging to~$C^\infty(\overline{\Omega})$ vanishing on~$\Gamma^{\varepsilon}_{\mathbf{D}}$. It is symmetric, nonnegative and closable. Its closure for the~$H^1$ norm is the quadratic form 
\begin{align*}
    Q:  u \in H^{1}_{0,\Gamma^{\varepsilon}_{\mathbf{D}}}(\Omega)\mapsto \int_{\Omega}{\vert \nabla u\vert}^{2}.
\end{align*} 
Let $\mathcal{L}^{\varepsilon}$ be the Friedrichs extension associated with the quadratic form $Q$. This operator is nonnegative self-adjoint on $L^{2}(\Omega)$, and defined on the domain
\begin{align*}
 \mathcal{D}(\mathcal{L}^{\varepsilon})=\left\{u \in H^{1}_{0,\Gamma^{\varepsilon}_{\mathbf{D}}}(\Omega) \, \middle| \,  \exists f\in L^{2}(\Omega),\, \forall v \in H^{1}_{0,\Gamma^{\varepsilon}_{\mathbf{D}}},\,Q(u,v)=\left\langle f,v\right\rangle_{L^{2}(\Omega)} \right\}
 \end{align*}
 by $\mathcal{L}^{\varepsilon} u=  -\Delta u=f$.
 It is then standard to check that, equivalently, the domain of the Laplacian with mixed Dirichlet--Neumann boundary conditions is given by~\eqref{Domain N holes}. Indeed, one first shows that~$-\Delta u = f$ in the sense of distributions, hence~$\Delta u \in L^2(\Omega)$ since~$f \in L^2(\Omega)$; and next sees that~$\partial_n u = 0$ on~$\Gamma^{\varepsilon}_{\mathbf{N}}$ as elements of~$H^{-1/2}(\partial \Omega)$ by an integration by parts. 
\end{proof}

\begin{rem}
  Concerning the regularity of the functions $u \in \mathcal{D}(\mathcal{L}^{\varepsilon})$, let us recall that because of the mixed boundary conditions, $u$ is not in $H^2(\Omega)$ (as would be the case for pure Dirichlet or Neumann boundary conditions) but only in $H^{3/2-\alpha}(\Omega)$ for any $\alpha > 0$ (this result holds for any $C^{1,1}$ domain). Actually it can be shown that $u\in B_{2,\infty}^{3/2}(\Omega)$,   where $B_{2,\infty}^{3/2}(\Omega)$ is defined by interpolation: $B_{2,\infty}^{3/2}(\Omega)=[H^{1}(\Omega),H^{2}(\Omega)]_{1/2,\infty}$.
  We refer for example to \cite{Savar} for more details. 
 \end{rem}

We are now in position to prove item {\em (i)} of Theorem~\ref{One eigenvalue results N holes}.

\begin{proof}[Proof of item (i) of Theorem \ref{One eigenvalue results N holes}]
We already stated in Proposition~\ref{def of domain} that $\mathcal{L}^{\varepsilon}$ is nonnegative and self-adjoint. It therefore only remains to prove that it has  compact resolvent. This is a consequence of the continuity of the inclusion $\mathcal{D}(\mathcal{L}^{\varepsilon})\subset H^{1}(\Omega)$, and of the compactness of the embedding $H^1(\Omega)\subset L^{2}(\Omega)$ (guaranteed by the boundedness of~$\Omega$). 
\end{proof}

A consequence of the above result is that the eigenvalues of the mixed Laplacian $\mathcal{L}^{\varepsilon}$, when ranked in increasing order and counted with their multiplicities, are given by the min-max principle: for $n=0$,
\begin{align}\label{min0}
    \lambda^{\varepsilon}_{0}={\underset{u\in H^{1}_{0,\Gamma^{\varepsilon}_{\mathbf{D}}}(\Omega)\backslash\{0\}}{\text{inf}}}\left\{ \frac{Q(u)}{\| u\|^{2}_{L^2(\Omega)}} \right\},
\end{align}
and, for $n \ge 1$,
\begin{align}\label{min}
    \lambda^{\varepsilon}_{n}=\underset{{E}_{n}:\,\, \text{dim}\,{E}_{n}=n}{\text{sup}}\,\,{\underset{u\in \left(H^{1}_{0,\Gamma^{\varepsilon}_{\mathbf{D}}}(\Omega)\backslash\{0\}\right)\cap{E}^{\perp}_{n}}{\text{inf}}} \left\{ \frac{Q(u)}{\| u\|^{2}_{L^2(\Omega)}} \right\}.
\end{align}

A corollary of~\eqref{min} is a classical estimate on the second
eigenvalue of the mixed Dirichlet--Neumann Laplacian eigenvalue
problem~\eqref{mixed on disk}, by comparing it to the second
eigenvalue of either the Neumann or the Dirichlet Laplacian. This
result will be useful in the proof of  item~{\em (ii)} of
Theorem~\ref{One eigenvalue results N holes} below.

 In order to state the result, let us introduce the Dirichlet Laplacian eigenvalue problem on~$\Omega$:
\begin{align}\label{Dirichlet}
 \left\{
    \begin{array}{rll}
      -\Delta  u^{\mathbf{D}} & \!\!\!\! =\lambda^{\mathbf{D}} u^{\mathbf{D}} & \quad \text{in} \,\,\Omega,\\
     u^{\mathbf{D}} & \!\!\!\! =0 & \quad \text{on}\,\,\partial\Omega.
     \end{array}
\right.
    \end{align}
    Since $\Omega$ is a smooth bounded domain,
    the spectrum of this operator  is discrete and consists of eigenvalues $(\lambda^{\mathbf{D}}_{i})_{i\in\mathbb{N}}$ having finite multiplicities such that
    \begin{align}
      0<\lambda^{\mathbf{D}}_{0}<\lambda^{\mathbf{D}}_{1}\leq\lambda^{\mathbf{D}}_{2}\leq\lambda^{\mathbf{D}}_{3}\leq \lambda^{\mathbf{D}}_{4}\dots,
      \qquad
      \lambda^{\mathbf{D}}_{k} \xrightarrow[k \to \infty]{} \infty.
    \end{align}
Likewise, let us introduce the Neumann Laplacian eigenvalue problem on $\Omega$:
 \begin{align}\label{Neumann}
 \left\{
   \begin{array}{rll}
    -\Delta u^{\mathbf{N}} & \!\!\!\! =\lambda^{\mathbf{N}} u^{\mathbf{N}} & \quad\text{in} \,\,\Omega,\\
    \partial_{n}u^{\mathbf{N}} & \!\!\!\! =0 & \quad\text{on}\,\,\partial\Omega.
     \end{array}
\right.
    \end{align}
Again, the spectrum of this operator is discrete and consists of eigenvalues $(\lambda^{\mathbf{N}}_{i})_{i\in\mathbb{N}}$ with finite multiplicities, tending to infinity:
 \begin{align*}
   0=\lambda^{\mathbf{N}}_{0}<\lambda^{\mathbf{N}}_{1}\leq\lambda^{\mathbf{N}}_{2}\leq\lambda^{\mathbf{N}}_{3}\leq \lambda^{\mathbf{N}}_{4}\dots,
   \qquad
   \lambda^{\mathbf{N}}_{k} \xrightarrow[k \to \infty]{} \infty.
 \end{align*}
        
\begin{prop}\label{ineq}
     Let $\lambda^{\varepsilon}_{1}$ be the second eigenvalue of the mixed  Dirichlet--Neumann problem~\eqref{mixed on disk}. Then, for any~$\varepsilon \in (0,\varepsilon_0)$,
     \begin{align}
         \lambda^{\mathbf{N}}_{1}\leq \lambda^{\varepsilon}_{1}\leq \lambda^{\mathbf{D}}_{1},
     \end{align}
      where $\lambda^{\mathbf{N}}_{1}$ (resp. $\lambda^{\mathbf{D}}_{1}$) is the  second eigenvalue of the Neumann (resp. Dirichlet) problem (see respectively~\eqref{Neumann} and~\eqref{Dirichlet}).
\end{prop}

\begin{proof}
  Let us introduce
\begin{align}
  H^{1}_{0}(\Omega):= \left\{u \in H^{1}(\Omega), \, \, u_{|_{\partial\Omega}}=0 \right\},
\end{align}
the space of functions in $H^{1}(\Omega)$ whose trace vanishes on $\partial\Omega$.
The min-max principle implies that
\begin{align*}
     \lambda^{\mathbf{N}}_{1}&=\underset{{\psi_{1}}\in L^{2}(\Omega)\backslash\{0\}}{\text{sup}}\quad\underset{u\in \text{Span}(\psi_{1})^{\perp}}{\underset{u\in H^{1}(\Omega)\backslash\{0\}}{\text{inf}}} \left\{ \frac{Q(u)}{\| u\|^{2}_{L^2(\Omega)}} \right\}\\
     &\leq \underset{{\psi_{1}}\in L^{2}(\Omega)\backslash\{0\}}{\text{sup}}\quad\underset{u\in \text{Span}(\psi_{1})^{\perp}}{\underset{u\in H^{1}_{0,\Gamma^{\varepsilon}_{\mathbf{D}}}(\Omega)\backslash\{0\}}{\text{inf}}}\left\{ \frac{Q(u)}{\| u\|^{2}_{L^2(\Omega)}} \right\} = \lambda^{\varepsilon}_{1} \\
     &\leq \underset{{\psi_{1}}\in L^{2}(\Omega)\backslash\{0\}}{\text{sup}}\quad\underset{u\in \text{Span}(\psi_{1})^{\perp}}{\underset{u\in H^{1}_{0}(\Omega)\backslash\{0\}}{\text{inf}}} \left\{ \frac{Q(u)}{\| u\|^{2}_{L^2(\Omega)}} \right\} =\lambda^{\mathbf{D}}_{1},
\end{align*}
where we used the  inclusions 
$H^{1}_{0}(\Omega)\subset H^{1}_{0,\Gamma^{\varepsilon}_{\mathbf{D}}}(\Omega)\subset H^{1}(\Omega)$.
\end{proof}

\subsection{Proof of item {\em (ii)} of Theorem~\ref{One eigenvalue results N holes}}
\label{Proof of Theo 1.1}

We provide in this section the proof of item~\emph{(ii)} of Theorem~\ref{One eigenvalue results N holes}, which relies on estimates on the Dirichlet form evaluated at the quasi-mode~$\varphi^\varepsilon$. The proof is based on various technical estimates, which we state and prove before concluding the section with the actual proof of Theorem~\ref{One eigenvalue results N holes}.

The first two lemmas allow us to localize the zero set of the quasi-mode $\varphi^\varepsilon$ defined in~\eqref{QuasiNhole}. In fact, we write localization results for slightly more general functions
\begin{equation}
  \label{QuasiNhole_alpha}
  \forall x \in \Omega, \qquad {\varphi}_\alpha^{\varepsilon}(x)=-\frac{1}{\sqrt{\pi}}-\frac{1}{\alpha\sqrt{\pi}}\sum_{k=1}^{N}K^{(k)}_{\varepsilon}\left(\text{log}\hspace{0.1cm}|x-x^{(k)}|+\frac{1-|x|^{2}}{4}\right),
\end{equation}
for some scaling parameter~$\alpha \in (0,+\infty)$. Note that~$\varphi^\varepsilon_1$ coincides with~$\varphi^\varepsilon$.

\begin{lem}
\label{Cercle 1}
Fix~$\alpha\in(0,+\infty)$ and let ${\varphi}_\alpha^{\varepsilon}$ be the function defined in~\eqref{QuasiNhole_alpha}. Then, for
\begin{equation}
  \label{eq:choix_C_alpha_moins}
        {C}_{\alpha,-}> \alpha+\dfrac{1}{8}+\frac{\log(2)}{2},
\end{equation}
and for any~$\eps \in (0,\varepsilon_{0})$, if $|x-x^{(k)}|\leq \mathrm{e}^{-C_{\alpha,-}/K^{(k)}_{\varepsilon}}$ for some $k\in\{1,\dots,N\}$, then~${\varphi}_\alpha^{\varepsilon}(x)> 0$.
\end{lem}

\begin{proof}
Consider~$x \in \Omega$ such that~$|x-x^{(k)}| \leq \mathrm{e}^{-C_{\alpha,-}/K^{(k)}_{\varepsilon}}$  for some $k\in\{1,\dots,N\}$ (and for some ${C}_{\alpha,-}>0$ to be determined). Using that $|x-x^{(k')}|\leq 2$ for all~$k'\in\{1,\dots,N\}$ with~$k'\ne k$, we obtain, for any~$\eps \in (0,\varepsilon_{0})$,
\begin{align*}
     {\varphi}^{\varepsilon}(x)&\geq -\frac{1}{\sqrt{\pi}}+\frac{{C}_{\alpha,-}}{\alpha\sqrt{\pi}}-\left(\sum_{\substack{k'=1 \\ k'\neq k}}^{N}K^{(k')}_{\varepsilon}\right)\frac{\text{log}(2)}{\alpha\sqrt{\pi}}-\left(\sum_{\substack{k=1}}^{N}K^{(k)}_{\varepsilon}\right)\frac{1}{4\alpha\sqrt{\pi}}\\
     &\geq \frac{1}{{\sqrt{\pi}}} \left(-1 +\frac{{C}_{\alpha,-}}{\alpha}-\frac{\log(2)}{\alpha} \overline{K}_\varepsilon-\frac{\overline{K}_\varepsilon}{4\alpha}\right) \\
     &\geq \frac{1}{{\sqrt{\pi}}} \left(-1 + \frac1\alpha\left[C_{\alpha,-}-\frac{\text{log}(2)}{2}-\frac{1}{8}\right]\right)
\end{align*}
where we used the inequality $\overline{K}_\varepsilon<1/2$ implied by~\eqref{eq:barK_hyp}. The right hand side of the last inequality is positive for the choice~\eqref{eq:choix_C_alpha_moins}.
\end{proof}

\begin{lem}\label{cercle 2}
  Fix~$\alpha\in(1/2,+\infty)$ and let ${\varphi}_\alpha^{\varepsilon}$ be the function defined in~\eqref{QuasiNhole_alpha}. Then, for
\begin{equation}
  \label{eq:choix_C_alpha_plus}
        0 < {C}_{\alpha,+} < \alpha - \frac12, 
\end{equation}
and for any $\eps \in (0,\varepsilon_{0})$, if $|x-x^{(k)}|\geq \mathrm{e}^{-C_{\alpha,+}/K^{(k)}_{\varepsilon}}$ for all $k\in\{1,\dots,N\}$, then ${\varphi}^{\varepsilon}_\alpha(x)< 0$.
\end{lem}

\begin{proof}
Consider~$x\in \Omega$ such that $|x-x^{(k)}|\geq \mathrm{e}^{-C_{\alpha,+}/K^{(k)}_{\varepsilon}}$ for all $k\in\{1,\dots,N\}$. Introduce $j(x)\in\{1,\ldots,N\}$ such that
\begin{align*}
  \forall k\in\{1,\dots,N\}, \qquad \left|x-x^{(j(x))}\right|\leq \left|x-x^{(k)}\right|.
\end{align*} 
  Using~\eqref{rho0}, it follows $\rho_0 \le |x^{(k)} - x^{(j(x))}| \le |x^{(k)} - x|+ |x - x^{(j(x))}|\le 2 |x^{(k)} - x|$, so that $|x-x^{(k)}|\ge \rho_0/2$ for all  $k\ne j(x)$. Then, making use of the inequality~$-\log\left(\frac{\rho_0}{2}\right) \ge 0$, we obtain, for $\eps \in (0,\varepsilon_{0})$,
\begin{align*}
   {\varphi}^{\varepsilon}_\alpha(x)&\leq -\frac{1}{\sqrt{\pi}}+\frac{{C}_{\alpha,+}}{\alpha\sqrt{\pi}} - \frac{1}{\alpha\sqrt{\pi}}\sum_{\substack{k'=1 \\ k'\neq k}}^N K_\varepsilon^{(k')} \log\left(\frac{\rho_0}{2}\right) \le \frac{1}{\sqrt{\pi}}\left( -1 + \frac{{C}_{\alpha,+}}{\alpha} - \frac{\overline{K}_\varepsilon}{\alpha} \log\left(\frac{\rho_0}{2}\right) \right).
    \end{align*}
The latter quantity is negative in view of~\eqref{eq:barK_hyp} and~\eqref{eq:choix_C_alpha_plus}.
    \end{proof}
    
In the following, we denote by
\begin{align}\label{r+ and r-}
{r^{(k)}_{\varepsilon,-}}=\mathrm{e}^{-C_{1,-}/K^{(k)}_{\varepsilon}},
\qquad
{r^{(k)}_{\varepsilon,+}}=\mathrm{e}^{-C_{1,+}/K^{(k)}_{\varepsilon}},
\end{align}
the radii which appear respectively in the statements of
Lemmas~\ref{Cercle 1} and \ref{cercle 2} for~$\alpha=1$. In view of
these results, as already stated in the introduction,
$\varphi_\varepsilon$ vanishes in $\Omega$ on $N$ disjoint curves
$(\Gamma^{\varepsilon}_{\mathbf{D}_k})_{k=1\ldots,N}$, and for $k\in
\{1,\dots,N\}$, the curve $\Gamma^{\varepsilon}_{\mathbf{D}_k}$ is
located  in a neighborhood of~$x^{(k)}$ and is contained in an annulus
centered at~$x^{(k)}$, with radii  $r^{(k)}_{\varepsilon,-}$ and
$r^{(k)}_{\varepsilon,+}$ (see also
Lemma~\ref{lem:widetilde_Gamma_D_well_def} below).

The following results on the functions $f_k$ are useful in the proof of item~\emph{(ii)} of Theorem~\ref{One eigenvalue results N holes} below, to get precise estimate on the energy of the quasi-mode $\varphi^\varepsilon$ .

\begin{lem}\label{prop of f_k}
    Let $f_k$ be the function  defined in~\eqref{QuasiNhole}.  The function $f_k$ satisfies
        \begin{align}\label{normal of f_k}
        \partial_n {f}_{k}=0\,\,\,\text{on}\,\,\partial\Omega\setminus \{x^{(k)}\},
    \end{align}
        and \begin{align}\label{laplacian fk}
        \Delta f_{k}=-1 \text{ in } \Omega.
    \end{align}
Moreover, there exists $C>0$ independent of $k$ such that 
    \begin{align}\label{Lp norm f_k}
           \| {f}_{k}\|_{L^{2}(\Omega)}\leq C.
    \end{align}
Finally, there exists $C>0$ independent of $k$ and of~$\varepsilon$ such that, for any $\eps \in (0,\varepsilon_{0})$ and~$r\in(0,1)$, 
    \begin{align}\label{L2 grad norm}
\int_{\Omega \setminus {\mathbf{B}}\left(x^{(k)},\,r\right)}{\vert  \nabla f_k(x) \vert}^{2}\,\mathrm{d}x\leq C\left(1+ \log\left(\frac{2}{r}\right)\right),
    \end{align}
    where for any $x \in \Omega$ and $r>0$, $\Omega \setminus {\mathbf{B}}(x,r)$ denotes the complement of the disk~${\mathbf{B}}(x,r)$ in~$\Omega$.
\end{lem}

\begin{proof}
Let us first prove~\eqref{normal of f_k}. A simple computation shows that
       \begin{align}
       \label{grad f_k}
    \nabla f_{k}(x) = \frac{x-x^{(k)}}{|x-x^{(k)}|^{2}}-\frac{x}{2}.
       \end{align}
       Using that $\vec{n}(x)=\frac{x}{|x|} \,\,\,\text{on}\,\,\partial\Omega$, we obtain that
         \begin{align*}
 \nabla f_{k}(x) \cdot\vec{n}(x) = \frac{x-x^{(k)}}{|x-x^{(k)}|^{2}}\cdot \frac{x}{|x|}-\frac{1}{2}.
       \end{align*}
       Using  polar coordinates, we can write  $x^{(k)}=(\text{cos}(\theta_{k}),\text{sin}(\theta_{k}))$  and $x=(\text{cos}(\theta_x),\text{sin}(\theta_x))$ for $x\in \partial\Omega$ where $\theta_x,\theta_{k}\in[0,2\pi)$ and $\theta_x\ne\theta_{k}$ for~$x \in \partial\Omega \setminus \{x^{(k)}\}$.
   Then,
     \begin{align*}
 \forall x \in \partial\Omega \setminus \{x^{(k)}\}, \qquad \nabla f_{k}(x)\cdot\vec{n}(x)=\frac{1-\text{cos}(\theta_x-\theta_{k})}{{2(1-\text{cos}(\theta_x-\theta_{k}))}}-\frac{1}{2}=0. 
       \end{align*}
       
      Moreover, since 
      $\frac{x-x^{(k)}}{|x-x^{(k)}|^{2}}=\begin{pmatrix}\partial_{x_2} \\ -\partial_{x_1}\end{pmatrix} \arctan\left(\frac{x_2-x^{(k)}_2}{x_1-x^{(k)}_1}\right)$, one has
       \begin{align*}
        \div\left(\frac{x-x^{(k)}}{|x-x^{(k)}|^{2}}\right)=0,
      \end{align*} so that  that
      \begin{align*}
\Delta f_{k}=\div( \nabla f_{k})=-1.
      \end{align*} 

The estimate~\eqref{Lp norm f_k} is immediate from the definition of $f_k$.

Let us finally prove~\eqref{L2 grad norm}. Denote  by ${\mathbf{B}}^{\mathrm{c}}(x,r)=\Omega \setminus {\mathbf{B}}(x,r)$ the complement of the disk~${\mathbf{B}}(x,r)$ in~$\Omega$. 
    From~\eqref{grad f_k}, a Cauchy--Schwarz inequality gives
\begin{align*}
\vert \nabla f_k\vert^{2} &= \frac{1}{\vert x-x^{(k)}\vert^{2}}- \frac{x\cdot(x-x^{(k)})}{\vert x-x^{(k)}\vert^{2}}+\frac{|x|^{2}}{4} \le 2 \left( \frac{1}{\vert x-x^{(k)}\vert^{2}} + \frac{|x|^{2}}{4} \right).
\end{align*}
Then, using that $\vert x\vert \leq 1$ on ${\mathbf{B}}^{\mathrm{c}}\left(x^{(k)},\,r\right)\subset\Omega$,   
\begin{align*}
\int_{{\mathbf{B}}^{\mathrm{c}}\left(x^{(k)},r\right)}{\vert  \nabla f_k \vert}^{2}\,\mathrm{d}x&\leq 2 \left( \int_{{\mathbf{B}}^{\mathrm{c}}\left(x^{(k)},\,r\right)} \frac{1}{\vert x-x^{(k)}\vert^{2}} \,\mathrm{d}x+\frac{\pi}{4}\right).
 \end{align*}
To estimate the first term in the right-hand side, we use polar coordinates centered at $x^{(k)}$ and the fact that  ${\mathbf{B}}^{\mathrm{c}}\left(x^{(k)},r\right)\subset \Omega \subset\mathbf{B}\left(x^{(k)},2\right)$ to get
$$
\int_{{\mathbf{B}}^{\mathrm{c}}\left(x^{(k)},\,r\right)} \frac{1}{\vert x-x^{(k)}\vert^{2}} \,\mathrm{d}x \leq \int_{-\frac{\pi}{2}}^{\frac{\pi}{2}}\int_{r}^{2}\frac{1}{r'}\,\mathrm{d}r'\,\mathrm{d}\theta= \pi \text{log}\left(\frac{2}{r}\right).
$$
This concludes the proof of~\eqref{L2 grad norm} with $C=2\pi$.
    \end{proof}

 We are now in position to conclude the proof of Theorem~\ref{One eigenvalue results N holes}.
\begin{proof}[Proof of item (ii) of Theorem \ref{One eigenvalue results N holes}]
Let us first prove  that 
\begin{align}\label{greater 1}
  \text{dim}\hspace{0.07cm}\text{Ran}\hspace{0.07cm}\pi_{[0,c\hspace{0.01cm}\overline{K}_{\varepsilon}]}(\mathcal{L}^{\varepsilon})\geq 1.
\end{align}
To this end, we use~\eqref{min0} and  show that $\mathcal{L}^{\varepsilon}$ admits at least one eigenvalue of order $\mathrm{O}\left(\overline{K}_{\varepsilon}\right)$ by constructing an appropriate test function. Consider
\[
\frac12 < \alpha < \frac78 - \frac{\log(2)}{2},
\]
so that~\eqref{eq:choix_C_alpha_moins} holds with~$C_{\alpha,-} = 1$ and~\eqref{eq:choix_C_alpha_plus} for some positive constant~$C_{\alpha,+}$. Recall the definition~\eqref{QuasiNhole_alpha} of~$\varphi^\varepsilon_\alpha$, and define~$\Omega_{\alpha,\varepsilon}:=\{x\in\Omega, \, {\varphi}^{\varepsilon}_\alpha(x)<0\}$ and
\begin{align}\label{quasimode disk}
   \forall x\in\Omega, \qquad {\phi}^{\varepsilon}_\alpha(x)= {\varphi}^{\varepsilon}_\alpha(x){\mathds{1}}_{\Omega_{\alpha,\varepsilon}}(x).
\end{align}
By construction,
\[
\Gamma^{\varepsilon}_{\mathbf{D}}  = \partial \Omega \cap \left( \bigcup_{k=1}^N {\mathbf{B}}\left(x^{(k)},\mathrm{e}^{-1/K_\varepsilon^{(k)}}\right) \right) \subset \Omega \setminus \Omega_{\alpha,\varepsilon}.
\]
so that ${\phi}_\alpha^{\varepsilon}$ satisfies Dirichlet boundary conditions on $\Gamma^{\varepsilon}_{\mathbf{D}}$, and therefore ${\phi}_\alpha^{\varepsilon}\in H^{1}_{0,\Gamma^{\varepsilon}_{\mathbf{D}}}(\Omega)$. Using~\eqref{min}, it follows that 
\begin{align}
    \lambda^{\varepsilon}_{0}\leq\frac{Q({\phi}_\alpha^{\varepsilon})}{\| {\phi}_\alpha^{\varepsilon}\|^{2}_{L^2(\Omega)}}.
\end{align}
An upper bound on $\lambda^{\varepsilon}_{0}$ can thus be obtained from an upper bound on~$Q({\phi}^{\varepsilon}_\alpha)$ and a lower bound on~$\| {\phi}^{\varepsilon}_\alpha \|^{2}_{L^2(\Omega)}$. For the numerator, we use a discrete Cauchy--Schwarz inequality to write
\begin{align}\label{upper}
  Q({\phi}^{\varepsilon}_\alpha)&\leq  \frac{N}{\pi \alpha^2}\sum_{k=1}^{N}\left(K^{(k)}_{\varepsilon}\right)^2\| \nabla f_{k}\|_{L^{2}(\Omega_{\alpha,\varepsilon})}^{2} \\
  & \leq \frac{N}{\pi \alpha^2}\sum_{k=1}^{N}\left(K^{(k)}_{\varepsilon}\right)^2\| \nabla f_{k}\|_{L^{2}(\Omega \setminus \mathbf{B}(x^{(k)},\exp(-1/K_\varepsilon^{(k)})))}^{2} \leq C\overline{K}_{\varepsilon},
\end{align}
where the second inequality follows from Lemma~\ref{Cercle 1} and the last one from~\eqref{L2 grad norm}. For the denominator, we use Lemma~\ref{cercle 2} and~\eqref{Lp norm f_k} to write 
\begin{align}
   \Vert {\phi}^{\varepsilon}_\alpha\Vert_{L^{2}(\Omega)}= \Vert {\varphi}^{\varepsilon}_\alpha \Vert_{L^{2}(\Omega_{\alpha,\varepsilon})}&\geq \frac{|\Omega_{\alpha,\varepsilon}|^{1/2}}{\sqrt{\pi}} -\frac{1}{\alpha\sqrt{\pi}}\sum_{k=1}^{N}K^{(k)}_{\varepsilon}\|f_{k}\|_{L^{2}(\Omega_{\alpha,\varepsilon})}\nonumber\\&\geq 1- C\sum_{k=1}^{N}\left(\mathrm{e}^{-C_{\alpha,+}/K_\varepsilon^{(k)}}\right)^2 -C\overline{K}_{\varepsilon},\label{lower}
\end{align}
since
\begin{equation}
  \label{eq:lower_bound_mesure_Omega_alpha}
  |\Omega_{\alpha,\varepsilon}| \geq \pi \left(1- \frac{1}{2} \sum_{k=1}^N \left(\mathrm{e}^{-C_{\alpha,+}/K_\varepsilon^{(k)}}\right)^2 \right).
\end{equation}
For further use, let us notice that, thanks to~\eqref{Lp norm f_k}, it can easily be shown that~$\Vert {\phi}^{\varepsilon}_\alpha\Vert_{L^{2}(\Omega)} \leq 1+C\overline{K}_{\varepsilon}$, so that one actually has, for any~$\alpha > 1/2$,
\begin{equation}
  \label{eq:normphi}
  \Vert {\varphi}^{\varepsilon}_\alpha\Vert_{L^{2}(\Omega_{\alpha,\varepsilon})} = 1 + \mathrm{O}(\overline{K}_{\varepsilon}).
\end{equation}   
Gathering~\eqref{upper} and~\eqref{lower}, one obtains that $\lambda_0^\varepsilon \le c \overline{K}_\varepsilon$, for a constant $c$ independent of $\varepsilon$, which yields~\eqref{greater 1}.

In order to conclude the proof, it only remains to notice that since
the second eigenvalue $\lambda_1^\varepsilon$ of the operator $\mathcal L^\varepsilon$ is
bounded from below by $\lambda_1^{\mathbf{N}}$ (a positive constant
independent of $\varepsilon$), then necessarily, for $\varepsilon$
sufficiently small, $\text{dim}\hspace{0.07cm}\text{Ran}\hspace{0.02cm}\pi_{[0,c\hspace{0.01cm}\overline{K}_\varepsilon]}(\mathcal{L}^{\varepsilon})\leq 1$.

This concludes the proof of item~\emph{(ii)} of Theorem~\ref{One eigenvalue results N holes}.
\end{proof}

\subsection{Numerical illustration of Theorem~\ref{One eigenvalue results N holes}}\label{Numerical illustration of 1.1}
We numerically study here the eigenvalue problem~\eqref{mixed on disk} on $\Omega$, in order to illustrate the theoretical results obtained in this section. We choose $K_\varepsilon = -1/\ln \varepsilon$, so that the holes have a radius~$\varepsilon$. The numerical simulations were performed using FreeFem++~\cite{MR3043640}. We used piecewise linear continuous $P_1$ finite elements. The mesh was produced using the automatic mesh generator of FreeFem++, with about $80$ cells to discretize the exit regions~$\Gamma_{\mathbf{D}_k}$ and $160$ cells to mesh the remaining part of the boundary.

Theorem \ref{One eigenvalue results N holes} shows that the smallest eigenvalue $\lambda^{\varepsilon}_0$ of the problem~\eqref{mixed on disk} is non-degenerate, and that it tends to $0$ as $\varepsilon\rightarrow 0$.  Proposition~\ref{ineq} states that the second eigenvalue $\lambda_1^\varepsilon$ is bounded from below (and above) for all $\varepsilon>0$. Moreover, as discussed after~\eqref{eq:normalization_u_0_eps}, the first eigenfunction~$u^{\varepsilon}_0$ does not vanish in $\Omega$ and can therefore be chosen to be negative and with~$L^2$ norm equal to~1. As can be inferred from the proof of item~\emph{(ii)} of Theorem~\ref{One eigenvalue results N holes}, one then expects $u^{\varepsilon}_0$ to be close to the quasimode $\varphi^\varepsilon$, and thus close to the constant function $-1/\sqrt{\pi}$, as $\varepsilon\rightarrow 0$. 

 \begin{figure}
\hspace{-2.5cm}
\includegraphics[scale=0.25]{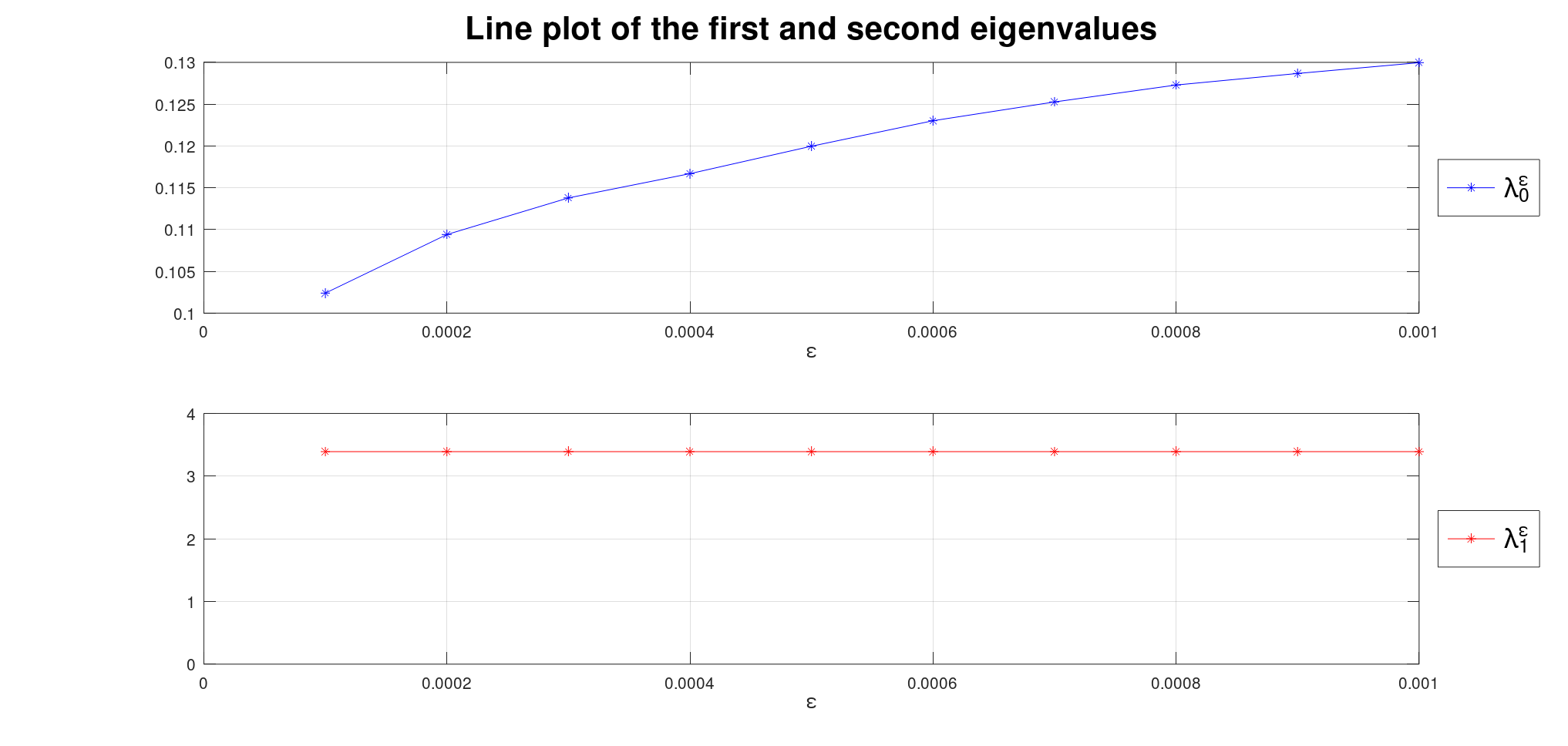}
\caption{\label{fig:variation_eigenvalue} First and second eigenvalues of~$\mathcal{L}^\varepsilon$ as a function of $\varepsilon$ for a single exit point~$x^{(1)}=(1,0)$ and $\mathrm{e}^{-1/K_\varepsilon^{(1)}} = \varepsilon = 0.1$.}
\end{figure}

\begin{figure}
\centering
\includegraphics[scale=0.32]{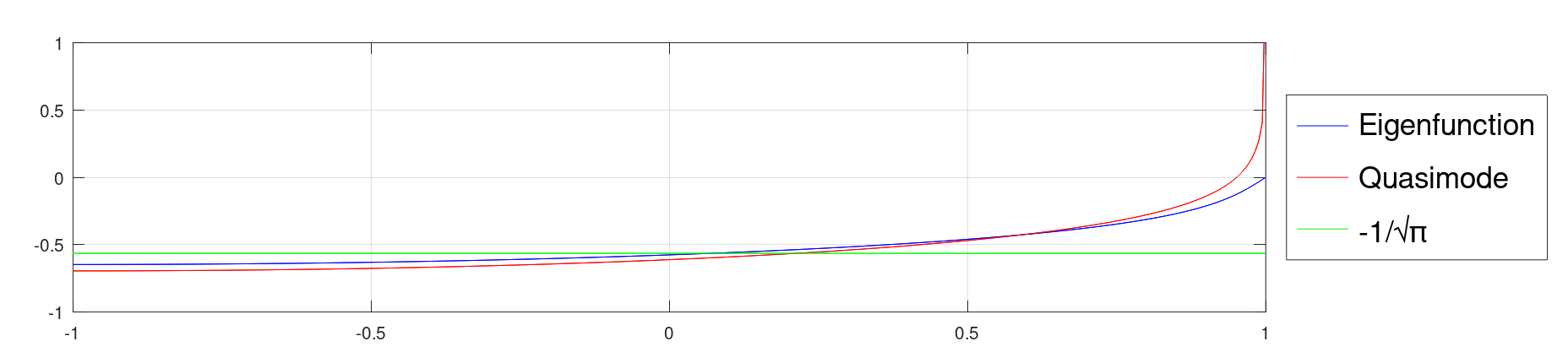}
\caption{\label{fig:variation_eigenfunction} Eigenfunction $u^{\varepsilon}_0$ and quasimode ${\varphi}^{\varepsilon}$ along the horizontal cut $\Omega\cap\{y=0\}$ for a single exit point~$x^{(1)}=(1,0)$ and $\mathrm{e}^{-1/K_\varepsilon^{(1)}} = \varepsilon = 0.1$.}
\end{figure}

\begin{figure}
\hspace{-1.8cm}
 \begin{minipage}{0.41\textwidth}
\centering
\includegraphics[scale=0.5]{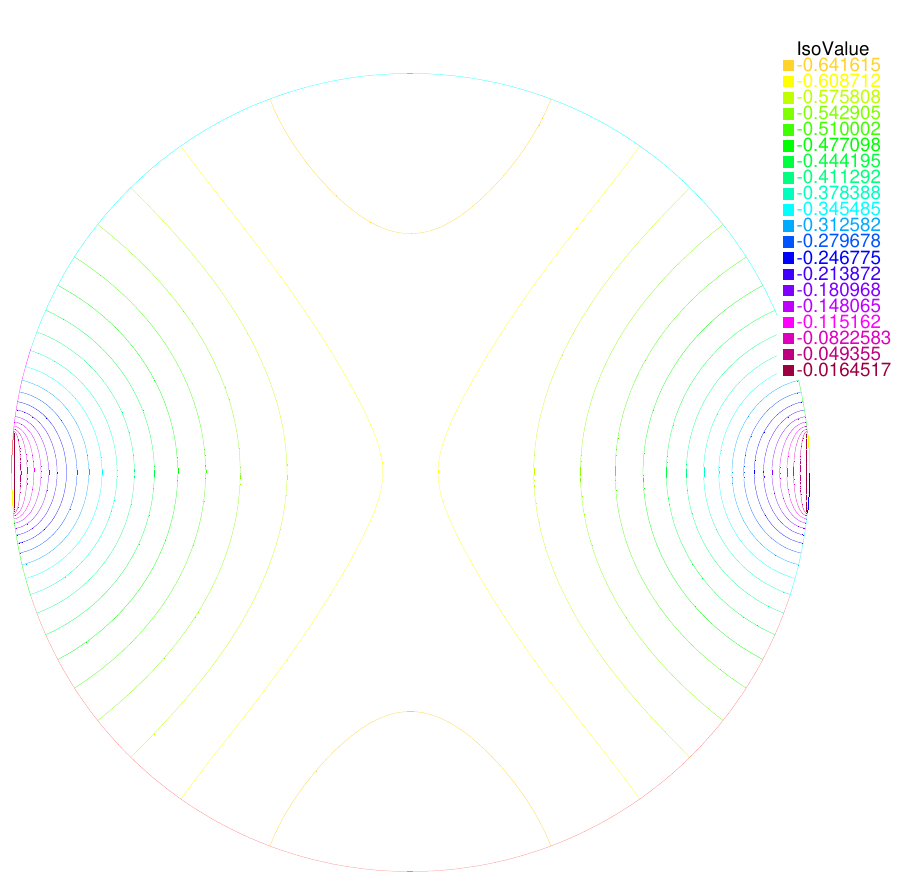}
\end{minipage}\hfill
\begin{minipage}{0.59\textwidth}
\centering
\includegraphics[scale=0.5]{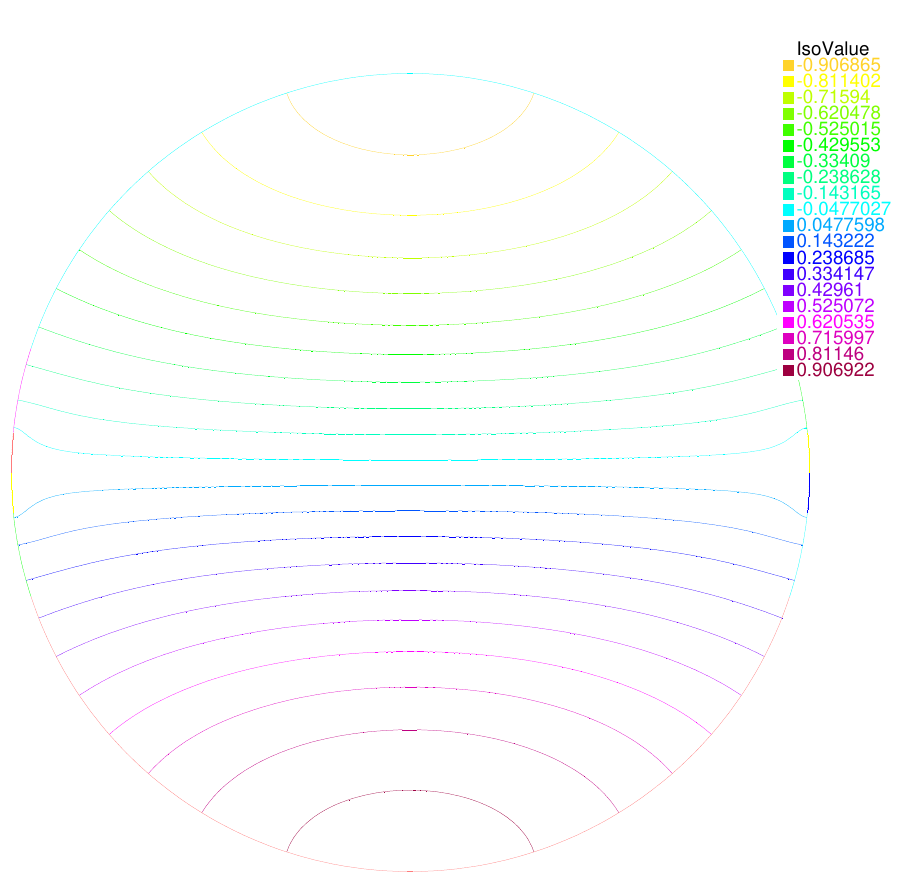}
\end{minipage}
\caption{First two eigenfunctions for holes of radii~$0.1$, with two exit windows. Left: first eigenfunction, associated with $\lambda^{\varepsilon}_0\approx 0.76$. Right: second eigenfunction, associated with  $\lambda^{\varepsilon}_1\approx 3.41$.}\label{image 0 forms}
\end{figure}

We first consider the case of a single absorbing window centered at $x^{(1)}=(1,0)$. As expected, we observe on Figure~\ref{fig:variation_eigenvalue} that the smallest eigenvalue $\lambda^{\varepsilon}_0$ of the problem~\eqref{mixed on disk}  tends to $0$ as $\varepsilon\rightarrow 0$, while the second eigenvalue $\lambda_{1}^{\varepsilon}$ is almost constant. In order to compare the first eigenfunction with the quasimode $\varphi^\varepsilon$, we represent in Figure~\ref{fig:variation_eigenfunction} the first eigenfunction of the operator $\mathcal{L}^{\varepsilon}$ for~$\varepsilon = 0.1$ along the line $\Omega\cap\{y=0\}$, together with the associated quasi-mode defined in~\eqref{QuasiNhole}. 
We observe that the two functions are indeed very close, except around the exit region.

To confirm this observation, we also consider the case of two small disjoint absorbing windows~$\Gamma^{\varepsilon}_{\mathbf{D}_{1}}$ and~$\Gamma^{\varepsilon}_{\mathbf{D}_{2}}$ centered at $x^{(1)}=(1,0)$ and $x^{(2)}=(-1,0)$. We represent in Figure~\ref{image 0 forms} the numerical approximations of the first two eigenfunctions for~$\varepsilon=0.1$. We again observe that the first eigenfunction is close to the constant function~$-1/\sqrt{\pi} \approx -0.564$.


\section{Mean first exit time}
\label{Section law of first exit time}

In this section, we derive precise estimate on the average of the first exit time  $\widetilde{\tau}_\varepsilon$ (recall~\eqref{tildetau}) from the domain~$\widetilde{\Omega}_{\varepsilon}$, where $\widetilde{\Omega}_{\varepsilon}$ is given by~\eqref{The regularity domain}. We start by making precise in Section~\ref{sec:ppties_modified_operator} the operator~$\widetilde{\mathcal L}^{\varepsilon}$, then provide technical results in Section~\ref{sec:lemmas_for_exit_time}, which are used in Section~\ref{sec:proof_thm_exit_times} to prove Theorem~\ref{asymlambda N hole}.

\subsection{Properties of the operator $\widetilde{\mathcal L}^\varepsilon$}
\label{sec:ppties_modified_operator}

From now on, we consider the mixed Dirichlet--Neumann Laplacian
operator $\widetilde{\mathcal L}^\varepsilon$ on the domain
$\widetilde{\Omega}_\varepsilon$ defined by~\eqref{The regularity
  domain}. Let us recall (see~\eqref{domaintilde}) that its domain is given by:
\begin{align}\label{tilde domain}
     \mathcal{D}(\widetilde{\mathcal L}^{\varepsilon})=\left\{u \in  H^{1}(\widetilde\Omega_\varepsilon), \, \, \Delta u \in L^{2}(\widetilde\Omega_\varepsilon), \, \,  \partial_{n}u_{|_{\widetilde\Gamma^{\varepsilon}_{\mathbf{N}}}}=0,\, \,u_{|_{\widetilde\Gamma^{\varepsilon}_{\mathbf{D}}}}=0  \right\},
\end{align} 
where, as explained in the introduction, $\widetilde\Gamma^{\varepsilon}_{\mathbf{D}}=\cup_{k=1}^N \widetilde\Gamma^{\varepsilon}_{\mathbf{D}_{k}}$ with $\widetilde\Gamma^{\varepsilon}_{\mathbf{D}_{k}}$ the connected component of $\Omega \cap (\varphi^\varepsilon)^{-1}\{0\}$ closest to~$x^{(k)}$, and $\widetilde\Gamma^{\varepsilon}_{\mathbf{N}}$ the remaining part of the boundary $\partial \widetilde{\Omega}_\varepsilon$; see~\eqref{tilde partial}. The fact that the definition of~$\widetilde\Gamma^{\varepsilon}_{\mathbf{D}_{k}}$ makes sense is ensured by Lemma~\ref{lem:widetilde_Gamma_D_well_def} below.

Let us emphasize that the spectral results presented in
Section~\ref{Spectral analysis} on ${\mathcal L}^\varepsilon$
(Proposition \ref{def of domain}, Proposition~\ref{ineq}, and Theorem \ref{One eigenvalue results N holes}) also hold true for the modified operator $\widetilde{\mathcal L}^\varepsilon$:
\begin{itemize}
\item[(R1)] $\widetilde{\mathcal L}^{\varepsilon}$ is nonnegative self-adjoint and with compact resolvent;
 \item[(R2)] The second eigenvalue $\widetilde\lambda^{\varepsilon}_{1}$ of the operator~$\widetilde{\mathcal L}^\varepsilon$ satisfies 
   \begin{align}
     \label{eq:l_u_bounds_vp}
     \lambda^{\mathbf{N}}_{1}(\widetilde{\Omega}_\varepsilon)\leq \widetilde\lambda^{\varepsilon}_{1}\leq \lambda^{\mathbf{D}}_{1}(\widetilde{\Omega}_\varepsilon),
   \end{align}
   where $\lambda^{\mathbf{N}}_{1}(\widetilde{\Omega}_\varepsilon)$ (resp. $\lambda^{\mathbf{D}}_{1}(\widetilde{\Omega}_\varepsilon)$) is the  second eigenvalue of the Neumann (resp. Dirichlet) problem of the Laplacian on the domain $\widetilde{\Omega}_\varepsilon$.
 \item[(R3)]  There exist $c > 0$ and $\varepsilon_{0}>0$ such that for any $\varepsilon\in (0,\varepsilon_{0})$,  
$\operatorname{dim} \operatorname{Ran}\hspace{0.07cm}\pi_{[0,c\hspace{0.01cm}\overline{K}_{\varepsilon}]}(\widetilde{\mathcal L}^{\varepsilon})=1$;
\end{itemize}
(R1) and (R2) are directly obtained following the arguments of Section~\ref{properties of L}.
To prove~(R3), it suffices to adapt the proof written in Section~\ref{Proof
  of Theo 1.1} by replacing the domain~$\Omega$ with the modified
domain~$\widetilde{\Omega}_\varepsilon$, and the
quasi-mode~$\varphi_\alpha^\varepsilon$ in~\eqref{QuasiNhole_alpha}
by~$\varphi_1^\varepsilon = \varphi^\varepsilon$, as this function
satisfies by construction the required boundary conditions. Using this
quasimode, one immediately gets that 
$
\text{dim}\hspace{0.07cm}\text{Ran}\hspace{0.07cm}\pi_{[0,c\hspace{0.01cm}\overline{K}_{\varepsilon}]}(\widetilde{\mathcal{L}}^{\varepsilon})\geq
1$. To prove that $
\text{dim}\hspace{0.07cm}\text{Ran}\hspace{0.07cm}\pi_{[0,c\hspace{0.01cm}\overline{K}_{\varepsilon}]}(\widetilde{\mathcal{L}}^{\varepsilon})\leq
1$, in view of~\eqref{eq:l_u_bounds_vp}, it is sufficient to prove
that $\lambda^{\mathbf{N}}_{1}(\widetilde{\Omega}_\varepsilon)$ is
bounded from below by a positive constant, uniformly in~$\varepsilon$: this is stated in the next
lemma. Notice that this is
equivalent to bounding from above by a constant uniform
in~$\varepsilon$ the Poincaré--Wirtinger constant of
the domain~$\widetilde{\Omega}_\varepsilon$. This lemma is thus a direct consequence
of~\cite[Theorem~1.2]{Ruiz2012} (see also
\cite{Boulkhemair-chakib2007}), and of the fact that the family of domains
$(\widetilde\Omega_\varepsilon)_{\varepsilon >0}$ admits a uniform in
$\varepsilon$ interior cone condition in view
of~\eqref{eq:angle_intersection}. 

\begin{lem}\label{lower bound}
  Let $\lambda^{\mathbf{N}}_{1}(\widetilde{\Omega}_\varepsilon)$ be the second eigenvalue of the Laplacian on the domain $\widetilde{\Omega}_\varepsilon$, with Neumann boundary conditions. Then, 
  there exists $\underline \lambda^{\mathbf{N}}_{1}>0$ such that, for any $\varepsilon \in (0,\varepsilon_0)$,
  \begin{align}
    \underline \lambda^{\mathbf{N}}_{1}\leq \lambda^{\mathbf{N}}_{1}(\widetilde{\Omega}_\varepsilon).
  \end{align}
\end{lem}
A consequence of~\eqref{eq:l_u_bounds_vp} and Lemma~\ref{lower bound}
is that $\widetilde{\lambda}^\varepsilon_1$ admits $\underline
\lambda^{\mathbf{N}}_{1}$ as a lower bound for all~$\varepsilon \in (0,\varepsilon_0)$, and this concludes the proof
of (R3).

We conclude this section by a result ensuring that the Dirichlet
region $\widetilde\Gamma^{\varepsilon}_{\mathbf{D}} $ is indeed the
union of $K$ connected disjoint sets
$(\widetilde\Gamma^{\varepsilon}_{\mathbf{D}_{k}})_{k=1,\ldots,K}$,
for $\varepsilon$ sufficiently small. 

\begin{lem}
\label{lem:widetilde_Gamma_D_well_def}
Upon further reducing~$\varepsilon_0>0$ so that, for any~$\varepsilon \in (0,\varepsilon_0)$ and~$k \in \{1, \ldots,K\}$, 
\begin{equation}
  \label{eq:further_reduced_varepsilon_0}
  r^{(k)}_{\varepsilon,+} \le \frac{\rho_0}{2}, \qquad K_\varepsilon^{(k)} - r^{(k)}_{\varepsilon,+} \left(\frac{2}{\rho_0} + \frac12\right) \ge 0, 
\end{equation}
the set $\Omega \cap \mathbf{B}\left(x^{(k)},r^{(k)}_{\varepsilon,+}\right) \cap \left(\varphi^\varepsilon\right)^{-1}[0,+\infty)$ is star shaped with respect to~$x^{(k)}$ for any~$k \in \{1, \ldots,K\}$, and therefore connected. 
\end{lem}

In particular, $\widetilde\Gamma^{\varepsilon}_{\mathbf{D}_{k}}$ can
indeed be defined as the connected component of $\Omega \cap
(\varphi^\varepsilon)^{-1}\{0\}$ closest to~$x^{(k)}$. Note that the
second condition of~\eqref{eq:further_reduced_varepsilon_0} can indeed
be satisfied for $\varepsilon$ sufficiently small since~$r^{(k)}_{\varepsilon,+}= \mathrm{e}^{-C_+/K_\varepsilon^{(k)}}$ goes to zero much faster than $K_\varepsilon^{(k)}$.

\begin{proof}
  Fix $k \in \{1,\dots,K\}$ and $x \in \Omega \cap \mathbf{B}\left(x^{(k)},r^{(k)}_{\varepsilon,+}\right) \cap \left(\varphi^\varepsilon\right)^{-1}[0,+\infty)$. We prove that the segment $\{(1-t)x+tx^{(k)}, \ t\in[0,1]\}$ is included in $\Omega \cap \mathbf{B}\left(x^{(k)},r^{(k)}_{\varepsilon,+}\right) \cap \left(\varphi^\varepsilon\right)^{-1}[0,+\infty)$. This implies that $\Omega \cap \mathbf{B}\left(x^{(k)},r^{(k)}_{\varepsilon,+}\right) \cap \left(\varphi^\varepsilon\right)^{-1}[0,+\infty)$ is star shaped with respect to $x^{(k)}$.
  
Consider the function $[0,1) \ni t \mapsto h(t) = \varphi^\varepsilon((1-t)x+tx^{(k)})$. It suffices to prove that~$h$ is nondecreasing for $t \in [0,1)$. Indeed, if this is the case, then $h(t) \ge h(0)\ge 0$ for any~$t \in [0,1)$, so that $\{(1-t)x+tx^{(k)}, \ t\in[0,1]\}$ is included in $\Omega \cap \mathbf{B}\left(x^{(k)},r^{(k)}_{\varepsilon,+}\right) \cap \left(\varphi^\varepsilon\right)^{-1}[0,+\infty)$.

To prove that~$h$ is nondecreasing, we show that its derivative is nonnegative. In view of~\eqref{QuasiNhole} and~\eqref{grad f_k}, a simple computation shows that
  \[
  \begin{aligned}
  h'(t) & = \frac{K_\varepsilon^{(k)}}{\sqrt{\pi}} \left(x-x^{(k)} \right) \cdot \left[\frac{x-x^{(k)}}{(1-t)|x-x^{(k)}|^2} - \frac{(1-t)x+tx^{(k)}}{2}\right] \\
  & \quad + \left(x-x^{(k)} \right) \cdot \sum_{\ell \neq k} \frac{K_\varepsilon^{(\ell)}}{\sqrt{\pi}} \nabla f_\ell\left((1-t)x+tx^{(k)}\right) \\
  & \geq \frac{K_\varepsilon^{(k)}}{\sqrt{\pi}} \left(\frac{1}{1-t} - \frac12\right) - r^{(k)}_{\varepsilon,+} \sum_{\ell \neq k} \frac{K_\varepsilon^{(\ell)}}{\sqrt{\pi}} \left|\nabla f_\ell\left((1-t)x+tx^{(k)}\right) \right|,
  \end{aligned}
  \]
  where we used $|(1-t)x+t x^{(k)}|\leq 1$ for all $t \in [0,1]$ and~$|x-x^{(k)}| \leq r^{(k)}_{\varepsilon,+} \leq 1$. For $\ell \neq k$, it holds, for any~$t \in [0,1]$,
  \begin{align*}
\left|\nabla f_\ell\left((1-t)x+tx^{(k)}\right) \right|
&\le
\frac{1}{|(1-t)(x-x^{(k)}) + (x^{(k)} -x^{(\ell)})|} + \frac{1}{2} \leq \frac{1}{\rho_0- r^{(k)}_{\varepsilon,+}} + \frac{1}{2},
  \end{align*}
  since $|(1-t)(x-x^{(k)}) + (x^{(k)} -x^{(\ell)})| \geq |x^{(k)} -x^{(\ell)}| - (1-t) |x-x^{(k)}| \geq \rho_0 - (1-t)r^{(k)}_{\varepsilon,+} \geq \rho_0 - r^{(k)}_{\varepsilon,+} >0$ in view of~\eqref{eq:further_reduced_varepsilon_0}.
Thus, using~\eqref{eq:barK_hyp},
 \begin{align*}
  h'(t)&\ge
   \frac{K_\varepsilon^{(k)}}{2\sqrt{\pi}} - r^{(k)}_{\varepsilon,+} \frac{\overline{K}_\varepsilon}{\sqrt{\pi}} \left(\frac{1}{\rho_0- r^{(k)}_{\varepsilon,+}} + \frac{1}{2}\right) \ge \frac{1}{2\sqrt{\pi}} \left[ K_\varepsilon^{(k)} - r^{(k)}_{\varepsilon,+} \left(\frac{1}{\rho_0- r^{(k)}_{\varepsilon,+}} + \frac{1}{2}\right)\right] \geq 0,
  \end{align*}
 where the last inequality follows from the two conditions
 in~\eqref{eq:further_reduced_varepsilon_0}. This allows us to conclude the proof.
\end{proof}

\subsection{Useful technical results}
\label{sec:lemmas_for_exit_time}

Let us now state several lemmas that will be useful to prove Theorem~\ref{asymlambda N hole}.
The following lemma concerns the function ${\varphi}^{\varepsilon}$.

\begin{lem}\label{laplacian of the quasimode N hole}
Let ${\varphi}^{\varepsilon}$ be the function  defined in~\eqref{QuasiNhole}. Then, ${\varphi}^{\varepsilon}$ belongs to~$\mathcal{D}(\widetilde{\mathcal{L}}^{\varepsilon})$ and satisfies 
\begin{align}\label{laplacian of phi}
    \Delta{\varphi}^{\varepsilon}=\frac{\overline{K}_{\varepsilon}}{\sqrt{\pi}} \text{ in } \widetilde{\Omega}_\varepsilon,
\end{align}
where $\overline{K}_{\varepsilon}$ is defined by~\eqref{eq:barK}.
   \end{lem}
   
\begin{proof}
  The function  ${\varphi}^{\varepsilon}$ belongs to $ C^{\infty}(\widetilde{\Omega}_{\varepsilon})$  and satisfies ${\varphi}^{\varepsilon}=0$ on~$\widetilde{\Gamma}^{\varepsilon}_{\mathbf{D}}$ by construction. The equality~\eqref{normal of f_k} implies moreover that~$\partial_n {\varphi}^{\varepsilon}=0$ on~$\widetilde{\Gamma}^{\varepsilon}_{\mathbf{N}}$, so that ${\varphi}^{\varepsilon}\in  \mathcal{D}(\widetilde{\mathcal{L}}^{\varepsilon})$. Finally, \eqref{laplacian of phi} is a consequence of~\eqref{QuasiNhole} and~\eqref{laplacian fk}, which concludes the proof. 
\end{proof}

We can then write the following estimate on the quasi-mode ${\varphi}^{\varepsilon}$ for~$\widetilde{\mathcal{L}}^{\varepsilon}$, giving in particular bounds on the distance to normalized elements in the first eigenspace~${\rm Ran}(\pi_{[0,c\hspace{0.02cm}\overline{K}_{\varepsilon}]}(\widetilde{\mathcal{L}}^{\varepsilon}))$.

\begin{prop}
  \label{estimate}
 There exists $C\in\mathbb{R}_{+}$ such that, for any $\varepsilon \in (0,\varepsilon_{0})$,
 \begin{align}
   \label{eq:varphi_minus_u}
    \left\|{\varphi}^{\varepsilon}-\pi_{[0,c\hspace{0.02cm}\overline{K}_{\varepsilon}]}(\widetilde{\mathcal{L}}^{\varepsilon}){\varphi}^{\varepsilon}\right\|_{L^2(\widetilde{\Omega}_{\varepsilon})}\leq C\overline{K}_{\varepsilon}.
\end{align}
In particular, upon possibly reducing~$\varepsilon_{0}$, it
holds~$\left\|\pi_{[0,c\hspace{0.02cm}\overline{K}_{\varepsilon}]}(\widetilde{\mathcal{L}}^{\varepsilon}){\varphi}^{\varepsilon}\right\|_{L^2(\widetilde{\Omega}_{\varepsilon})}
\geq 1/2$, so that the first eigenfunction
of~$\widetilde{\mathcal{L}}^{\varepsilon}$ (with normalization~\eqref{eq:normalization_widetilde_u_0_eps}) can be obtained from~$\varphi^\varepsilon$ as
\begin{align}\label{the firsteigenfunction}
     \widetilde{u}^{\varepsilon}_{0}=\frac{\pi_{[0,c\hspace{0.02cm}\overline{K}_{\varepsilon}]}(\widetilde{\mathcal{L}}^{\varepsilon}){\varphi}^{\varepsilon}}{\left\| \pi_{[0,c\hspace{0.02cm}\overline{K}_{\varepsilon}]}(\widetilde{\mathcal{L}}^{\varepsilon}){\varphi}^{\varepsilon} \right\|_{L^{2}(\widetilde{\Omega}_{\varepsilon})}}.
\end{align}
Moreover,  
\begin{equation}\label{eq:estimL2}
  \left\|{\varphi}^{\varepsilon}- \widetilde{u}^{\varepsilon}_{0}\right\|_{L^2(\widetilde{\Omega}_{\varepsilon})}\leq C\overline{K}_{\varepsilon}.
\end{equation}
In particular, ${\varphi}^{\varepsilon}$ is a quasi-mode associated with the eigenvalue $\widetilde{\lambda}^{\varepsilon}_{0}$.
\end{prop}

Note that the equality~\eqref{the firsteigenfunction} fixes the sign convention for~$\widetilde{u}^{\varepsilon}_{0}$.

\begin{proof}
  We start by proving~\eqref{eq:varphi_minus_u}.
Note first that $\sigma(\widetilde{\mathcal L}^\varepsilon) \cap
B(0,\underline \lambda^{\mathbf{N}}_{1}/2) =
\{\widetilde{\lambda}_0^\varepsilon\}$ thanks to~(R3), (R2), and the lower bound on the second eigenvalue $\tilde{\lambda}_0^\varepsilon$ of $\widetilde{\mathcal L}^\varepsilon$ obtained in Lemma~\ref{lower bound}. Since ${\varphi}^{\varepsilon}\in  \mathcal{D}(\widetilde{\mathcal{L}}^{\varepsilon})$ (see Lemma \ref{laplacian of the quasimode N hole}), we obtain that
\begin{align*}
      \left(1-\pi_{[0,c\hspace{0.02cm}\overline{K}_{\varepsilon}]}(\widetilde{\mathcal{L}}^{\varepsilon})\right){\varphi}^{\varepsilon} &
      =\frac{1}{2\pi \mathrm{i}}\int_{\mathscr{C}\left(0,\,{ \underline \lambda^{\mathbf{N}}_{1}}/{2}\right)} \left(z^{-1} - (z-\widetilde{\mathcal{L}}^{\varepsilon})^{-1}\right){\varphi}^{\varepsilon}\,\mathrm{d}z \\
      & =-\frac{1}{2\pi \mathrm{i}}\int_{\mathscr{C}\left(0,\,{ \underline \lambda^{\mathbf{N}}_{1}}/{2}\right)}z^{-1}(z-\widetilde{\mathcal{L}}^{\varepsilon})^{-1}\widetilde{\mathcal{L}}^{\varepsilon}{\varphi}^{\varepsilon}\,\mathrm{d}z,
    \end{align*}
    where $\mathscr{C}\left(0,\, { \underline \lambda^{\mathbf{N}}_{1}}/{2}\right)\subset \mathbb{C}$ is the circle of radius ${ \underline \lambda^{\mathbf{N}}_{1}}/{2}$ centered at $0$. Using again the upper bound on~$\tilde{\lambda}_0^\varepsilon$ and the lower bound on~$\tilde{\lambda}_1^\varepsilon$, and classical resolvent estimates,
 we obtain that
    \begin{align*}
         \forall z \in \mathscr{C}\left(0,\, { \underline \lambda^{\mathbf{N}}_{1}}/{2}\right), 
         \qquad 
         \left\|(z- \widetilde{\mathcal{L}}^{\varepsilon})^{-1} \right\|_{\mathcal{B}({L^2(\widetilde{\Omega}_{\varepsilon})})}\leq \frac{C}{{ \underline \lambda^{\mathbf{N}}_{1}}},
    \end{align*} where $C>0$ is independent of $\varepsilon$. 
    Here and in the following,  $\mathcal{B}(L^2(\widetilde{\Omega}_{\varepsilon}))$ is the Banach space of bounded operators from $L^2(\widetilde{\Omega}_{\varepsilon})$ to $L^2(\widetilde{\Omega}_{\varepsilon})$, with~$\left\|\cdot \right\|_{\mathcal{B}({L^2(\widetilde{\Omega}_{\varepsilon})})}$ the associated operator norm. As a consequence, using Lemma \ref{laplacian of the quasimode N hole}, for $\varepsilon$ small enough, we have
    \begin{align*}
        \left\|  \left(1-\pi_{[0,c\hspace{0.02cm}\overline{K}_{\varepsilon}]}(\widetilde{\mathcal{L}}^{\varepsilon})\right){\varphi}^{\varepsilon}\right\|_{L^2(\widetilde{\Omega}_{\varepsilon})}\leq 
         \frac{C}{\underline \lambda^{\mathbf{N}}_{1}}\left\|\widetilde{\mathcal{L}}^{\varepsilon}{\varphi}^{\varepsilon}\right\|_{L^2(\widetilde{\Omega}_{\varepsilon})}\leq C\overline{K}_{\varepsilon},
    \end{align*}
    where $C>0$ is independent of $\varepsilon$, which concludes the
    proof of~\eqref{eq:varphi_minus_u}.

    An immediate consequence of~\eqref{eq:varphi_minus_u} is that
    \begin{equation}
      \label{eq:estimee_pi_varphi}
    \left\|\pi_{[0,c\hspace{0.02cm}\overline{K}_{\varepsilon}]}(\widetilde{\mathcal{L}}^{\varepsilon}){\varphi}^{\varepsilon}\right\|_{L^2(\widetilde{\Omega}_{\varepsilon})}^2 = \left\|{\varphi}^{\varepsilon}\right\|_{L^2(\widetilde{\Omega}_{\varepsilon})}^2 - \left\|{\varphi}^{\varepsilon}-\pi_{[0,c\hspace{0.02cm}\overline{K}_{\varepsilon}]}(\widetilde{\mathcal{L}}^{\varepsilon}){\varphi}^{\varepsilon}\right\|_{L^2(\widetilde{\Omega}_{\varepsilon})}^2 = 1 + \mathrm{O}(\overline{K}_\varepsilon),
    \end{equation}
    where we used the estimate~\eqref{eq:normphi} on~$\| {\varphi}^{\varepsilon}\|_{L^{2}(\widetilde{\Omega}_{\varepsilon})}$ (with~$\alpha=1$).
    
    Let us finally turn to~\eqref{eq:estimL2}. Using the triangle inequality, then~\eqref{eq:varphi_minus_u} and~\eqref{eq:estimee_pi_varphi},
    \[
    \begin{aligned}
      \left\|{\varphi}^{\varepsilon}- \widetilde{u}^{\varepsilon}_{0}\right\|_{L^2(\widetilde{\Omega}_{\varepsilon})} & \leq \left\|{\varphi}^{\varepsilon}-\pi_{[0,c\hspace{0.02cm}\overline{K}_{\varepsilon}]}(\widetilde{\mathcal{L}}^{\varepsilon}){\varphi}^{\varepsilon} \right\|_{L^2(\widetilde{\Omega}_{\varepsilon})} + \left\|\pi_{[0,c\hspace{0.02cm}\overline{K}_{\varepsilon}]}(\widetilde{\mathcal{L}}^{\varepsilon}){\varphi}^{\varepsilon}-\frac{\pi_{[0,c\hspace{0.02cm}\overline{K}_{\varepsilon}]}(\widetilde{\mathcal{L}}^{\varepsilon}){\varphi}^{\varepsilon}}{\| \pi_{[0,c\hspace{0.02cm}\overline{K}_{\varepsilon}]}(\widetilde{\mathcal{L}}^{\varepsilon}){\varphi}^{\varepsilon}\|_{L^{2}(\widetilde{\Omega}_{\varepsilon})}} \right\|_{L^2(\widetilde{\Omega}_{\varepsilon})} \\
      & = \left\|{\varphi}^{\varepsilon}-\pi_{[0,c\hspace{0.02cm}\overline{K}_{\varepsilon}]}(\widetilde{\mathcal{L}}^{\varepsilon}){\varphi}^{\varepsilon} \right\|_{L^2(\widetilde{\Omega}_{\varepsilon})} + 1- \left\|\pi_{[0,c\hspace{0.02cm}\overline{K}_{\varepsilon}]}(\widetilde{\mathcal{L}}^{\varepsilon}){\varphi}^{\varepsilon}\right\|_{L^2(\widetilde{\Omega}_{\varepsilon})} = \mathrm{O}(\overline{K}_\varepsilon),
    \end{aligned}
    \]    
    which allows to conclude the proof of Proposition~\ref{estimate}.
\end{proof}
 
\subsection{Proof of Theorem~\ref{asymlambda N hole}}
\label{sec:proof_thm_exit_times}

We are now in position to prove Theorem \ref{asymlambda N hole}. It is convenient to introduce the normalized function
\begin{align}\label{normalizedquasimode}
    {\psi}^{\varepsilon}=\frac{{\varphi}^{\varepsilon}}{\| {\varphi}^{\varepsilon}\|_{L^{2}(\widetilde{\Omega}_{\varepsilon})}}.
\end{align}
We claim that 
\begin{equation}
  \label{eq:norm_pi_psi}
  \left\|\pi_{[0,c\hspace{0.02cm}\overline{K}_{\varepsilon}]}(\widetilde{\mathcal{L}}^{\varepsilon}){\psi}^{\varepsilon}\right\|_{L^2(\widetilde{\Omega}_{\varepsilon})} = 1 + \mathrm{O}\left( \overline{K}^{2}_{\varepsilon} \right),
\end{equation}
and
\begin{align}
    \label{valueofnorm}
    \left \|\psi^\varepsilon - \pi_{[0,c\hspace{0.02cm}\overline{K}_{\varepsilon}]}(\widetilde{\mathcal{L}}^{\varepsilon}){\psi}^{\varepsilon}\right\|_{L^2(\widetilde{\Omega}_{\varepsilon})} = \mathrm{O}\left( \overline{K}_{\varepsilon} \right),
\end{align}
Indeed, using~\eqref{eq:normphi}, \eqref{eq:varphi_minus_u} and the fact that~$\| {\psi}^{\varepsilon}\|_{L^{2}(\widetilde{\Omega}_{\varepsilon})}= 1$, 
 \begin{align*}
     \left\|\pi_{[0,c\hspace{0.02cm}\overline{K}_{\varepsilon}]}(\widetilde{\mathcal{L}}^{\varepsilon}){\psi}^{\varepsilon}\right\|_{L^2(\widetilde{\Omega}_{\varepsilon})}^2&=\|{\psi}^{\varepsilon}\|_{L^2(\widetilde{\Omega}_{\varepsilon})}^2-\left\|\left(1-\pi_{[0,c\hspace{0.02cm}\overline{K}_{\varepsilon}]}(\widetilde{\mathcal{L}}^{\varepsilon})\right){\psi}^{\varepsilon}\right\|_{L^2(\widetilde{\Omega}_{\varepsilon})}^2\\
     &=1-\frac{\left\|\left(1-\pi_{[0,c\hspace{0.02cm}\overline{K}_{\varepsilon}]}(\widetilde{\mathcal{L}}^{\varepsilon})\right){\varphi}^{\varepsilon}\right\|_{L^2(\widetilde{\Omega}_{\varepsilon})}^2}{\|\varphi^\varepsilon\|_{L^2(\widetilde{\Omega}_{\varepsilon})}^2}=1+\frac{\mathrm{O}\left(\overline{K}_\varepsilon^2\right)}{1+\mathrm{O}(\overline{K}_\varepsilon)},
 \end{align*}
 which leads to~\eqref{eq:norm_pi_psi}. The estimate~\eqref{valueofnorm} then follows from the equality
 \[
 \left \|\psi^\varepsilon - \pi_{[0,c\hspace{0.02cm}\overline{K}_{\varepsilon}]}(\widetilde{\mathcal{L}}^{\varepsilon}){\psi}^{\varepsilon}\right\|_{L^2(\widetilde{\Omega}_{\varepsilon})}^2 = 1 - \left\|\pi_{[0,c\hspace{0.02cm}\overline{K}_{\varepsilon}]}(\widetilde{\mathcal{L}}^{\varepsilon}){\psi}^{\varepsilon}\right\|_{L^2(\widetilde{\Omega}_{\varepsilon})}^2.
 \]

 We can now turn to the estimation of the first eigenvalue. Noting that~\eqref{the firsteigenfunction} holds with~$\varphi^\varepsilon$ replaced by~$\psi^\varepsilon$, and using next~\eqref{eq:norm_pi_psi}, 
    \begin{align}\label{eq1}
    \begin{split}
    \widetilde{\lambda}_{0}^{\varepsilon}= -\left\langle \Delta \widetilde{u}^{\varepsilon}_{0}, \widetilde{u}^{\varepsilon}_{0}\right\rangle_{L^2(\widetilde{\Omega}_{\varepsilon})}=-\left\langle \Delta{\psi}^{\varepsilon}, \pi_{[0,c\hspace{0.02cm}\overline{K}_{\varepsilon}]}(\widetilde{\mathcal{L}}^{\varepsilon}){\psi}^{\varepsilon}\right\rangle_{L^2(\widetilde{\Omega}_{\varepsilon})}\left(1+\mathrm{O}\left(\overline{K}^{2}_{\varepsilon}\right)\right),
         \end{split}
    \end{align}
    since the orthogonal projector $\pi_{[0,c\hspace{0.02cm}\overline{K}_{\varepsilon}]}(\widetilde{\mathcal{L}}^{\varepsilon})$  and $\widetilde{\mathcal{L}}^{\varepsilon}$ 
commute on $\mathcal{D}(\widetilde{\mathcal{L}}^{\varepsilon})$ and~${\psi}^{\varepsilon}\in  \mathcal{D}(\widetilde{\mathcal{L}}^{\varepsilon})$. In addition,
\begin{align*}
    \left\langle \Delta{\psi}^{\varepsilon}, \pi_{[0,c\hspace{0.02cm}\overline{K}_{\varepsilon}]}(\widetilde{\mathcal{L}}^{\varepsilon}){\psi}^{\varepsilon}\right\rangle_{L^2(\widetilde{\Omega}_{\varepsilon})}= \left\langle \Delta{\psi}^{\varepsilon}, {\psi}^{\varepsilon}\right\rangle_{L^2(\widetilde{\Omega}_{\varepsilon})}-\left\langle \Delta{\psi}^{\varepsilon}, \left(1-\pi_{[0,c\hspace{0.02cm}\overline{K}_{\varepsilon}]}(\widetilde{\mathcal{L}}^{\varepsilon})\right){\psi}^{\varepsilon}\right\rangle_{L^2(\widetilde{\Omega}_{\varepsilon})}.
\end{align*}
 For the second term on the right hand side of the previous equality, 
 \begin{align}\label{eq2}
 \begin{split}
     \left|\left\langle \Delta{\psi}^{\varepsilon}, \left(1-\pi_{[0,c\hspace{0.02cm}\overline{K}_{\varepsilon}]}(\widetilde{\mathcal{L}}^{\varepsilon})\right){\psi}^{\varepsilon}\right\rangle_{L^2(\widetilde{\Omega}_{\varepsilon})}\right|&\leq   \left\|\left(1-\pi_{[0,c\hspace{0.02cm}\overline{K}_{\varepsilon}]}(\widetilde{\mathcal{L}}^{\varepsilon})\right){\psi}^{\varepsilon}\right\|_{L^2(\widetilde{\Omega}_{\varepsilon})} \|\Delta{\psi}^{\varepsilon}\|_{L^2(\widetilde{\Omega}_{\varepsilon})}\\&\leq C\overline{K}^{2}_{\varepsilon},
      \end{split}
 \end{align}
 where we used~\eqref{valueofnorm} for the first factor, and Lemma \ref{laplacian of the quasimode N hole} together with the estimate~\eqref{eq:normphi} (with~$\alpha=1$) for the second one.
 For the first term, one has
 \begin{align*}
     \left\langle \Delta{\psi}^{\varepsilon}, {\psi}^{\varepsilon}\right\rangle_{L^2(\widetilde{\Omega}_{\varepsilon})}=\left\langle \Delta{\varphi}^{\varepsilon}, {\varphi}^{\varepsilon}\right\rangle_{L^2(\widetilde{\Omega}_{\varepsilon})}\left(1+\mathrm{O}\left(\overline{K}_{\varepsilon}\right)\right),
 \end{align*}
 where we used again the estimate~\eqref{eq:normphi} on $\| {\varphi}^{\varepsilon}\|_{L^{2}(\widetilde{\Omega}_{\varepsilon})}$. Moreover, using~\eqref{QuasiNhole} and~\eqref{laplacian of phi},   
 \begin{align*}
     &\left\langle \Delta{\varphi}^{\varepsilon}, {\varphi}^{\varepsilon}\right\rangle_{L^2(\widetilde{\Omega}_{\varepsilon})}\\
     &\qquad =  -\frac{\overline{K}_{\varepsilon}}{\pi}\int_{\widetilde{\Omega}_{\varepsilon}}\mathrm{d}x- \frac{\overline{K}_{\varepsilon}}{\pi}\sum_{k=1}^{N}{K}^{(k)}_{\varepsilon}\int_{\widetilde{\Omega}_{\varepsilon}}\text{log}\hspace{0.1cm}|x-x^{(k)}|\,\mathrm{d}x- \frac{\overline{K}^{2}_{\varepsilon}}{\pi}\int_{\widetilde{\Omega}_{\varepsilon}}\left(\frac{1-|x|^{2}}{4}\right)\,\mathrm{d}x\\
    &\qquad =- \overline{K}_{\varepsilon} +\frac{\overline{K}_{\varepsilon}}{\pi}\left|\Omega\backslash\widetilde{\Omega}_{\varepsilon}\right| - \frac{\overline{K}_{\varepsilon}}{\pi}\sum_{k=1}^{N}{K}^{(k)}_{\varepsilon}\int_{\widetilde{\Omega}_{\varepsilon}}\text{log}\hspace{0.1cm}|x-x^{(k)}|\,\mathrm{d}x- \frac{\overline{K}^{2}_{\varepsilon}}{\pi}\int_{\widetilde{\Omega}_{\varepsilon}}\left(\frac{1-|x|^{2}}{4}\right)\,\mathrm{d}x.
 \end{align*}
 In view of Lemma~\ref{levelset}, it holds
$|\Omega\backslash\widetilde{\Omega}_{\varepsilon}|\leq \pi\displaystyle\sum_{k=1}^{N}\left({r^{(k)}_{\varepsilon,+}}\right)^{2} = \mathrm{O}(\overline{K}_\varepsilon)$.
Concerning the second term,
 \begin{align*}
 \frac{\overline{K}_{\varepsilon}}{\pi} \sum_{k=1}^{N}{K}^{(k)}_{\varepsilon}\int_{\widetilde{\Omega}_{\varepsilon}}\text{log}\hspace{0.1cm}|x-x^{(k)}|\,\mathrm{d}x = \mathrm{O}\left( \overline{K}^{2}_{\varepsilon} \right). 
\end{align*}
Finally, 
\begin{align*}
  \frac{\overline{K}^2_{\varepsilon}}{\pi}\left\vert\int_{\widetilde{\Omega}_{\varepsilon}}\left(\frac{1-|x|^{2}}{4}\right)\,\mathrm{d}x\right\vert \leq \frac{\overline{K}^2_{\varepsilon}}{4}.
\end{align*}
The proof of Theorem~\ref{asymlambda N hole} follows by gathering the above estimates in~\eqref{eq1}.
    
 \section{Law of the first exit point}\label{Section Law of the first exit point}
 
 We  study in this section the law of the first exit point $X_{\widetilde{\tau_{\varepsilon}}}$ from the domain $\widetilde{\Omega}_{\varepsilon}$, where $\widetilde{\Omega}_{\varepsilon}$ and~$\widetilde{\tau_\varepsilon}$ are respectively defined in~\eqref{The regularity domain} and~\eqref{tildetau}. Recall that the law of $X_{\widetilde{\tau_{\varepsilon}}}$ is given by  (see~\eqref{law of tildeX_t}) 
    \begin{align}\label{law of X_tildetau}\frac{\partial_{n}\widetilde{u}_{0}^{\varepsilon}(x) \, \sigma(dx)}{\displaystyle\int_{\partial \widetilde{\Omega}_\varepsilon}\partial_{n}\widetilde{u}_{0}^{\varepsilon}(y)\,\sigma(\mathrm{d}y)},
     \end{align}
     where $\sigma$ here denotes the Lebesgue surface measure on $\partial \widetilde{\Omega}_\varepsilon$. Let us recall that by Lemma~\ref{lem:reg_L2_normal_derivative}, $\partial_{n}\widetilde{u}_{0}^{\varepsilon}$ is well defined as an $L^2$ function on $\partial \widetilde{\Omega}_\varepsilon$.

We now prove several lemmas that lead to the result of Theorem~\ref{exit point}. We begin with two results on the $L^{2}(\widetilde{\Omega}_{\varepsilon})$-normalized principal eigenfunction~$\widetilde{u}^{\varepsilon}_{0}$ of~$\widetilde{\mathcal{L}}^{\varepsilon}$. Let us recall that the sign of this eigenfunction is fixed in order for~\eqref{the firsteigenfunction} to hold.

\begin{lem}\label{integral of normalderivative on GammaD}
  The principal eigenfunction~$\widetilde{u}^{\varepsilon}_{0}$ satisfies, for any~$\varepsilon\in(0,\varepsilon_0)$,
\begin{align}\label{First estimate exit point}
\int_{\widetilde{\Gamma}^{\varepsilon}_{\mathbf{D}}} \partial_n \widetilde{u}^{\varepsilon}_{0}\,\mathrm{d}\sigma = \sqrt{\pi}\ \overline{K}_\varepsilon\left(1+\mathrm{O}\left(\overline{K}_\varepsilon\right)\right).
    \end{align}
\end{lem}

\begin{proof}
Recall that
$\widetilde{\lambda}_{0}^{\varepsilon}=\| \nabla\widetilde{u}^{\varepsilon}_{0}\|_{L^{2}(\widetilde{\Omega}_{\varepsilon})}^{2}=\left\langle \nabla\widetilde{u}^{\varepsilon}_{0},\nabla\widetilde{u}^{\varepsilon}_{0}\right\rangle_{L^2(\widetilde{\Omega}_{\varepsilon})}$. Using~\eqref{the firsteigenfunction}, the fact that
\[
\nabla\pi_{[0,c\overline{K}_{\varepsilon}]}(\widetilde{\mathcal{L}}^{\varepsilon})=\pi_{[0,c\overline{K}_{\varepsilon}]}(\Delta^{(1)})\nabla,
\]
where $\Delta^{(1)}$ here denotes the Laplacian on $1$-forms with mixed tangential-normal boundary conditions (see~\eqref{domain p laplacian} and Proposition \ref{propertiesonpforms}, applied with $\Omega=\widetilde{\Omega}_{\varepsilon}$, $\Gamma^{\varepsilon}_{\mathbf{D}}=\widetilde\Gamma^{\varepsilon}_{\mathbf{D}}$ and $\Gamma^{\varepsilon}_{\mathbf{N}}=\widetilde\Gamma^{\varepsilon}_{\mathbf{N}}$), and~$\pi_{[0,c\overline{K}_{\varepsilon}]}(\Delta^{(1)})\nabla\widetilde{u}^{\varepsilon}_{0}=\nabla\widetilde{u}^{\varepsilon}_{0}$, we get
\begin{align*}
  \widetilde{\lambda}_{0}^{\varepsilon}&=\frac{\left\langle \nabla {\varphi}^{\varepsilon},\nabla\widetilde{u}^{\varepsilon}_{0}\right\rangle_{L^2(\widetilde{\Omega}_{\varepsilon})}}{\left\| \pi_{[0,c\hspace{0.02cm}\overline{K}_{\varepsilon}]}(\widetilde{\mathcal{L}}^{\varepsilon}){\varphi}^{\varepsilon}\right\|_{L^{2}(\widetilde{\Omega}_{\varepsilon})}}. 
\end{align*}
Using~\eqref{eq:estimee_pi_varphi} and Theorem~\ref{asymlambda N hole}, it follows that
\begin{align}
  \left\langle \nabla {\varphi}^{\varepsilon},\nabla\widetilde{u}^{\varepsilon}_{0}\right\rangle_{L^2(\widetilde{\Omega}_{\varepsilon})}=\overline{K}_{\varepsilon}\left(1+\mathrm{O}\left(\overline{K}_{\varepsilon}\right)\right).
\end{align}

Besides, using Green's formula and the regularity result from Lemma~\ref{lem:reg_L2_normal_derivative}, 
\begin{align*}
     \left\langle \nabla {\varphi}^{\varepsilon},\nabla\widetilde{u}^{\varepsilon}_{0}\right\rangle_{L^2(\widetilde{\Omega}_{\varepsilon})}&= \left\langle\nabla  \left(\frac{1}{\sqrt{\pi}}+{\varphi}^{\varepsilon}\right),\nabla\widetilde{u}^{\varepsilon}_{0}\right\rangle_{L^2(\widetilde{\Omega}_{\varepsilon})}\\
     &=-\left\langle  \left(\frac{1}{\sqrt{\pi}}+{\varphi}^{\varepsilon}\right),\Delta\widetilde{u}^{\varepsilon}_{0}\right\rangle_{L^2(\widetilde{\Omega}_{\varepsilon})}+ \int_{\partial\widetilde{\Omega}_{\varepsilon}}\left(\frac{1}{\sqrt{\pi}}+{\varphi}^{\varepsilon}\right)\partial_n \widetilde{u}^{\varepsilon}_{0}.
\end{align*}
Since ${\varphi}^{\varepsilon}=0$ on $\widetilde{\Gamma}^{\varepsilon}_{\mathbf{D}}$ and $\partial_n {\widetilde u}^{\varepsilon}_{0}=0$ on $\widetilde{\Gamma}^{\varepsilon}_{\mathbf{N}}$, the last term in the previous equality is equal to
\begin{align}
   \int_{\partial\widetilde{\Omega}_{\varepsilon}}\left(\frac{1}{\sqrt{\pi}}+{\varphi}^{\varepsilon}\right)\partial_n \widetilde{u}^{\varepsilon}_{0} = \frac{1}{\sqrt{\pi}}\int_{\widetilde{\Gamma}^{\varepsilon}_{\mathbf{D}}}\partial_n \widetilde{u}^{\varepsilon}_{0}.
\end{align}
Moreover, 
\begin{align*}
  \left| \left\langle  \left(\frac{1}{\sqrt{\pi}}+{\varphi}^{\varepsilon}\right),\Delta\widetilde{u}^{\varepsilon}_{0}\right\rangle_{L^2(\widetilde{\Omega}_{\varepsilon})} \right| &\leq \left\| \frac{1}{\sqrt{\pi}}+{\varphi}^{\varepsilon}\right\|_{L^{2}(\widetilde{\Omega}_{\varepsilon})}\| \Delta\widetilde{u}^{\varepsilon}_{0}\|_{L^{2}(\widetilde{\Omega}_{\varepsilon})}\\
    &= \widetilde{\lambda}_{0}^{\varepsilon}\left\| \frac{1}{\sqrt{\pi}}+{\varphi}^{\varepsilon}\right\|_{L^{2}(\widetilde{\Omega}_{\varepsilon})}
    \leq C\overline{K}^{2}_{\varepsilon},
\end{align*}
where we used Theorem~\ref{One eigenvalue results N holes} and
\begin{equation}
\label{eq:bound_varphi_eps_minus_cst}
\left\| \frac{1}{\sqrt{\pi}}+{\varphi}^{\varepsilon}\right\|_{L^{2}(\widetilde{\Omega}_{\varepsilon})}=\mathrm{O}(\overline{K}_\varepsilon)
\end{equation}
since (using~\eqref{Lp norm f_k} for the last inequality)
\begin{align}\label{difference with constant}
  \left\| \frac{1}{\sqrt{\pi}}+{\varphi}^{\varepsilon}\right\|_{L^2(\widetilde{\Omega}_{\varepsilon})}&= \frac{1}{\sqrt{\pi}}\left\|\displaystyle\sum_{k=1}^{N} K^{(k)}_{\varepsilon} f_{k}\right\|_{L^2(\widetilde{\Omega}_{\varepsilon}))} \leq \frac{1}{\sqrt{\pi}}\sum_{k=1}^{N} K^{(k)}_{\varepsilon} \left\|f_{k}\right\|_{L^2(\Omega)} \leq C\overline{K}_\varepsilon.
\end{align}   
This concludes the
proof of~\eqref{First estimate exit point}.
\end{proof}

\begin{lem}\label{exit point Lemma 3}
There exists $C\in\mathbb{R}_{+}$ such that, for any $k\in\{1,\dots,N\}$ and for any  $\varepsilon \in (0,\varepsilon_{0})$,
\begin{align}\label{second estimate exit point}
\left\vert\int_{\widetilde{\Gamma}^{\varepsilon}_{\mathbf{D}_{k}}}\partial_n \left(\widetilde{u}_{0}^{\varepsilon}-{\varphi}^{\varepsilon}\right)\mathrm{d}\sigma \right\vert\leq C \overline{K}^{3/2}_\varepsilon,
    \end{align}
    where $\widetilde{u}_{0}^{\varepsilon}$ is the  $L^{2}(\widetilde{\Omega}_{\varepsilon})$-normalized principal  eigenfunction  of $\widetilde{\mathcal{L}}^{\varepsilon}$ (with the appropriate sign convention such that~\eqref{the firsteigenfunction} holds) and ${\varphi}^{\varepsilon}$ is the quasi-mode defined in~\eqref{QuasiNhole}.
\end{lem}

\begin{proof}
Let us introduce ${v}^{\varepsilon}:=\widetilde{u}_{0}^{\varepsilon}-{\varphi}^{\varepsilon}$. We start by writing preliminary estimates on~$\Delta v^\varepsilon$ and~$\nabla v^\varepsilon$, and then relate with a Green formula the integral on~$\widetilde{\Gamma}^{\varepsilon}_{\mathbf{D}_{k}}$ in~\eqref{second estimate exit point} to integrals over~$\widetilde{\Omega}_{\varepsilon}$ involving~$\Delta v^\varepsilon$ and~$\nabla v^\varepsilon$.

We first claim that
\begin{align}\label{Laplacian of difference}
       \|\Delta{v}^{\varepsilon}\|_{L^{2}(\widetilde{\Omega}_{\varepsilon})}=\mathrm{O}\left(\overline{K}^{2}_\varepsilon\right).
    \end{align}
Indeed, in view of~\eqref{mixed} and Lemma~\ref{laplacian of the quasimode N hole}, 
   \begin{align*}
      \Delta{v}^{\varepsilon}&= -\widetilde{\lambda}_{0}^{\varepsilon}\widetilde{u}_{0}^{\varepsilon}- \frac{\overline{K}_\varepsilon}{\sqrt{\pi}} =\left(-\widetilde{\lambda}_{0}^{\varepsilon}+\overline{K}_\varepsilon\right)\widetilde{u}_{0}^{\varepsilon}+\overline{K}_\varepsilon\left({\varphi}^{\varepsilon}-\widetilde{u}_{0}^{\varepsilon}\right)-\overline{K}_\varepsilon\left(\frac{1}{\sqrt{\pi}}+{\varphi}^{\varepsilon}\right).
   \end{align*}
   Therefore, Theorem~\ref{asymlambda N hole},~\eqref{eq:estimL2}, and~\eqref{eq:bound_varphi_eps_minus_cst} imply that~\eqref{Laplacian of difference} holds.

We next prove that
    \begin{align}\label{gradient of the difference}
       \|\nabla{v}^{\varepsilon}\|_{L^{2}(\widetilde{\Omega}_{\varepsilon})}=\mathrm{O}\left(\overline{K}^{3/2}_\varepsilon\right).
    \end{align}
Since~${v}^{\varepsilon}=0$ on $\widetilde{\Gamma}^{\varepsilon}_{\mathbf{D}}$ and $\partial_n {v}^{\varepsilon}=0$ on $\widetilde{\Gamma}^{\varepsilon}_{\mathbf{N}}$, by integration by parts, 
\[
\int_{\widetilde{\Omega}_{\varepsilon}}|\nabla{v}^{\varepsilon}|^{2} = \int_{\widetilde{\Omega}_{\varepsilon}} \left(\Delta{v}^{\varepsilon}\right) {v}^{\varepsilon}.
\]
Now, using  a Cauchy--Schwarz inequality,~\eqref{eq:estimL2}, and~\eqref{Laplacian of difference}, we  obtain that, for $\varepsilon$ small enough,
enough,
\begin{align*}
\left |\int_{\widetilde{\Omega}_{\varepsilon}} \left(\Delta{v}^{\varepsilon}\right) {v}^{\varepsilon} \right|\leq   \|\Delta {v}^{\varepsilon}\|_{L^{2}(\widetilde{\Omega}_{\varepsilon})}\|{v}^{\varepsilon}\|_{L^{2}(\widetilde{\Omega}_{\varepsilon})}\leq C\overline{K}^{3}_\varepsilon,
\end{align*}
which yields~\eqref{gradient of the difference}. 

We now fix $k\in\{1,\dots,N\}$ and relate the integral on~$\widetilde{\Gamma}^{\varepsilon}_{\mathbf{D}_{k}}$ in~\eqref{second estimate exit point} to integrals over~$\widetilde{\Omega}_{\varepsilon}$ involving~$\Delta v^\varepsilon$ and~$\nabla v^\varepsilon$. We introduce to this end a smooth cut-off function~$\chi^{\ell}_{k}:{\Omega}\rightarrow[0,1]$, defined using some parameter~$0<\ell<1$, which satisfies
\[
\chi^{\ell}_{k}(x)= 1 \textrm{ on } {\mathbf{B}} \left(x^{(k)},\frac{\ell}{2}\right)\cap\Omega,
\qquad
\chi^{\ell}_{k}(x)= 0 \textrm{ on } \Omega\backslash{\mathbf{B}} \left(x^{(k)},{\ell}\right).
\]
The parameter~$\ell \in (0,\rho_0)$ is chosen such that, for any~$\varepsilon \in(0,\varepsilon_{0})$,
   \begin{align*}
\widetilde{\Gamma}^{\varepsilon}_{\mathbf{D}_{k}}\subset{\mathbf{B}} \left(x^{(k)},\frac{\ell}{2}\right)\cap\Omega, 
\qquad
\widetilde{\Gamma}^{\varepsilon}_{\mathbf{D}_{k'}} \cap \mathbf{B} \left(x^{(k)},\ell\right) = \emptyset \textrm{ for } k' \neq k.
   \end{align*}
  Since ${v}^{\varepsilon}\in  \mathcal{D}(\widetilde{\mathcal{L}}^{\varepsilon})$ (because $\widetilde{u}_{0}^{\varepsilon}$ and ${\varphi}^{\varepsilon}$ belong to~$\mathcal{D}(\widetilde{\mathcal{L}}^{\varepsilon})$, see Lemma \ref{laplacian of the quasimode N hole}), Lemma~\ref{lem:reg_L2_normal_derivative} ensures that~$\partial_{n}{v}^{\varepsilon}\in L^{2}(\partial\widetilde{\Omega}_{\varepsilon})$. Then, using an integration by parts, we have
   \begin{align}\label{integration formula 2}
    \int_{\widetilde{\Omega}_{\varepsilon}}\Delta{v}^{\varepsilon}  \chi^{\ell}_{k} = -  \int_{\widetilde{\Omega}_{\varepsilon}}\nabla{v}^{\varepsilon}  \cdot \nabla\chi^{\ell}_{k} + \int_{\widetilde{\Gamma}^{\varepsilon}_{\mathbf{D}_{k}}} \partial_{n}{v}^{\varepsilon}\chi^{\ell}_{k}\,\mathrm{d}\sigma,
   \end{align}
   where we used that  $\partial_{n}{v}^{\varepsilon}=0$ on $\widetilde{\Gamma}^{\varepsilon}_{\mathbf{N}}$ and $\chi^{\ell}_{k}=0$ on $\widetilde{\Gamma}^{\varepsilon}_{\mathbf{D}_{k'}}$ for $k'\ne k$.
   Using a Cauchy--Schwarz inequality and~\eqref{Laplacian of difference} we obtain that,  for $\varepsilon$ small enough,
\begin{align}\label{estimate 2}
\left |\int_{\widetilde{\Omega}_{\varepsilon}}\Delta{v}^{\varepsilon}  \chi^{\ell}_{k} \right|\leq   \|\Delta{v}^{\varepsilon}  \|_{L^{2}(\widetilde{\Omega}_{\varepsilon})}\left\|\chi^{\ell}_{k}\right\|_{L^{2}(\widetilde{\Omega}_{\varepsilon})}\leq C\overline{K}^{2}_\varepsilon.
\end{align}
Let us next deal with the first  term on the right hand side
of~\eqref{integration formula 2}.  A Cauchy--Schwarz inequality and~\eqref{gradient of the difference} imply that,  for $\varepsilon$ small
enough,
\begin{align}\label{estimate 3}
\left |\int_{\widetilde{\Omega}_{\varepsilon}}\nabla{v}^{\varepsilon} \cdot \nabla\chi^{\ell}_{k}\,\mathrm{d}x\right|\leq   \|\nabla{v}^{\varepsilon}  \|_{L^{2}(\widetilde{\Omega}_{\varepsilon})}\left\|\nabla\chi^{\ell}_{k}  \right\|_{L^{2}(\widetilde{\Omega}_{\varepsilon})}\leq C\overline{K}^{3/2}_\varepsilon.
\end{align}
Finally,  using~\eqref{integration formula 2}, \eqref{estimate 2} and~\eqref{estimate 3}, we obtain that, for $\varepsilon$ small enough,
\begin{align*}\int_{\widetilde{\Gamma}^{\varepsilon}_{\mathbf{D}_{k}}} \partial_{n}{v}^{\varepsilon}\chi^{\ell}_{k}\,\mathrm{d}\sigma=\int_{\widetilde{\Gamma}^{\varepsilon}_{\mathbf{D}_{k}}} \partial_{n}{v}^{\varepsilon}\,\mathrm{d}\sigma=\mathrm{O}\left(\overline{K}^{3/2}_\varepsilon\right),
\end{align*}
where  we used that $\chi^{\ell}_{k}=1$ on $\widetilde{\Gamma}^{\varepsilon}_{\mathbf{D}_{k}}$. This concludes the
proof of~\eqref{second estimate exit point}.
\end{proof}

The next two lemmas provide explicit computations on the exit flux for the quasimode $\varphi^\varepsilon$.

\begin{lem}\label{Lemma exit point 1}
Let $k\in\{1,\dots,N\}$. Then, for any~$\varepsilon\in(0,\varepsilon_0)$, 
\begin{align}\label{third estimate exit point}
\int_{\widetilde{\Gamma}^{\varepsilon}_{\mathbf{D}_{k}}}\partial_n {\varphi}^{\varepsilon}\,\mathrm{d}\sigma=-\int_{\Omega \cap \mathscr{C}\left(x^{(k)},r^{(k)}_{\varepsilon,+}\right)}\partial_n {\varphi}^{\varepsilon}\,\mathrm{d}\sigma+\mathrm{O}\left(\overline{K}^{2}_\varepsilon\right),
\end{align}
where ${\varphi}^{\varepsilon}$ is the quasi-mode defined in~\eqref{QuasiNhole}, the normal vector on the left-hand side of~\eqref{third estimate exit point} is the unit normal vector on~$\widetilde{\Gamma}^{\varepsilon}_{\mathbf{D}_{k}}$ which points outward $\widetilde{\Omega}_\varepsilon$, and the normal vector on the right-hand side corresponds to the outward unit normal to $\mathbf{B}\left(x^{(k)}, r^{(k)}_{\varepsilon,+}\right)$.
\end{lem}

\begin{proof}
  Let $k\in\{1,\dots,N\}$.  Let us introduce (see Figure~\ref{figure for S})
  \begin{align}\label{Domaine Stokes}
    \mathcal{S}^{\varepsilon}:=\widetilde{\Omega}_{\varepsilon}\cap {\mathbf{B}} \left(x^{(k)},r^{(k)}_{\varepsilon,+}\right).
  \end{align}
  \begin{figure}[H]
    \begin{center}
\begin{tikzpicture}[scale=1.8]
  \draw[red] (2.6,-0.05) arc (88:264:0.5);
 \draw[blue](2.51,0.4) arc (90:263:0.95);
 \draw [<-] (1.55, -0.50)--(0.5, 0.3);
 \draw (0.4, 0.5) node[scale=0.9] {$\Omega \cap \mathscr{C}\left(x^{(k)},r^{(k)}_{\varepsilon,+}\right)$};
   \draw[black] (2.5,0.4) arc (15:6:3);
    \draw[black] (2.52,-1.05) arc (-10:-19:3);
         \draw (2.85,-0.6) node[scale=0.9] {$x^{(k)}$};
         \draw (1.8, -0.50) node[scale=0.9] {$\mathcal{S}^{\varepsilon}$};
   \draw (2.59, -0.56) node[scale=0.48] {$\bullet$};
\draw [<-] (2.1, -0.50)--(3.5, 0.4);
\draw (3.7, 0.5) node[scale=0.9] {$\widetilde{\Gamma}^{\varepsilon}_{\mathbf{D}_{k}}$};
\draw [<-] (2.45, -1.3)--(3, -1.5);
\draw (3.5, -1.5) node[scale=0.9] {$\widetilde{\Gamma}^{\varepsilon}_{\mathbf{N}} \subset \partial\Omega$};
\end{tikzpicture}
\caption{The domain $\mathcal{S}^{\varepsilon}$   in the neighborhood of $x^{(k)}$.}\label{figure for S}
\end{center}
\end{figure}
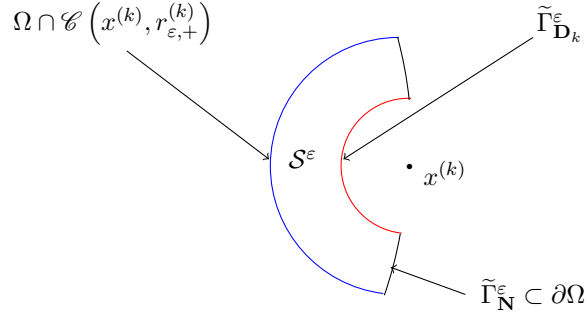
     Using  Stokes’ formula on~$\mathcal{S}^{\varepsilon}$ (which is licit since ${\varphi}^{\varepsilon}\in C^{\infty}(\widetilde{\Omega}_{\varepsilon})$),
          \begin{align} \label{integration formula 3}
          \begin{split}
\int_{\mathcal{S}^{\varepsilon}}\Delta{\varphi}^{\varepsilon}&= \int_{\partial\mathcal{S}^{\varepsilon}}\partial_n {\varphi}^{\varepsilon}\,\mathrm{d}\sigma = \int_{\widetilde{\Gamma}^{\varepsilon}_{\mathbf{D}_{k}}}\partial_n {\varphi}^{\varepsilon}\,\mathrm{d}\sigma+\int_{ \Omega \cap \mathscr{C}\left(x^{(k)},r^{(k)}_{\varepsilon,+}\right) } \partial_n {\varphi}^{\varepsilon}\,\mathrm{d}\sigma,
\end{split}
      \end{align}
   where we used that   $\partial_n {\varphi}^{\varepsilon}=0$ on $\widetilde{\Gamma}^{\varepsilon}_{\mathbf{N}}$. Moreover, using Lemma \ref{laplacian of the quasimode N hole},
   \begin{align}
       \left|\int_{\mathcal{S}^{\varepsilon}}\Delta{\varphi}^{\varepsilon}\right| =  \frac{\overline{K}_{\varepsilon}}{\sqrt{\pi}} \left|\mathcal{S}^{\varepsilon}\right| \leq \frac{\overline{K}_{\varepsilon}}{\sqrt{\pi}} \left|{\mathbf{B}} \left(x^{(k)},r^{(k)}_{\varepsilon,+}\right)\right| = \overline{K}_{\varepsilon}\sqrt{\pi}\left(r^{(k)}_{\varepsilon,+}\right)^2 = \mathrm{O}\left(\overline{K}^{2}_\varepsilon\right),
   \end{align}
which allows to conclude the proof of~\eqref{third estimate exit point}.
\end{proof}

The first term on the right hand side of~\eqref{third estimate exit point} is estimated in the following lemma.

\begin{lem}\label{Lemma exit point 2}
Let $k\in\{1,\dots,N\}$. Then, for any~$\varepsilon\in(0,\varepsilon_0)$, 
\begin{align}\label{fourth estimate exit point}
\int_{ \Omega \cap \mathscr{C}\left(x^{(k)},r^{(k)}_{\varepsilon,+}\right)}\partial_n {\varphi}^{\varepsilon}\,\mathrm{d}\sigma=-\sqrt{\pi}{K}^{(k)}_\varepsilon+\mathrm{O}\left(\overline{K}^{2}_\varepsilon\right),
    \end{align}
    where ${\varphi}^{\varepsilon}$ is the quasi-mode defined in~\eqref{QuasiNhole}. 
\end{lem}

\begin{proof}
Let us introduce the short-hand notation $C^{(k)}_\varepsilon = \Omega \cap \mathscr{C}\left(x^{(k)},r^{(k)}_{\varepsilon,+}\right)$. From the definition~\eqref{QuasiNhole}, we have 
\begin{align}\label{estimate 7}
\begin{split}
   \int_{C^{(k)}_\varepsilon}\partial_n {\varphi}^{\varepsilon}\,\mathrm{d}\sigma =  -\frac{K^{(k)}_{\varepsilon}}{\sqrt{\pi}}\int_{C^{(k)}_\varepsilon}\partial_n f_{k}\,\mathrm{d}\sigma -\frac{1}{\sqrt{\pi}}\sum_{\substack{k'=1 \\ k'\neq k}}^{N} K^{(k')}_{\varepsilon}\int_{C^{(k)}_\varepsilon}\partial_n f_{k'}\,\mathrm{d}\sigma.
   \end{split}
\end{align}
In view of~\eqref{grad f_k}, the first integral on the right hand side of~\eqref{estimate 7} reads 
\begin{equation}\label{estimate 10}
\begin{aligned}
  -\frac{K^{(k)}_{\varepsilon}}{\sqrt{\pi}}\int_{C^{(k)}_\varepsilon}\partial_n
  f_{k}
  \,\mathrm{d}\sigma&=-\frac{K^{(k)}_{\varepsilon}}{\sqrt{\pi}}\int_{C^{(k)}_\varepsilon}\frac{x-x^{(k)}}{|x-x^{(k)}|^{2}}\cdot\vec{n}(x)\,\sigma(\mathrm{d}x)\\
& \quad +\frac{K^{(k)}_{\varepsilon}}{\sqrt{\pi}}\int_{C^{(k)}_\varepsilon}\frac{x}{2}\cdot\vec{n}(x)\,\sigma(\mathrm{d}x).
 \end{aligned}
\end{equation}
As~$\vec{n}(x)=\frac{x-x^{(k)}}{|x-x^{(k)}|}$ on~$\mathscr{C}\left(x^{(k)},r^{(k)}_{\varepsilon,+}\right)$ (see Figure \ref{figure for S}), 
\begin{align*}
   \int_{C^{(k)}_\varepsilon}\frac{x-x^{(k)}}{|x-x^{(k)}|^{2}}\cdot\vec{n}(x)\,\sigma(\mathrm{d}x) = \int_{C^{(k)}_\varepsilon}\frac{1}{|x-x^{(k)}|}\,\sigma(\mathrm{d}x)=\pi-\theta_{\varepsilon},
\end{align*}
where $\theta_{\varepsilon}= \mathrm{arcsin}(r^{(k)}_{\varepsilon,+}/2)$. 
Concerning the second integral on the right hand side of~\eqref{estimate 10}, using that $|x|\leq 1$ for $x \in C^{(k)}_\varepsilon$, we obtain  
\begin{align*}
    \left|\int_{C^{(k)}_\varepsilon}\frac{x}{2}\cdot\vec{n}(x)\,\sigma(\mathrm{d}x)\right|\leq \frac12 \left|C^{(k)}_\varepsilon\right| = \frac12 r^{(k)}_{\varepsilon,+}(\pi-\theta_{\varepsilon}).
\end{align*}
Then, for $\varepsilon$ small enough,
\begin{align}\label{estimate 6}
   -\frac{K^{(k)}_{\varepsilon}}{\sqrt{\pi}}\int_{C^{(k)}_\varepsilon}\partial_n f_{k}\,\mathrm{d}\sigma=-{K^{(k)}_{\varepsilon}}{\sqrt{\pi}}+\mathrm{O}\left(\overline{K}^{2}_{\varepsilon}\right).
\end{align}

Let us now deal with  the second integral on the right hand side of~\eqref{estimate 7}. Note that, in view of~\eqref{rho0}, there exists $\rho>0$ independent of~$\varepsilon$ such that, for $\varepsilon$ small enough,  
\begin{align}\label{Distance 1}
    \forall x\in C^{(k)}_\varepsilon, \quad \forall k'\ne k, \qquad \left|x-x^{(k')}\right|\geq \rho.
\end{align}
Using again~\eqref{grad f_k},
\begin{align*}
   &\int_{C^{(k)}_\varepsilon}\partial_n f_{k'}\,\mathrm{d}\sigma = \int_{C^{(k)}_\varepsilon}\frac{x-x^{(k')}}{|x-x^{(k')}|^{2}}\cdot\frac{x-x^{(k)}}{|x-x^{(k)}|}\,\sigma(\mathrm{d}x) -\int_{C^{(k)}_\varepsilon}\frac{x}{2}\cdot\frac{x-x^{(k)}}{|x-x^{(k)}|}\,\sigma(\mathrm{d}x).
\end{align*}
With~\eqref{Distance 1} and since $|x|\leq 1$ on $C^{(k)}_\varepsilon$, we obtain that
\[
\left|\int_{C^{(k)}_\varepsilon}\partial_n f_{k'}\,\mathrm{d}\sigma\right|\leq
\left|C^{(k)}_\varepsilon\right| \left(\frac1\rho + \frac12 \right).
\]
Therefore, 
\begin{align}\label{estimate 5}
& \left| -\frac{1}{\sqrt{\pi}}\sum_{\substack{k'=1 \\ k'\neq k}}^{N} K^{(k')}_{\varepsilon}\int_{C^{(k)}_\varepsilon}\partial_n f_{k'}\,\mathrm{d}\sigma \right|\leq \frac{1}{\sqrt{\pi}}\left(\frac1\rho + \frac12 \right) \left(\sum_{\substack{k'=1 \\ k'\neq k}}^{N} K^{(k')}_{\varepsilon}\right)r^{(k)}_{\varepsilon,+}(\pi-\theta_{\varepsilon})=\mathrm{O}\left(\overline{K}^{2}_{\varepsilon}\right).
\end{align}
Gathering the latter inequality with~\eqref{estimate 7} and~\eqref{estimate 6} finally gives the desired result. 
\end{proof}

We are now in position to prove Theorem \ref{exit point}.

\begin{proof}[Proof of Theorem \ref{exit point}]
Note first that the estimate~\eqref{normal derivatice on Gamma D} is given by Lemma~\ref{integral of normalderivative on GammaD}. Fix next~$k\in\{1,\dots,N\}$. Using Lemmas \ref{Lemma exit point 1} and \ref{Lemma exit point 2}, we have, for $\varepsilon$ small enough,
\begin{align}\label{estimate 8}
\int_{\widetilde{\Gamma}^{\varepsilon}_{\mathbf{D}_{k}}}\partial_n {\varphi}^{\varepsilon}\,\mathrm{d}\sigma=\sqrt{\pi}{K}^{(k)}_\varepsilon+\mathrm{O}\left(\overline{K}^{2}_\varepsilon\right).
\end{align}
Lemma \ref{exit point Lemma 3} and~\eqref{estimate 8} then imply that, for $\varepsilon$ small enough,
\begin{align}\label{estimate 9}
\begin{split}
\int_{\widetilde{\Gamma}^{\varepsilon}_{\mathbf{D}_{k}}}\partial_n \widetilde{u}_{0}^{\varepsilon}\,\mathrm{d}\sigma&= \int_{\widetilde{\Gamma}^{\varepsilon}_{\mathbf{D}_{k}}}\partial_n {\varphi}^{\varepsilon}\,\mathrm{d}\sigma+\int_{\widetilde{\Gamma}^{\varepsilon}_{\mathbf{D}_{k}}} \partial_n(\widetilde{u}_{0}^{\varepsilon} - {\varphi}^{\varepsilon})\,\mathrm{d}\sigma\\
&=\sqrt{\pi}{K}^{(k)}_\varepsilon+\mathrm{O}\left(\overline{K}^{3/2}_\varepsilon\right),
\end{split}
\end{align}
which is~\eqref{normal derivatice on Gamma Dk}. Moreover, using~\eqref{law of X_tildetau}, \eqref{normal derivatice on Gamma D}, and \eqref{normal derivatice on Gamma Dk}, one obtains that, for $\varepsilon$ small enough, 
\begin{align*} \mathbb{P}_{\widetilde{\nu}_{0}^{\varepsilon}}\left[X_{\widetilde{\tau_{\varepsilon}}}\in\widetilde\Gamma^{\varepsilon}_{\mathbf{D}_{k}}\right]=\frac{\displaystyle\int_{\widetilde{\Gamma}^{\varepsilon}_{\mathbf{D}_{k}}}\partial_n\widetilde{u}_{0}^{\varepsilon}\,\mathrm{d}\sigma}{\displaystyle\int_{\widetilde{\Gamma}^{\varepsilon}_{\mathbf{D}}}\partial_n \widetilde{u}_{0}^{\varepsilon}\,\mathrm{d}\sigma}=\frac{{K}^{(k)}_\varepsilon}{\overline{K}_\varepsilon}+\mathrm{O}\left(\sqrt{\overline{K}_\varepsilon}\right),
\end{align*}
which concludes the proof of Theorem \ref{exit point}.
\end{proof}

\appendix
\section{The mixed Laplacian on \texorpdfstring{$p$}{}-forms}\label{Laplacian p forms}

In this appendix, we provide additional results on the spectrum of the mixed Laplacian on $p$-forms, to gain intuition on the problem. Notice that the results of Theorem~\ref{Oneholepropertie}(iii) are not used in the main body of this article, but provide a natural extension of Theorem~\ref{One eigenvalue results N holes}. A spectral analysis of the mixed Laplacian on $p$-forms is provided in Section~\ref{Spectral analysis of the p forms}, and the results are illustrated by numerical simulations in Section~\ref{Numerical illustration of Theo A.3}. 

\subsection{Spectral analysis of the Laplacian on \texorpdfstring{$p$}{}-forms}\label{Spectral analysis of the p forms}

In this section, we provide a  spectral analysis of the  Laplacian on $p$-forms  with mixed tangential-normal boundary conditions. For simplicity and having in mind the numerical illustrations presented in Section~\ref{Numerical illustration of Theo A.3}, we present the results for the Laplacian  on~$\Omega$ with tangential (resp. normal) boundary conditions on $\Gamma^{\varepsilon}_{\mathbf{D}}$ (resp. $\Gamma^{\varepsilon}_{\mathbf{N}}$). Let us emphasize that the  results presented here also hold with $\Omega=\widetilde{\Omega}_{\varepsilon}$, $\Gamma^{\varepsilon}_{\mathbf{D}}=\widetilde \Gamma^{\varepsilon}_{\mathbf{D}}$ and $\Gamma^{\varepsilon}_{\mathbf{N}}=\widetilde \Gamma^{\varepsilon}_{\mathbf{N}}$ as introduced around~\eqref{The regularity domain} (see~\cite{gol2011hodge} for a general setting for these results, on Riemannian Lispchitz manifolds).

We first need to introduce some definitions and notation from Riemannian geometry (we refer to classical textbooks on Riemannian geometry such as~\cite{gallot2004curvature} for more details, see also e.g.~\cite[Appendix]{LelievreNier}).
For $p\in\left\{0,1,2\right\}$, one denotes by $\Lambda^{p} C^{\infty}(\overline{\Omega})$ (respectively $\Lambda^{p} C_{\mathrm{c}}^{\infty}(\Omega)$) the space of $p$-forms which are $C^{\infty}$
 on $\Omega$ (respectively $C^{\infty}$
 with compact support in $\Omega$). Let us moreover introduce the space
\begin{align}
  \Lambda^{p} C_{0}^{\infty}(\Omega)=  \left\{w\in \Lambda^{p} C^{\infty}(\overline{\Omega}) \ \middle | \ \textbf{t}w_{|_{\Gamma^{\varepsilon}_{\mathbf{D}}}}=0\hspace{0.2cm}\text{and}\hspace{0.2cm}\textbf{n}w_{|_{\Gamma^{\varepsilon}_{\mathbf{N}}}}=0\right\},\end{align}
where $\textbf{t}$ denotes the tangential trace and $\textbf{n}$ the normal trace on forms. One denotes by~$\Lambda^{p}H^{m}(\Omega)$ the  Sobolev spaces of $p$-forms with regularity index $m$ on $\Omega$,  the space $\Lambda^{p}H^{0}(\Omega)$ being also denoted by $\Lambda^{p}L^{2}(\Omega)$. We denote by
\[
\Lambda H^{m}(\Omega) = \bigoplus_{p=0}^d \Lambda^{p}H^{m}(\Omega),
\qquad
\Lambda C^{\infty}(\Omega) = \bigoplus_{p=0}^d \Lambda^{p} C^{\infty}(\Omega),
\]
and by $d$ the differential on $\Lambda C^{\infty}(\Omega)$:
\begin{align}
d^{(p)}:\Lambda^{p}C^{\infty}(\Omega)\to \Lambda^{p+1}C^{\infty}(\Omega).
\end{align}
Moreover, $d^{\star}$ denotes  its formal adjoint with respect to the $L^{2}$-scalar product inherited from the Riemannian
structure: 
\begin{align}
d^{(p),\star}:\Lambda^{p+1}C^{\infty}(\Omega)\to \Lambda^{p}C^{\infty}(\Omega).
\end{align}
The following proposition is taken from~\cite[Proposition 4.4]{gol2011hodge}.

\begin{prop}\label{d}
The unbounded operators $d^{(p)}$ and $d^{(p),\star}$   defined on  $\Lambda^{p}L^{2}(\Omega)$  with domains
\begin{align}
 \mathcal{D}\left(d^{(p)}\right)=  \left\{w\in \Lambda^{p} L^{2}({\Omega}) \ \middle| \ d^{(p)}w\in \Lambda^{p+1} L^{2}({\Omega}), \, \mathbf{t}w_{|_{\Gamma^{\varepsilon}_{\mathbf{D}}}}=0\right\}    
\end{align}
and 
\begin{align}
\mathcal{D}\left(d^{(p),\star}\right)=  \left\{w\in \Lambda^{p} L^{2}({\Omega}) \ \middle| \ d^{(p),\star}w\in \Lambda^{p-1} L^{2}({\Omega}), \, \mathbf{{n}}w_{|_{\Gamma^{\varepsilon}_{\mathbf{N}}}}=0\right\}, 
\end{align}
are closed, densely defined, and adjoints of each other in $\Lambda^{p}L^{2}(\Omega)$.
\end{prop}

On can check that (see \cite[Equation (130)]{GesGiaLeliPeutrec})
\begin{align}\label{propriété de diff et coddiff}
 \left\{
    \begin{array}{ll}
 \overline{\text{Im}\hspace{0.1cm}d}\subset \text{Ker}\hspace{0.1cm}d\,\,\,\text{and}\,\,\, d\circ d=0,\\
  \overline{\text{Im}\hspace{0.1cm}d^{\star}}\subset \text{Ker}\hspace{0.1cm}d^{\star}\,\,\,\text{and}\,\,\,d^{\star}\circ d^{\star}=0.
     \end{array}
     \right.
    \end{align}
We are now in position to define the  Laplacians on $p$-forms with mixed tangential-normal boundary conditions on $\partial\Omega$ (see also \cite[p. 89]{GesGiaLeliPeutrec}). The following result is a consequence of~\cite[Theorem 4.5]{gol2011hodge}.

\begin{prop}
Define on $\Lambda^{p}L^{2}(\Omega)$ the Laplacian on $p$-forms
\begin{align}\label{p laplacian}
    \Delta^{(p)}= d^{(p-1)}\circ d^{(p),\star}+d^{(p+1),\star}\circ d^{(p)},
\end{align}
with domain
\begin{align}\label{domain p laplacian}
\begin{split}
   \mathcal{D}\left(\Delta^{(p)}\right)=\biggl\{&w\in \Lambda^{p} L^{2}({\Omega})\ \Big|\,\, dw, \, d^{\star}w,\,  dd^{\star}w,\,  d^{\star}dw \in \Lambda L^{2}({\Omega}),\\ &\qquad \mathbf{t}w_{|_{\Gamma^{\varepsilon}_{\mathbf{D}}}}=0,\,  \mathbf{n}w_{|_{\Gamma^{\varepsilon}_{\mathbf{N}}}}=0,\,  \mathbf{t}d^{\star}w_{|_{\Gamma^{\varepsilon}_{\mathbf{D}}}}=0,\,  \mathbf{n}dw_{|_{\Gamma^{\varepsilon}_{\mathbf{N}}}}=0\biggr\}.
    \end{split}
\end{align}
This operator is nonnegative selfadjoint in $\Lambda^{p}L^{2}(\Omega)$.
In addition, the domain $\mathcal{D}(Q^{(p)})$
of the closed quadratic form $Q^{(p)}$
 associated with
$\Delta^{(p)}$ is given by 
\begin{align}\label{domain form}
\begin{split}
    \mathcal{D}\left(Q^{(p)}\right)&= \mathcal{D}\left(d^{(p)}\right)\cap\mathcal{D}\left(d^{(p),\star}\right)\\
    &=\biggl\{w\in \Lambda^{p} L^{2}({\Omega})\Big|\,\, dw, \, d^{\star}w \in \Lambda L^{2}({\Omega}),\, \mathbf{t}w_{|_{\Gamma^{\varepsilon}_{\mathbf{D}}}}=0,\,  \mathbf{n}w_{|_{\Gamma^{\varepsilon}_{\mathbf{N}}}}=0\biggr\},
    \end{split}
\end{align} 
and, for $w\in\mathcal{D}(Q^{(p)})$,
\begin{align}
    Q^{(p)}(w)= \|dw\|^{2}_{L^2(\Omega)}+\|d^{\star}w\|^{2}_{L^2(\Omega)}.
\end{align}
\end{prop}

In order to make a link with the notation used in the main body of
this article, $\mathcal L^\varepsilon$ is nothing else than the
operator $\Delta^{(0)}$, and likewise, $\widetilde{\mathcal L}^\varepsilon$ is the operator $\Delta^{(0)}$ with~$\Omega=\widetilde{\Omega}_{\varepsilon}$, $\Gamma^{\varepsilon}_{\mathbf{D}}=\widetilde \Gamma^{\varepsilon}_{\mathbf{D}}$ and $\Gamma^{\varepsilon}_{\mathbf{N}}=\widetilde \Gamma^{\varepsilon}_{\mathbf{N}}$. Notice in particular that in this section, for the ease of notation, we do not explicitly indicate the dependence of the operators~$\Delta^{(p)}$ on $\varepsilon$.

\begin{prop}\label{prop:compact}
For $p\in\left\{0,1,2\right\}$, let 
$\Delta^{(p)}$  be the unbounded  operator~\eqref{p laplacian} defined on $\Lambda^{p}L^{2}(\Omega)$ with domain given by~\eqref{domain p laplacian}.
 The operator $\Delta^{(p)}$  is nonnegative self-adjoint and has compact resolvent.
\end{prop}
\begin{proof}
The fact that these operators are nonnegative and self-adjoint are direct consequences of their definitions as Friedrichs extensions of quadratic forms. It remains to prove that these operators have compact resolvent.

    For $p=0$ and $p=2$, this is a consequence of the compact embedding of $H^1(\Omega)$ in $L^2(\Omega)$  and the fact that $\mathcal{D}(\Delta^{(p)})\subset \mathcal{D}(Q^{(p)}) \subset H^1(\Omega)$ (2-forms can be identified with scalar valued functions in dimension 2).
    For $p=1$, it is known that the injection $\mathcal{D}(Q^{(1)}) \subset L^2(\Omega)$ is compact, see~\cite{BPS16,BPS16_bis,jochmann1997compactness}. Therefore, $\mathcal{D}(\Delta^{(1)})\subset L^2(\Omega)$ is compact.
\end{proof}
A consequence of the latter proposition is that the operators $\Delta^{(p)}$ have discrete spectrum, for $p \in \{0,1,2\}$.
The following proposition concerns the commutation property between the spectral projectors of $\Delta^{(p)}$ (for $p\in\left\{0,1,2\right\}$) and the differential  and codifferential operators. These are standard results that we recall since they are used in the main body of the article.

\begin{prop}\label{propertiesonpforms}
Consider the Laplacian on $p$-forms defined in~\eqref{p laplacian} and~\eqref{domain p laplacian}, as well as the form domain $ \mathcal{D}(Q^{(p)})$ defined in~\eqref{domain form}. The differential $d$  and codifferential $d^{\star}$ satisfy the following
commutation property: for all $z\in\mathbb{C} \backslash\sigma(\Delta^{(p)})$  and $w\in \mathcal{D}(Q^{(p)})$,
\begin{align}
  \label{commutation between p laplacian and diff}
  \begin{aligned}
  d\left(z-\Delta^{(p)}\right)^{-1}w & =\left(z-\Delta^{(p+1)}\right)^{-1}dw, \\
  d^{\star}\left(z-\Delta^{(p)}\right)^{-1}w & =\left(z-\Delta^{(p-1)}\right)^{-1}d^{\star}w.
  \end{aligned}
\end{align}
Consequently, for any $t\in\mathbb{R}_{+}$,
\begin{align}\label{commutation between projector and diff}
    d\circ\pi_{[0,\,t]}\left(\Delta^{(p)}\right)=\pi_{[0,\,t]}\left(\Delta^{(p+1)}\right)\circ d,
    \qquad 
    d^{\star}\circ\pi_{[0,\,t]}\left(\Delta^{(p)}\right)=\pi_{[0,\,t]}\left(\Delta^{(p-1)}\right)\circ d^{\star}.
\end{align}
\end{prop}

\begin{proof}
Since $\Lambda^{p} C^{\infty}(\Omega) \cap \mathcal{D}(Q^{(p)})$ is dense in $\mathcal{D}(Q^{(p)})$, it is sufficient to consider the case of smooth forms~$w\in \Lambda^{p} C^{\infty}(\Omega)$. For $z\in\mathbb{C} \backslash\sigma(\Delta^{(p)})$, let us introduce 
\begin{align}\label{resolvent equality}
    u=\left(z-\Delta^{(p)}\right)^{-1}w.
\end{align}
Due to the ellipticity  of the associated boundary problem, \eqref{resolvent equality} implies that~$u$ belongs to~$\Lambda^{p} C^{\infty}(\Omega)$. Using the fact that $d$ and $d^{\star}$  commute with $\Delta^{(p)}$, we obtain the following relations as differential operators on~$\Omega$:
\begin{align}\label{identity1}
     d\left(z-\Delta^{(p)}\right)u=\left(z-\Delta^{(p+1)}\right)du=dw,
\end{align}
and 
\begin{align}
    d^{\star}\left(z-\Delta^{(p)}\right)u=\left(z-\Delta^{(p+1)}\right)d^{\star}u=d^{\star}w.
\end{align}
Since $u$ belongs to $\mathcal{D}(\Delta^{(p)})$, it satisfies \begin{align}\label{cond u}
\textbf{t}u_{|_{\Gamma^{\varepsilon}_{\mathbf{D}}}}=0,\quad \textbf{n}u_{|_{\Gamma^{\varepsilon}_{\mathbf{N}}}}=0,\quad  \textbf{t}d^{\star}u_{|_{\Gamma^{\varepsilon}_{\mathbf{D}}}}=0,\quad \textbf{n}du_{|_{\Gamma^{\varepsilon}_{\mathbf{N}}}}=0.
\end{align}
Let us check that $du\in\mathcal{D}(\Delta^{(p+1)})$. Using \cite[Appendix]{LelievreNier}, we have that $\textbf{t}$ commutes with the
differential. Using~\eqref{cond u}, we obtain that $\textbf{t}du=d\textbf{t}u=0$ on $\Gamma^{\varepsilon}_{\mathbf{D}}$. We have that $\textbf{t}d^{\star}(du)=\textbf{t}\Delta^{(p)}u-\textbf{t}dd^{\star}u=z\textbf{t}u-\textbf{t}w-d\textbf{t}d^{\star}u=0$ on $\Gamma^{\varepsilon}_{\mathbf{D}}$, where we used~\eqref{p laplacian} and~\eqref{cond u}. We get directly $\textbf{n}du=0$ on $\Gamma^{\varepsilon}_{\mathbf{N}}$ from~\eqref{cond u} and $\textbf{n}d(du)=0$ from the fact that $d\circ d=0$. 
Hence $du\in\mathcal{D}(\Delta^{(p+1)})$  and the identity~\eqref{identity1} yield
\begin{align*}
     d\left(z-\Delta^{(p)}\right)^{-1}w=du=\left(z-\Delta^{(p+1)}\right)^{-1}dw,
\end{align*}
which proves the first  commutation relation.  For the second one, it is sufficient
to show that $d^{\star}u\in\mathcal{D}(\Delta^{(p-1)})$.  Using \cite[Appendix]{LelievreNier}, we have that $\textbf{n}$ commutes with the
codifferential. With~\eqref{cond u}, we obtain that $\textbf{n}d^{\star}u=d^{\star}\textbf{n}u=0$ on $\Gamma^{\varepsilon}_{\mathbf{N}}$.  We have that $\textbf{n}d(d^{\star}u)=\textbf{n}\Delta^{(p)}u-\textbf{n}d^{\star}du=z\textbf{n}u-\textbf{n}w-d^{\star}\textbf{n}du=0$ on $\Gamma^{\varepsilon}_{\mathbf{N}}$, where we used~\eqref{p laplacian} and~\eqref{cond u}. We  directly get $\textbf{t}d^{\star}u=0$ on $\Gamma^{\varepsilon}_{\mathbf{D}}$ from~\eqref{cond u} and $\textbf{t}d^{\star}(d^{\star}u)=0$ from the fact that $d^{\star}\circ d^{\star}=0$.  Finally, \eqref{commutation between projector and diff} follows directly from~\eqref{commutation between p laplacian and diff}, by integrating over contours around the discrete eigenvalues in $[0,t]$.
\end{proof}
Let us conclude this section with the following theorem which makes precise the number of small eigenvalues of $\Delta^{(p)}$ for $p \in \{1,2\}$. It is thus a generalization of Theorem~\ref{One eigenvalue results N holes} which already yields $\operatorname{dim} \operatorname{Ran}\hspace{0.07cm}\pi_{[0,c\hspace{0.01cm}\overline{K}_\varepsilon]}\left(\Delta^{(0)}\right)
   =1$.

\begin{theo}\label{Oneholepropertie}
For $p\in\left\{1,2\right\}$, let 
$\Delta^{(p)}$  be the unbounded  operator defined on $\Lambda^{p}L^{2}(\Omega)$ with domain given by~\eqref{domain p laplacian}.
Then,
\begin{enumerate}
    \item [(i)] For any eigenvalue $\lambda$ of $\Delta^{(p)}$ and associated eigenform $w\in\mathcal{D}(\Delta^{(p)})$, it holds $dw\in\mathcal{D}(\Delta^{(p+1)})$ and $d^{\star}w\in\mathcal{D}(\Delta^{(p-1)})$,
with
\begin{align}\label{eigenvalue}
     d\Delta^{(p)}w=\Delta^{(p+1)}dw=\lambda \,dw,
     \end{align}
     and
     \begin{align}
    d^{\star}\Delta^{(p)}w=\Delta^{(p-1)}d^{\star}w=\lambda\, d^{\star}w.
\end{align}
 \item [(ii)] There exist $c > 0$ and $\overline{\varepsilon}>0$ such that, for any $\varepsilon\in (0,\overline{\varepsilon})$, 
   \[ \operatorname{dim} \operatorname{Ran}\hspace{0.07cm}\pi_{[0,c\hspace{0.01cm}\overline{K}_\varepsilon]}\left(\Delta^{(p)}\right)
   = \left\{ \begin{aligned}
     N & \quad \text{if} \,\,p=1,\\
     0 & \quad  \text{if}\,\,p=2.
   \end{aligned} \right.
    \]
\end{enumerate}
\end{theo}

\begin{proof}
 Item $(i)$ follows straightforwardly from the characterization of the domain of $\Delta^{(p)}$,~\eqref{propriété de diff et coddiff} and Proposition~\ref{propertiesonpforms}. It therefore suffices to prove item $(ii)$. We first show that~$\operatorname{dim}\text{Ran}\hspace{0.07cm}\pi_{[0,c\hspace{0.01cm}\overline{K}_\varepsilon]}(\Delta^{(2)})=0.$ 
Since we are working in a two-dimensional space, 2-forms can be identified with scalar valued functions (and thus 0-forms), and the eigenvalue problem for the mixed Laplacian~$\Delta^{(2)}$ on $2$-forms can be rewritten as an eigenvalue problem for a mixed Laplacian on $0$-forms, with Neumann boundary conditions on $\Gamma^{\varepsilon}_{\mathbf{D}}$ and Dirichlet boundary conditions on $\Gamma^{\varepsilon}_{\mathbf{N}}$ (this is the so-called Poincaré duality between $p$-forms and $(d-p)$-forms in dimension $d$). As a consequence, since $\Gamma^{\varepsilon}_{\mathbf{D}} \subset \Gamma^{\varepsilon'}_{\mathbf{D}}$ if $\varepsilon < \varepsilon'$, the first eigenvalue of $\Delta^{(2)}$ 
is uniformly (in $\varepsilon$) lower bounded by a positive constant when~$\varepsilon$ goes to~$0$.
This implies that for $\varepsilon$ small enough, $\operatorname{dim}\text{Ran}\hspace{0.07cm}\pi_{[0,c\hspace{0.01cm}\overline{K}_\varepsilon]}(\Delta^{(2)})=0$.

Let us now show that $\operatorname{dim}\text{Ran}\hspace{0.07cm}\pi_{[0,c\hspace{0.01cm}\overline{K}_\varepsilon]}(\Delta^{(1)})=N$. By classical Hodge theory (see~\cite{gol2011hodge}), 
\begin{equation}\label{eq:Hodge}
\text{Ran}\hspace{0.07cm}\pi_{[0,c\hspace{0.01cm}\overline{K}_\varepsilon]}\left(\Delta^{(1)}\right) = d \, \text{Ran}\hspace{0.07cm}\pi_{[0,c\hspace{0.01cm}\overline{K}_\varepsilon]}\left(\Delta^{(0)}\right)\oplus d^\star \, \text{Ran}\hspace{0.07cm}\pi_{[0,c\hspace{0.01cm}\overline{K}_\varepsilon]}\left(\Delta^{(2)}\right) \oplus \mathcal{N}\left(\Delta^{(1)}\right),
\end{equation}
where $\mathcal{N}(\Delta^{(1)})$ denote the space of harmonic $1$-forms with mixed boundary conditions:
\begin{align}
   \mathcal{N}\left(\Delta^{(1)}\right):=\biggl\{w\in \Lambda^{1} L^{2}({\Omega}),\, dw=d^{\star}w=0\,\,\text{in}\,\,\Omega, \textbf{t}w_{|_{\Gamma^{\varepsilon}_{\mathbf{D}}}}=0\,\,\text{and}\,\, \textbf{n}w_{|_{\Gamma^{\varepsilon}_{\mathbf{N}}}}=0\biggr\}.
\end{align}
Let us study the dimensions of the three linear spaces in the right-hand side of~\eqref{eq:Hodge}.

By Theorem \ref{One eigenvalue results N holes}, since $\operatorname{dim}\text{Ran}\hspace{0.07cm}\pi_{[0,c\hspace{0.01cm}\overline{K}_\varepsilon]}(\Delta^{(0)})=1$, there is only one eigenfunction $u \in \mathcal D(\Delta^{(0)})$ associated with the small eigenvalue $\lambda^{\varepsilon}_{0}$ in $[0,c\hspace{0.01cm}\overline{K}_\varepsilon]$ (see \ref{lambdaNhole}). Using~\eqref{eigenvalue}, and since $u$ is not constant, the $1$-form $du$ is an eigenform for the operator~$\Delta^{(1)}$ associated with the eigenvalue~$\lambda^{\varepsilon}_{0}$. In addition, since  $\operatorname{dim}\text{Ran}\hspace{0.07cm}\pi_{[0,c\hspace{0.01cm}\overline{K}_\varepsilon]}(\Delta^{(2)})=0$, $d^\star \text{Ran}\hspace{0.07cm}\pi_{[0,c\hspace{0.01cm}\overline{K}_\varepsilon]}(\Delta^{(2)})=\{0\}$.

Let us now consider the space $\mathcal{N}(\Delta^{(1)})$ of harmonic $1$-forms with mixed boundary conditions.
Using \cite[Theorems~1.1 and~5.3]{gol2011hodge} for example, $\mathcal{N}(\Delta^{(1)})$ is a finite-dimensional space with a dimension equal to the topological Betti number $b_{1}(\overline{\Omega},\Gamma^{\varepsilon}_{\mathbf{D}})$ of $\overline{\Omega}$ relative to $\Gamma^{\varepsilon}_{\mathbf{D}}$. Using  \cite[Example 7.1]{Licht},  since $\Gamma^{\varepsilon}_{\mathbf{D}}$ is the union of $N$  open subsets of $\partial\Omega$ (we have $N$ holes on the boundary), we obtain that the first Betti number $b_{1}(\overline{\Omega},\Gamma^{\varepsilon}_{\mathbf{D}})$  is equal to $N-1$.  As a consequence, the dimension of $\mathcal{N}(\Delta^{(1)})$ is equal to  $N-1$.  
This finally allows to conclude that~$\operatorname{dim}\text{Ran}\hspace{0.07cm}\pi_{[0,c\hspace{0.01cm}\overline{K}_\varepsilon]}(\Delta^{(1)})=N$.
\end{proof}

\subsection{Numerical illustrations of Theorem~\ref{Oneholepropertie}}\label{Numerical illustration of Theo A.3}

The objective of this section is to illustrate the results of Theorem~\ref{Oneholepropertie} on the 2-dimensional disk $\Omega$, with two absorbing windows $\Gamma^{\varepsilon}_{\mathbf{D}_{1}}$ and $\Gamma^{\varepsilon}_{\mathbf{D}_{2}}$ centered at $x^{(1)}=(1,0)$ and $x^{(2)}=(-1,0)$, with radii $\mathrm{e}^{-1/K_\varepsilon^{(1)}}=\mathrm{e}^{-1/K_\varepsilon^{(2)}}=0.1$. We thus compute numerically the spectrum and associated eigenforms of $\Delta^{(p)}$ for $p \in \{1,2\}$ (see Figure~\ref{image 0 forms} for results on the same problem when $p=0$).

Let us first consider $p=1$. We use the fact that $1$-forms can be associated with vector fields, and we thus consider the following eigenvalue problem on a vector field $\mathbf{u}^{\varepsilon}:\Omega \to {\mathbb R}^2$:
\begin{align}\label{def1lapl}
 \left\{
    \hspace{-0.3cm}\begin{array}{ll}
  \hspace{0.2cm}\Delta^{(1)} \mathbf{u}^{\varepsilon}\hspace{-0.3cm}&=\mu^{\varepsilon}\mathbf{u}^{\varepsilon}\,\,\text{in}\,\,\Omega,\\
    \quad\quad\textbf{n}\mathbf{u}^{\varepsilon}\hspace{-0.3cm}&=0\,\,\,\text{on}\,\,\Gamma^{\varepsilon}_{\mathbf{N}},\\
   \quad\quad \textbf{t}\mathbf{u}^{\varepsilon}\hspace{-0.3cm}&=0\,\,\,\text{on}\,\,\Gamma^{\varepsilon}_{\mathbf{D}_{k}} \,\,\,\text{for}\,\,\,k\in \{1, 2\},
     \end{array}
\right.
    \end{align}
    where $\textbf{t}\mathbf{u}^{\varepsilon}$ denotes the tangential component and $\textbf{n}\mathbf{u}^{\varepsilon}$  the normal component of $\mathbf{u}^{\varepsilon}$. Note that since we are working in dimension 2, the operator $\Delta^{(1)}$ writes more explicitly:
    $$
   \Delta^{(1)} \mathbf{u}^{\varepsilon}=\text{curl}\hspace{0.1cm}\text{curl}\,\mathbf{u}^{\varepsilon}-\nabla \hspace{0.1cm}\text{div}\,\mathbf{u}^{\varepsilon}.$$
   In order to obtain reliable results, we 
 use the framework of the so-called Finite Element Exterior Calculus which provides finite element spaces which preserve the geometric and topological structures underlying the equations (see \cite{Arnold,arnold2014periodic} for more details). More precisely, we utilized Raviart--Thomas Orthogonal $RT_0Ortho$ finite elements~\cite{RT77} (a.k.a. Nedelec finite elements~\cite{nedelec1980mixed}). The results were obtained using FreeFem$++$. The mesh was produced using the automatic mesh generator of FreeFem++, with about $40$ cells to discretize the exit regions~$\Gamma_{\mathbf{D}_k}$ and about $80$ cells to mesh the remaining part of the boundary.

As expected in view of item~$(iii)$ of Theorem~\ref{Oneholepropertie} (since here $N=2$), we obtain two eigenvector fields 
$\mathbf{u}_0^\varepsilon$ and $\mathbf{u}_1^\varepsilon$ associated with the two eigenvalues $\mu_0^\varepsilon=0$ and $\mu_1^\varepsilon \approx 0.76$, see Figure~\ref{image 1 forms}. More precisely, in accordance with~\eqref{eq:Hodge}, the eigenvalue $\mu_0^\varepsilon=0$ is associated with an harmonic eigenform  $\mathbf{u}_0^\varepsilon$ and the other eigenvalue $\mu_1^\varepsilon \approx 0.76$ is nothing but the small eigenvalue~$\lambda_0^\varepsilon \approx 0.76$ of  the operator $\Delta^{(0)}={\mathcal L}^\varepsilon$ that we had obtained previously (see Figure~\ref{image 0 forms}), associated with the differential (namely the gradient) of the eigenfunction ${u}_0^{\varepsilon}$. Figure~\ref{div u1} indeed illustrates the fact that ${u}_0^{\varepsilon}$ can be obtained as the codifferential $d^{\star}$ of the eigen-vector field associated with $\lambda_0^\varepsilon$: $- \text{div}\,\mathbf{u}_1^\varepsilon$ is very close to $u_0^\varepsilon$ (compare Figure~\ref{image 0 forms} and  Figure~\ref{div u1}).

\begin{figure}
\hspace{-1.5cm}
\begin{minipage}{0.41\textwidth}
\centering
\includegraphics[scale=0.5]{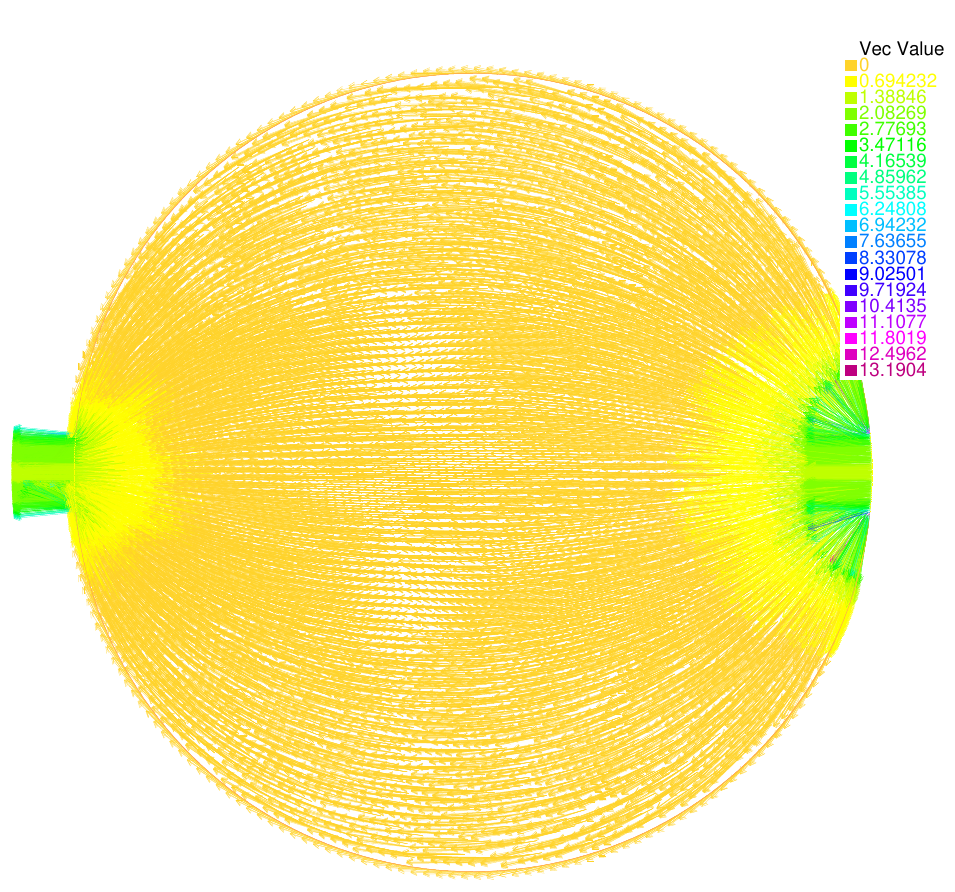}
\end{minipage}\hfill
\begin{minipage}{0.59\textwidth}
\centering
\includegraphics[scale=0.5]{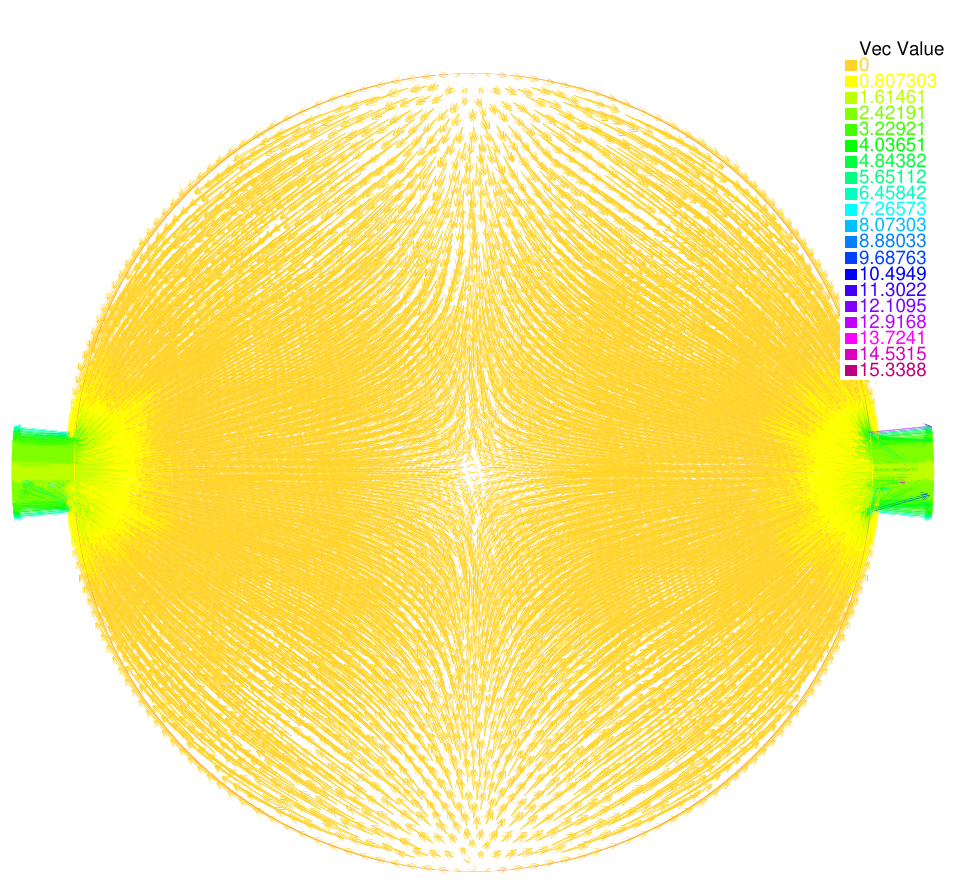}
\end{minipage}
\caption{First two eigen-vector fields of $\Delta^{(1)}$ (with two holes with radii $0.1$). Left: first eigen-vector field associated with $\mu^{\varepsilon}_0=0$. Right: second   eigen-vector field associated with  $\mu^{\varepsilon}_1\approx 0.76$.}\label{image 1 forms}
\end{figure}

\begin{figure}
\centering
\includegraphics[scale=0.5]{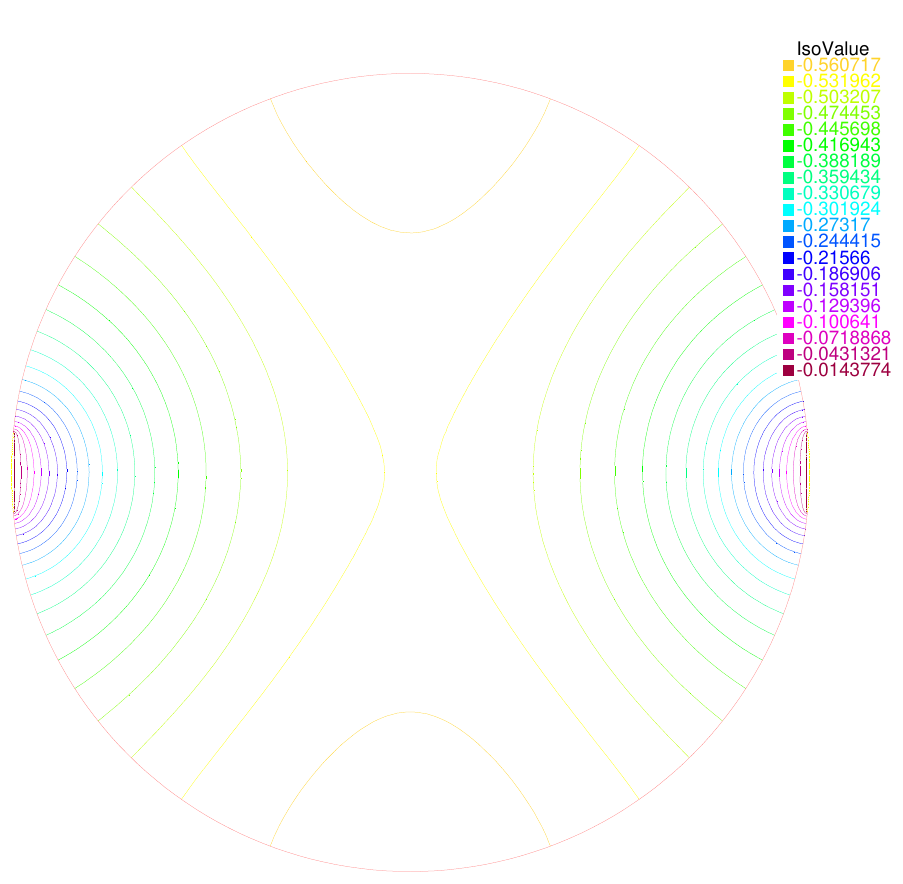}
\caption{Contour plot of $-\text{div}\hspace{0.05cm}{\mathbf{u}_1^\varepsilon}$. This result should be compared with the one on the left plot of Figure~\ref{image 0 forms}.}\label{div u1}
\end{figure}

Let us now consider the operator $\Delta^{(2)}$. As explained in the proof of Theorem~\ref{Oneholepropertie}, the spectrum of this operator can be studied by considering the scalar Laplacian on $\Omega$, with Neumann conditions on the two absorbing windows, and Dirichlet conditions elsewhere. The results were obtained using FreeFem$++$, with $P_1$ finite elements. The mesh was produced using the automatic mesh generator of FreeFem++, with about $200$ cells to discretize the exit regions $\Gamma_{\mathbf{D}_k}$ and about $400$ cells to mesh the remaining part of the boundary.

The first two eigenfunctions of the operator $\Delta^{(2)}$ are represented on Figure~\ref{image NewDiri}: they are associated with eigenvalues which are large compared to the values of the eigenvalues we have obtained on~$\Delta^{(0)}$ and~$\Delta^{(1)}$. This is in accordance with the fact that we do not expect $\Delta^{(2)}$ to have small eigenvalues, see  item $(iii)$ of Theorem~\ref{Oneholepropertie}. These functions are  actually very close to the first two eigenfunctions of the Laplacian with full Dirichlet boundary conditions.

\begin{figure}
\hspace{-1.4cm}
\begin{minipage}{0.41\textwidth}
\centering
\includegraphics[scale=0.5]{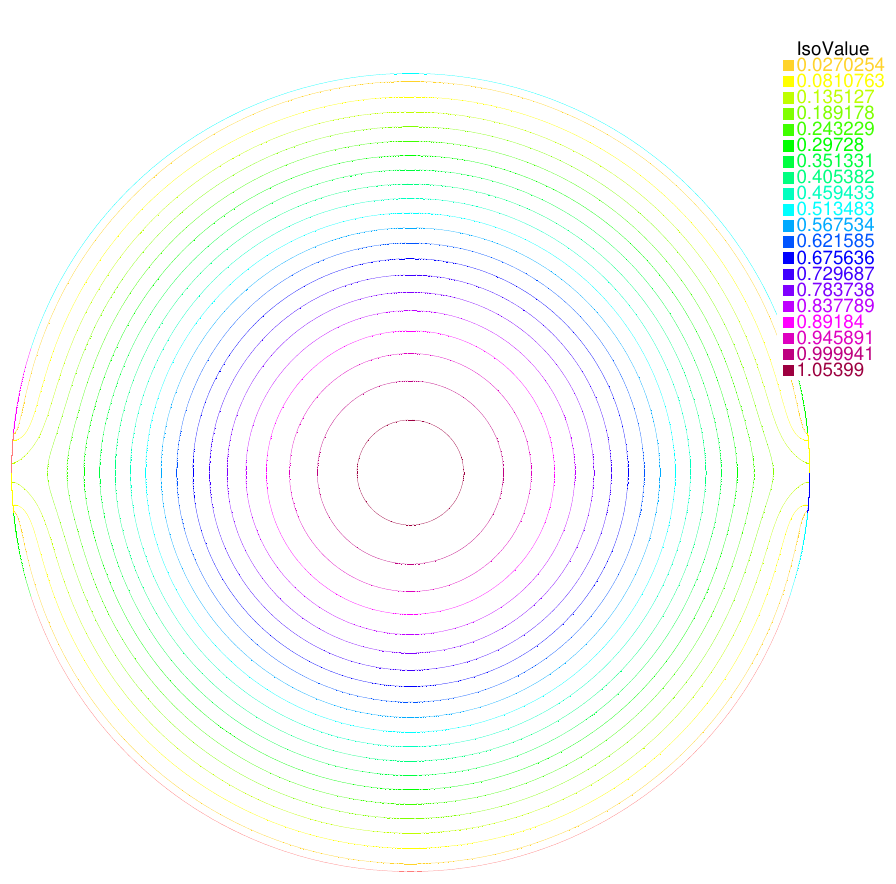}
\end{minipage}\hfill
\begin{minipage}{0.59\textwidth}
\centering
\includegraphics[scale=0.5]{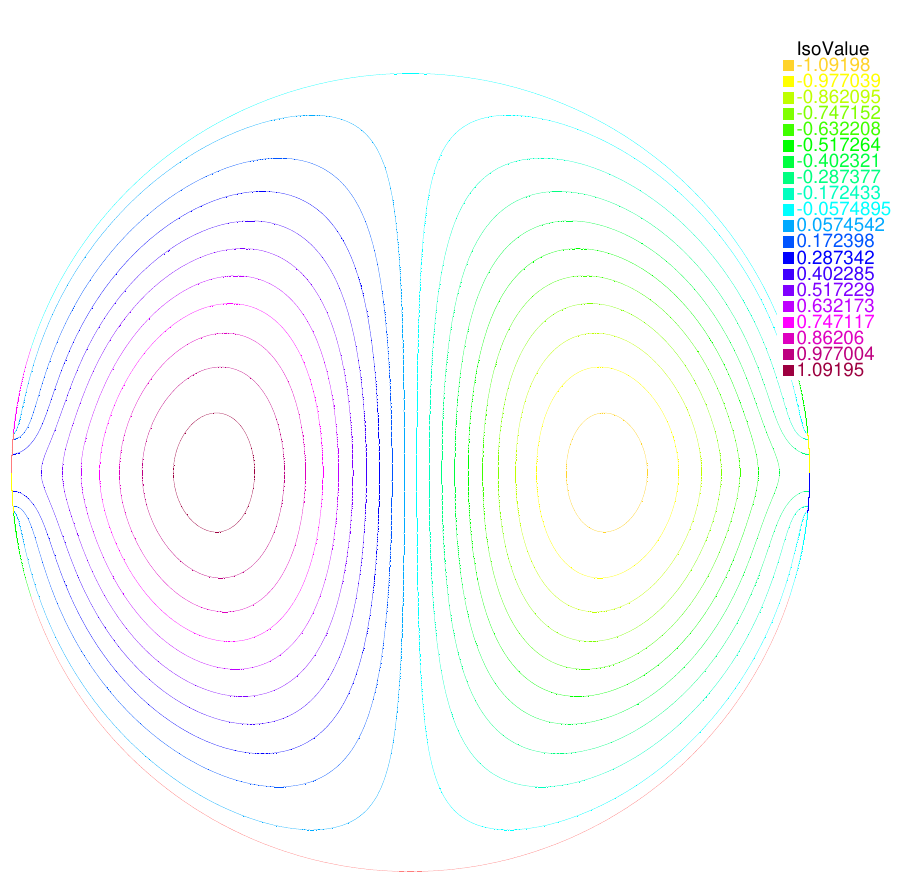}
\end{minipage}
\caption{First two eigen $2$-forms for two holes with radii~$0.1$. Left: first eigen $2$-forms associated with an eigenvalue close to $5.72$. Right: second  eigen $2$-forms associated with an eigenvalue close to $14.4$.}\label{image NewDiri}
\end{figure}

\section{The narrow escape problem on the three-dimensional ball}\label{app:3d}

We illustrate the generality of the approach presented in this work to study the narrow escape problem by quickly outlining how it could be applied to the three-dimensional ball.

Let us consider the three-dimensional unit ball $\Omega \subset \R^3$, with a partition of the boundary into a reflecting part $\Gamma^{\varepsilon}_{\mathbf{N}}$ and an  absorbing part $\Gamma^{\varepsilon}_{\mathbf{D}}$ consisting
of $N$ disjoint small connected regions $\Gamma^{\varepsilon}_{\mathbf{D}_{k}}$, $k\in \{ 1,\dots, N\}$   centered at  $x^{(k)}$. Again, one is interested in the first ($L^2$-normalized) eigenfunction $u^\varepsilon_0$ and the associated eigenvalue $\lambda_0^\varepsilon$ of the Laplacian with mixed Dirichlet (on $\Gamma^{\varepsilon}_{\mathbf{D}}$) and Neumann (on $\Gamma^{\varepsilon}_{\mathbf{N}}$) boundary conditions. Following the same reasoning as for the disk, one expects the following quasimode to be a very good approximation of the first eigenfunction:
       \begin{equation}
       \label{QuasiNhole 3 dimension}
       \begin{split}
   {\varphi}^{\varepsilon} &:= -\frac{1}{\sqrt{|\Omega|}}-\frac{1}{\sqrt{|\Omega|}}\sum_{k=1}^{N}K_\varepsilon^{(k)} f_{k},\\ 
   f_k(x) &=-|\Omega|\left( \frac{1}{2\pi \left|x-x^{(k)}\right|}+\frac{|x|^2}{8\pi}-\frac{1}{4\pi}\log\left(1-x \cdot x^{(k)}  + |x-x^{(k)}|\right)\right),
    \end{split}
\end{equation}
where $(K_\varepsilon^{(k)})_{k=1, \ldots,N}$ are positive real
numbers which converge to~$0$ as~$\varepsilon \to 0$, and which will
be related to the sizes of the $N$ absorbing parts
$(\Gamma^\varepsilon_{\mathbf{D}_k})_{k=1\ldots N}$ below. Indeed, it can readily be checked that $f_k$ satisfies (see~\cite[Lemma 2.1 and Appendix~A]{CWS})
 $$\left\{\begin{aligned}
     \Delta f_k&=-1, \\
     \partial_n f_k&=-|\Omega|\delta_{x^{(k)}}.
 \end{aligned}
 \right.$$

Following the formal reasoning presented in Section~\ref{sec:intro}, let us then introduce the modified domain
$$\widetilde{\Omega}_\varepsilon=\Omega \cap (\varphi^{\varepsilon})^{-1}(-\infty,0].$$
As in the two dimensional case, in the limit $\varepsilon \to 0$, the function $\varphi^\varepsilon$ vanishes in $\Omega$ on $N$ disjoint connected sets $(\widetilde\Gamma^{\varepsilon}_{\mathbf{D}_k})_{k=1 \ldots,N}$. For each $k$, using the fact that when $x$ is close to $x^{(k)}$, 
\[
\varphi^\varepsilon(x) \approx -\frac{1}{\sqrt{|\Omega|}} \left(1 - K_\varepsilon^{(k)}\frac{|\Omega|}{2\pi|x-x^{(k)}|} \right),
\]
one can check that $\widetilde\Gamma^{\varepsilon}_{\mathbf{D}_k}$ is  contained in a spherical shell centered at $x^{(k)}$ and with internal and external radii of the order of $$r^{(k)}_\varepsilon=\frac{2 K_\varepsilon^{(k)}}{3}.$$
Notice that this is very different from the scaling in the two dimensional case, see Lemma~\ref{levelset}.

Working in the modified domain $\widetilde{\Omega}_\varepsilon$ and following the same reasoning as in the two dimensional case, it should be possible to obtain the asymptotic behavior of the first eigenpair~$(\widetilde\lambda_0^\varepsilon,\widetilde u_0^\varepsilon)$, since $\varphi^\varepsilon$ is an excellent approximation of $\widetilde u_0^\varepsilon$. 
More precisely, we believe that the following can be proven. Concerning the eigenvalue, one expects, as in the two dimensional case, 
\[
\widetilde\lambda_0^\varepsilon\approx - \frac{\Delta \varphi^\varepsilon}{\varphi^\varepsilon} \approx \overline{K}_\varepsilon,
\]
where~$\overline{K}_\varepsilon$ is again defined by~\eqref{eq:barK}.
Moreover, 
\[
\int_{\widetilde{\Gamma}^{\varepsilon}_{\mathbf{D}}} \partial_n \widetilde{u}_{0}^{\varepsilon}=\int_{\Omega} \Delta \widetilde{u}_{0}^{\varepsilon} = -\widetilde\lambda_0^\varepsilon \int_{\Omega} \widetilde{u}_{0}^{\varepsilon}\approx  \overline{K}_\varepsilon\sqrt{|\Omega|}.
\]
Likewise, since $\varphi^\varepsilon \approx -|\Omega|^{-1/2} \left(1 + K_\varepsilon^{(k)} f_k\right)$ when $x$ is close to $x^{(k)}$, it is expected that 
\[
\int_{\widetilde{\Gamma}^{\varepsilon}_{\mathbf{D}_k}} \partial_n \widetilde{u}_{0}^{\varepsilon}\approx
\int_{\widetilde{\Gamma}^{\varepsilon}_{\mathbf{D}_k}} \partial_n \varphi^\varepsilon\approx-\frac{K^{(k)}_\varepsilon}{\sqrt{|\Omega|}} \int_{\widetilde{\Gamma}^{\varepsilon}_{\mathbf{D}_k}} \partial_n f_k \approx K^{(k)}_\varepsilon\sqrt{|\Omega|}.
\]
As a consequence, in the limit $\varepsilon\rightarrow0$, for all~$k\in\{1,\dots,N\}$, one expects
\begin{align*}
\mathbb{P}_{\widetilde{\nu}_{0}^{\varepsilon}}\left[X_{\widetilde{\tau}_\varepsilon}\in\Gamma^{\varepsilon}_{\mathbf{D} _k}\right] \approx \frac{K^{(k)}_\varepsilon}{\overline{K}_\varepsilon},
\end{align*}
as in the two dimensional case. As already explained, our objective in this appendix is not to provide rigorous proofs, but simply to illustrate that the method that we introduced in this work should be useful to study entropic metastability in rather general settings. We indeed intend to extend the mathematical results presented above in the simple setting of the two dimensional disk to more general geometries in future works.

\subsection*{Acknowledgments} 
We thank Doug Arnold (University of Minnesota), Alexandre Ern (Ecole des Ponts), Jean-Luc Guermond (Texas A\&M), and Martin Licht (EPFL) for discussions on the appropriate discretization of the eigenvalue problems used in Appendix~\ref{Laplacian p forms}. We also would like to thank Jean-François Bony (Université de Bordeaux), Martin Costabel (Université de Rennes), Monique Dauge (Université de Rennes), Dorian Le Peutrec (Nantes Université), Laurent Michel (Université de Bordeaux) and Boris Nectoux (Université Clermont Auvergne) for discussions on the mathematical analysis. This work was funded by the Agence Nationale de la Recherche, under grant ANR-19-CE40-0010-01 (QuAMProcs), and by the European Research Council (ERC) under the European Union's Horizon 2020 research and innovation programme (project EMC2, grant agreement No 810367).

\end{document}